\documentclass[twoside]{article}

\usepackage[T1]{fontenc}
\usepackage[latin1]{inputenc}
\usepackage[english]{babel}
\usepackage[babel]{csquotes}

\usepackage{cite}
\usepackage{amssymb}
\usepackage{amsmath}
\usepackage{amsthm}
\usepackage{latexsym}
\usepackage{graphicx}
\usepackage{mathrsfs}
\usepackage{mathrsfs}
\usepackage[hidelinks,draft=false]{hyperref}

\usepackage{bbm}
\usepackage[a4paper,top=2cm,bottom=2cm,left=2cm,right=2cm,bindingoffset=5mm]{geometry}

\DeclareFontFamily{U}{BOONDOX-calo}{\skewchar\font=45 }
\DeclareFontShape{U}{BOONDOX-calo}{m}{n}{
  <-> s*[1.05] BOONDOX-r-calo}{}
\DeclareFontShape{U}{BOONDOX-calo}{b}{n}{
  <-> s*[1.05] BOONDOX-b-calo}{}
\DeclareMathAlphabet{\mathcalboondox}{U}{BOONDOX-calo}{m}{n}
\SetMathAlphabet{\mathcalboondox}{bold}{U}{BOONDOX-calo}{b}{n}
\DeclareMathAlphabet{\mathbcalboondox}{U}{BOONDOX-calo}{b}{n}

\usepackage[usenames,dvipsnames]{color}

\theoremstyle{plain}
\newtheorem{thm}{Theorem}[section]

\newtheorem{lem}[thm]{Lemma}
\newtheorem{prop}[thm]{Proposition}
\newtheorem{defin}[thm]{Definition}
\theoremstyle{definition}
\newtheorem{rmk}[thm]{Remark}

\newcommand{\fhi}{\varphi}

\def\enne{\mathbb{N}}
\def\cN{\mathcal{N}}

\def\erre{\mathbb{R}}

\def\P{\mathbb{P}}

\def\calP{\mathcal{P}}
\def\E{\mathop{{}\mathbb{E}}}
\def\cL{\mathscr{L}}
\def\cF{\mathscr{F}}

\def\eps{\varepsilon}

\def\beq{\begin{equation}}
\def\eeq{\end{equation}}

\renewcommand{\d}{{\mathrm d}}

\def\to{\rightarrow}
\def\wto{\rightharpoonup}
\def\wstarto{\stackrel{*}{\rightharpoonup}}
\def\embed{\hookrightarrow}
\def\cembed{\stackrel{c}{\hookrightarrow}}

\def\norm #1{\left\|#1\right\|}

\def\sp #1#2{\left<#1,#2\right>}
\newcommand\ip\sp

\newcommand\testopari{\sc C.~Orrieri, E.~Rocca, L.~Scarpa}
\newcommand\testodispari{\sc Stochastic tumor growth models}
\title{\huge\rm Optimal control of stochastic phase-field models
			related to tumor growth
			\footnote{{\bf Acknowledgments.} This research was supported by 
			the Italian Ministry of Education, University and Research (MIUR): 
			Dipartimenti di Eccellenza Program (2018--2022) -- Dept. of Mathematics 
			``F. Casorati'', University of Pavia. 
			In addition, it has been performed in the framework of the project 
			Fondazione Cariplo-Regione Lombardia  MEGAsTAR
			``Matema\-tica d'Eccellenza in biologia ed ingegneria come acceleratore
			di una nuova strateGia per l'ATtRattivit\`a dell'ateneo pavese''.
			The present paper also benefits from the support of the GNAMPA 
			(Gruppo Nazionale per l'Analisi Matematica, la Probabilit\`a e le loro Applicazioni)
			of INdAM (Istituto Nazionale di Alta Matematica)
			through the project ``Trasporto ottimo per dinamiche con interazione''.
			LS is also funded by Vienna Science and Technology Fund (WWTF) 
			through Project MA14-009.}\\[1cm]}

\author{	{\large\sc Carlo Orrieri}$^{(1)}$\\
		{\normalsize E-mail: \texttt{carlo.orrieri@unitn.it}}\\[.4cm]
		{\large\sc Elisabetta Rocca}$^{(2)}$\\
		{\normalsize E-mail: \texttt{elisabetta.rocca@unipv.it}}\\[.4cm]
		{\large\sc Luca Scarpa}$^{(3)}$\\
		{\normalsize E-mail: \texttt{luca.scarpa@univie.ac.at}}\\[.75cm]
		$^{(1)}$
		{\small Department of Mathematics, University of Trento}\\
		{\small Via Sommarive 14, 38123 Povo (Trento), Italy}\\[.4cm]		
		$^{(2)}$
		{\small Department of Mathematics, University of Pavia, and IMATI - C.N.R.}\\
		{\small Via Ferrata 5, 27100 Pavia, Italy}\\[.4cm]
		$^{(3)}$
		{\small Faculty of Mathematics, University of Vienna}\\
		{\small Oskar-Morgenstern-Platz 1, 1090 Vienna, Austria}
	}

\begin{document}

\maketitle

\begin{abstract}
We study a stochastic phase-field model for tumor growth dynamics coupling a stochastic Cahn-Hilliard equation for the tumor phase parameter
with a stochastic reaction-diffusion equation governing the nutrient proportion. 
We prove strong well-posedness of the system in a general framework through monotonicity and stochastic compactness arguments.
We introduce then suitable controls representing the concentration of cytotoxic drugs administered in medical treatment and we analyze  a related optimal control problem.  
We derive existence of an optimal strategy and deduce first-order 
necessary optimality conditions by studying the corresponding linearized 
system and the backward adjoint system.
  \\[.4cm]
  {\bf AMS Subject Classification:} 35R60, 35K55, 49J20, 78A70.\\[.2cm]
  {\bf Key words and phrases:} stochastic systems of partial differential equations, Cahn-Hilliard equation, optimal control, first-order necessary conditions, tumor growth.
\end{abstract}

\tableofcontents


\thispagestyle{empty}

\section{Introduction}
\setcounter{equation}{0}
\label{sec:intro}

In the recent years phase-field systems have been used to describe many complex systems, 
in particular related to biomedical applications and specially to tumor growth dynamics. 
In this paper we consider a version of a phase-field model for tumor growth recently introduced in  \cite{GLSS16}, where we have neglected the effects 
of chemotaxis and active transport. The new feature of the present work consists in adding stochastic 
terms in both the PDEs ruling the tumor-dynamic and then studying a related optimal control problem. 

This model describes the evolution of a tumor mass surrounded by healthy tissues 
by taking into account proliferation of cells 
via nutrient consumption and apoptosis. In particular, the model under consideration 
in this paper fits into the framework
of {\sl diffuse interface}\/ models for tumor growth. In this setting the evolution of
the tumor is described by an order parameter $\varphi$ which represents the 
local concentration of tumor cells; the interface between the tumor and healthy cells
is supposed to be represented by a narrow transition layer separating
the pure regions where $\varphi=\pm1$, with $\varphi=1$ denoting the tumor phase 
and $\varphi= -1$ the healthy phase. 

We consider here the case of an incipient tumor, i.e., 
before the development of quiescent cells, when the equation ruling the evolution  of  the tumor
growth process is often given by a Cahn--Hilliard (CH) equation \cite{CH} for $\varphi$
coupled with a reaction-diffusion equation for the nutrient $\sigma$ 
(cf., e.g., \cite{CLLW09, GLSS16, HZO, HKNZ15}). We just mention here that 
more sophisticated models have been also developed, possibly including different tumor phases
 (e.g., proliferating and necrotic), or incorporating the effects of fluid flow in the system evolution. 
In this direction, multiphase Cahn-Hilliard-Darcy systems
\cite{Ciarletta,CL10,FLRS, GLSS16,GLNS, WLFC08} have been analyzed in the deterministic case. 
Further studies on tumor growth modelling are presented in \cite{mir-pert, pou-trel} with a particular emphasis to emergence of resistance to therapy.

All the references listed above deal with deterministic representations 
of the tumor growth.
However, it is widely accepted that tumor-dynamics
can be regarded as a random evolution due to stochastic proliferation 
and differentiation of cells (cf.~\cite{TanChen98}). 
Tumor metastases, for example, are generally activated randomly by
biological signals originating at the cellular level and these
processes influence the critical proliferation rates that directly govern the evolution of tumor cells.
In \cite{CM09, NDetAL05} the authors deal with stochastic angiogenesis models with the aim of generating more realistic structures
of capillary networks. 
Other interesting contributions in the bio-medical literature (cf., e.g., \cite{LAO14}) are concerned with stochastic avascular models, taking into account the uncertainty of
the (most influential) parameters. 
Finally, many studies have used stochastic perturbations in an attempt to model the effects of treatments on tumor growth.
In this way, additive noises could represents environmental disruptions caused by therapeutic procedures in cancer treatment: indeed, these are delivered to the entire 
tumor tissue, and most likely to the healthy surrounding tissues as well. 
Although this appears far too simplified to represent real clinical procedures, it can be of acceptable accuracy when applied to the evolution of tumors at early stages.
We refer to \cite{MM17} for the modelling of treatment in the logistic tumor growth dynamic by means of a suitable (white) noise.

General PDE-models of tumor growth that
account for randomness are rare in literature. 
For what concerns diffuse-interface descriptions, up to our knowledge, 
in this paper we give a first contribution in this direction.
The stochastic perturbation that we are considering has a twofold motivation.\\
On one hand, we adopt a statistical approach to deal with the (too much) complicated free energy describing the bio-medical properties of the system.
To do it, we directly add to the CH-equation a additive noise taking into account all the microscopical fluctuation affecting the evolution of the phase parameter.
A natural choice for such a perturbation would be a space-time white noise. However, from the mathematical point of view, space-time whites noises are
difficult to handle in higher space-dimension, and one usually considers
smoothed-in-space noises by introducing a suitable covariance operator.\\
On the other hand, we introduce a multiplicative noise in the reaction diffusion equation (which actually depends on the nutrient concentration) with the aim of modelling the effects of angiogenensis.  
A stochastic forcing of this type is indeed related to the oxygen received by cancerous cells:
this may result in enhancing its effectiveness, and therefore its contribution, 
to the total growth process of the tumor.

More precisely, including the two stochastic terms into the deterministic system we end up with  
the following stochastic
Cahn-Hilliard-reaction-diffusion model for tumor growth:
\begin{align}
  \label{eq:1}
  \d \varphi - \Delta\mu\,\d t = (\mathcal P\sigma- a - \alpha u)h(\varphi)\,\d t + G\,\d W_1 
  \qquad&\text{in } (0,T)\times D\,,\\
  \label{eq:2}
  \mu = -A\Delta\varphi + B\psi'(\varphi) \qquad&\text{in } (0,T)\times D\,,\\
  \label{eq:3}
  \d\sigma - \Delta\sigma\,\d t + c\sigma h(\varphi)\,\d t + b (\sigma-w)\,\d t 
  = \mathcal H(\sigma)\,\d W_2\qquad&\text{in } (0,T)\times D\,,\\
  \label{eq:bound}
  \partial_{\bf n}\varphi = \partial_{\bf n}\mu = \partial_{\bf n}\sigma = 0 \qquad&\text{in } (0,T)\times \partial D\,,\\
  \label{eq:init}
  \varphi(0) = \varphi_0\,, \quad \sigma(0)=\sigma_0 \qquad&\text{in } D\,.
\end{align}
Here, $D\subset\erre^3$ is a smooth bounded domain with smooth boundary, 
$T>0$ is a fixed final time,
$W_1$, $W_2$ are independent cylindrical Wiener processes on separable 
Hilbert spaces $U_1$ and $U_2$, respectively, 
defined on a stochastic basis $(\Omega, \cF, (\cF_t)_{t\in[0,T]}, \P)$,
$G$ is a stochastically integrable operator 
with respect to $W_1$ and ${\cal H}$ is a suitable 
Lipschitz-type operator.
Notice that the initial configuration of the system can be chosen random as well.

The parameters
$\mathcal P, a, \alpha, A, B, c, b$ are assumed to be strictly positive constants.
Namely, $\mathcal P$ denotes the tumor proliferation rate, $a$ the apoptosis rate, 
$\alpha$ the effectiveness rate of the cytotoxic drugs, 
$c$ the nutrient consumption rate, and 
$b$ the nutrient supply rate, while 
$A$ and $B$ are related to the tickness of the interface between the pure phases. 
The function $h$ is assumed to be monotone increasing, nonnegative in
the ``physical'' interval $[-1,1]$, and normalized so that $h(-1) = 0$ and 
$h(1)=1$. 
The term $\mathcal P \sigma h(\fhi)$ models the proliferation of tumor cells, 
which is proportional to the concentration of the nutrient, 
the term $a h(\fhi)$ describes the apoptosis (or death)
of tumor cells, and $c \sigma h(\fhi)$ represents 
the consumption of the nutrient by the tumor cells, which is higher if more tumor cells are present.
The control variables are $u$ in \eqref{eq:1} and $w$ in \eqref{eq:3}, which can be interpreted as 
a therapy (chemotherapy and antiangiogenic therapy, respectively, for example) distribution entering the system, 
either via the mass balance equation or the nutrient 
(cf. also \cite{garcke-lam-roc}, \cite{CGRS}, and \cite{prostate}) for similar choices for the control variables). 
Finally, 
$\psi'$~stands for the derivative of a double-well potential~$\psi$.
 A typical example of potential, meaningful in view of applications, has the expression
\begin{equation}\label{double}
  \psi_{pol}(r) = \frac 14 (r^2-1)^2\,,
  \quad r \in \erre\,.
\end{equation}
We may observe 
that in the first part of our analysis (related to well-posedness of the system) we can allow for more general 
regular potentials having at least cubic and at most exponential growth at infinity. 

The mathematical literature on the stochastic 
Cahn-Hilliard and Allen-Cahn equations is quite developed: we refer for example to 
\cite{bauz-bon-leb,orr-scar} for the stochastic Allen-Cahn equation,
and to \cite{Cornalba, DPD,deb-goud, EM,goud, scar-SCH, scar-SVCH}) for the stochastic 
Cahn-Hilliard equation. For completeness, let us quote also \cite{feir-petc, feir-petc2}
for a study on a stochastic diffuse interface model involving
the Cahn-Hilliard and Navier-Stokes equations.
Nevertheless, 
similar results for coupled stochastic Cahn-Hilliard reaction-diffusion systems 
were not previously studied, up to our knowledge:
in this sense, this contribution can be seen as a first work in this direction. 

After proving the well-posedness of the SPDE-system above, 
we are interested here in the study of the following optimal control problem:
\begin{description}
    \item{\bf(CP)} \textit{Minimize the cost functional}
\begin{align*}
  J(\varphi,u,w)&:=\frac{\beta_1}{2}\E\int_Q|\varphi-\varphi_Q|^2
  +\frac{\beta_2}{2}\E\int_D|\varphi(T)-\varphi_T|^2
  +\frac{\beta_3}{2}\E\int_D(\varphi(T) + 1)\\
  &\quad+\frac{\beta_4}{2}\E\int_Q|u|^2
  +\frac{\beta_5}{2}\E\int_Q|w|^2\,,
\end{align*}
\textit{subject to the control constraint} $(u,w)\in \mathcal U$, \textit{where}
\begin{align*}
  \mathcal U:=&\Big\{(u,w)\in
  L^2(\Omega; L^2(0,T; H))^2\text{ progr.~measurable:}
  \quad 0\leq u,w \leq 1 \text{ a.e.~in } \Omega\times(0,T)\times D\Big\}\,,
\end{align*}
\textit{and to the state system \eqref{eq:1}--\eqref{eq:init}.}
\end{description}
Let us comment on the optimal control problem \textbf{(CP)}. 
The function $\varphi_{Q}$  indicate some desired evolution for the tumor cells
and $\fhi_{T}$ stands for a desired final distribution of tumor cells
(for example suitable for surgery).
The first two terms of ${J}$ are of standard tracking type, 
while the third term of ${J}$ measures the size of the tumor at the end of the treatment. 
The fourth and fifth terms penalize large concentrations of the cytotoxic drugs
through integral over the full space-time domain of the squared 
nutrient and drug concentrations.
As it is presented in ${J}$, a large value of $|\fhi - \fhi_T|^{2}$ 
would mean that the patient suffers from the growth
of the tumor during the treatment, 
and a large value of $|u|^2$ (or $|w|^2$) would mean that the patient 
suffers from high toxicity of the drugs. 
The nonnegative coefficients $\beta_i$ indicate the importance 
of conflicting targets given in the
strategy to avoid unnecessary harm to the patient, 
and at the same time increment the quality of the approximation of 
$\varphi_Q$, $\varphi_T$.
By the optimal control problem \textbf{(CP)},
we aim at searching for a medical strategy $(u,w)$
such that the corresponding tumor concentration
is as close as possible to the targets $\varphi_Q$ and $\varphi_T$,
the tumoral size at the end of the treatment is minimal,
and the total amount of nutrient or drug supplied 
(which is restricted by the control constraints)
does not inflict any harm to the patient.

We shall prove, under suitable regularity on the data, the existence
of a \textit{relaxed} optimal control and derive first-order necessary optimality 
conditions by exploiting a stochastic counterpart of the maximum principle \`a la Pontryagin.
To do so, we introduce a backward (stochastic) system and we formulate necessary conditions for optimality through a variational inequality.  
The optimal control problem is studied in the particular case of the 
classical double well potential \eqref{double} and in case ${\cal H}\equiv0$, 
meaning that we are neglecting the stochastic perturbation entering
in the nutrient equation and that the only source of randomness
is contatined in the CH-equation: see Remark~\ref{Hzero} for further comments on this topic. 

In the recent mathematical literature a number of results related 
to optimal control for deterministic tumor growth models has appeared. 
We can quote \cite{garcke-lam-roc, CGRS, CRW, CGMR, Signori} 
for models coupling different variants of Cahn-Hilliard equations with 
reaction-diffusion equations and \cite{EK1, EK2, SW} 
for the case when also the velocity dynamic has been taken into account. 
Moreover, an optimal control problem taking into account also damages to healthy tissues was studied in \cite{pou-trel} in the context of integro-differential modelling.
For what concerns optimal control problems and  necessary condition for optimality in the stochastic setting we can refer to the monography \cite{YZ99} for a general overview. The infinite dimensional case is tackled e.g. in  
\cite{fuhr-hu-tess, fuhr-orr, mas-orr}, dealing with stochastic heat equation and reaction-diffusion systems. 
Optimal control for stochastic 
Cahn-Hilliard equations is discussed in  the recent paper 
\cite{Scarpa} where the control variable is contained in the chemical potential equation. 

Let us comment on the mathematical difficulties of this work.
First of all, we stress that the stochastic system \eqref{eq:1}--\eqref{eq:3} does not
fall in any available framework for studying stochastic evolution equations:
this is due to both the growth of $\psi$, which may be of any
polynomial order or first-order exponential, and more importantly
to the coupling terms appearing in \eqref{eq:1} and \eqref{eq:3},
which prevent from having a global monotonicity property 
for the pair $(\varphi,\sigma)$ in a reasonable product space. 
Indeed, \eqref{eq:1}--\eqref{eq:2} have a monotone behaviour in the 
dual space of $H^1(D)$, while \eqref{eq:3} is monotone on $L^2(D)$.
Such issue does not depend on the stochastic setting and represents
a crucial difficulty also in the deterministic case, which also  requires an {\em ad-hoc} analysis of the system.
Let us explain now the main differences when dealing 
with the stochastic case in comparison with the deterministic case.

Let us start by mentioning that the standard compactness techniques used in the deterministic setting cannot be directly applied. 
This mainly results from the lack of compactness of the embedding from $L^r(\Omega; \mathcal{X})$ into $L^r(\Omega, \mathcal{Y})$, $1\le r \le +\infty$, even if $\mathcal{X} \hookrightarrow \mathcal{Y}$ in a compact fashion. 
To bypass the problem, we rely on the combination of the Skorohod representation theorem with the well-known Gy\"ongy-Krylov's criterium \cite{gyongy-krylov}.
The crucial ingredient for the all machinery to work is a pathwise uniqueness result for the system, which is not a straightforward consequence of its deterministic counterpart. 

The main issue when dealing with the well-posedness of the system
\eqref{eq:1}--\eqref{eq:init} is the presence of proliferation terms 
in the Cahn-Hilliard equation, resulting in the non-conservation
of the spatial mean of $\varphi$. In the deterministic setting, this
difficulty is usually overcome by estimating directly the mean of $\varphi$
and by keeping track of it in the estimates on the solutions: in particular, 
such procedure hinges 
on the Lipschitz-continuity of $h$ and on proving the boundedness
of $\sigma\in[0,1]$ through a maximum principle for the reaction-diffiusion equation.
By contrast, in the stochastic setting, proving the boundedness of $\sigma$ in $[0,1]$
is much more delicate, due to a lack of a maximum principle in a general framework
for \eqref{eq:3}. In our setting, we are still 
able to show that $\sigma$
remains bounded in $[0,1]$ during the evolution thanks
to suitable assumptions on the 
operator $\mathcal H$: roughly speaking, 
we require that the noise in equation \eqref{eq:3} ``switches off''
whenever $\sigma$ touches the values $0$ or $1$. This 
procedure is very common from the mathematical point of view,
and is also pretty reasonable in terms of applications: indeed, 
as we have pointed out before, the multiplicative noise in \eqref{eq:3}
could model the angiogenesis phenomenon, i.e.~the formation
of new blood vessels to convey more nutrient towards the tumoral cells.
Hence, it is very natural 
to assume that this is neglectable when
$\sigma=1$ (i.e.~when the nutrient level is saturated)
and $\sigma=0$ (i.e.~when in principle no nutrient is available).

In order to keep track 
of the estimate on the spatial average of $\varphi$
one usually estimates $\varphi$ in terms of $\sigma$ from \eqref{eq:1}--\eqref{eq:2},
$\sigma$ in terms of $\varphi$ from \eqref{eq:3}, and concludes by 
``closing the estimate'' using the boundedness of $\sigma$ and  
a Gronwall-type argument. While this procedure perfectly works 
in the deterministic setting, there are several points deserving attention
in the stochastic case. The main problem is that, due to the presence
of the noise terms, one is forced to use estimates in expectation
through some maximal martingale inequalities, and cannot rely
on any pathwise estimate (i.e.~with $\omega\in\Omega$ being fixed).
Consequently, when trying to close the estimate through the Gronwall lemma,
some further regularity is needed on the solutions in terms of
existence of exponential moments.

A related issue arising in the stochastic setting concerns 
the continuous dependence of the solution $(\varphi,\sigma)$
on the data of the problem, in particular on the controls $(u,w)$.
Indeed, as it is common when dealing with 
optimal control problems, we need to deduce some
continuous dependence properties with respect to the controls
in stronger topologies in order to tackle the linearized system.
As we have anticipated, this hinges again on the combination
of estimates in expectations and the Gronwall lemma, resulting
in a need for boundedness of exponential moments of the solutions.
The continuous dependence property is achieved with much more
difficulty with respect to the deterministic case, as
a careful track of the specific moments of the solutions is necessary.

Let us now turn to the major issues arising in the stochastic optimal control problem.
In the deterministic setting, the first step consists in showing 
the Fr\'echet differentiability of the control-to-state map
in order to analyse the linearized system:
this is usually achieved (also for degenerate potentials)
by proving that $\psi''(\bar\varphi)\in L^\infty((0,T)\times D)$, 
where $\bar\varphi$ is the state variable 
corresponding to a fixed control 
$(\bar u, \bar w)$. 
Such a regularity can be obtained 
by requiring that the set of admissible controls is
bounded in $L^\infty$ and performing 
a maximum-principle-argument.
Nevertheless, in the stochastic case,
there is no hope to prove an $L^\infty$-bound on $\varphi$,
due to the presence of the additive noise in \eqref{eq:1}.
We overcome this problem as it was done in 
\cite{Scarpa}. First, we prove that the linearized system admits a 
unique variational solution. Then, we show that the control-to-state mapping
is G\^ateaux differentiable only in a weak sense, 
and that its derivative can be identified as the unique solution to the linearized system: 
this is still enough to deduce a first-order variational inequality.

The second main difficulty in the stochastic context
consists in solving the adjoint problem, which is a system of backward stochastic PDEs.
Recall that inverting time in stochastic dynamics is not straightforward and results in the introduction of 
two additional variables (one for $\varphi$ and one for $\sigma$), which guarantees that solutions are adapted to the fixed filtration $(\cF_t)_t$. 
Moreover, in the present case, the classical variational theory for backward SPDEs cannot be applied directly due to the presence of the coupling terms and to the growth of the potential.
To solve the problem, we firstly approximate the backward equations and subsequently exploit the (abstract) duality relation between the adjoint and the linearized dynamics. 
This permit to get uniform estimates on the adjoint variables through a (more easier) continuous dependence of the linearized system.
For an application of this strategy to stochatic reaction-diffusion equations we refer to \cite{fuhr-orr}.

Let us finally mention that this work is meant as a 
first contribution in the direction of analyzing stochastic mathematical models 
of tumor growth and many relevant questions remain open: 
it would be interesting to validate the results with numerical 
simulations and medical data (we refer for example to \cite{GP95} for a possible stochastic model of MRI data).
Moreover, 
it would be interesting to invesitgate uniqueness of optimal control problems and to include the duration of the therapy as another control variable in the formulation of the cost functional
(cf.~\cite{garcke-lam-roc} and \cite{CRW} for similar analysis in the deterministic case).

Here is the plan of the paper. In Section~\ref{sec:main} 
we state the main assumptions and the main results of the paper concerning 
existence and uniqueness of solutions, a refined well-posedness theorem 
and the results on the optimal control problem: 
existence of the optimal controls, G\^ateaux differentiability of the control-to-state map
and first order optimality conditions obtained by solving the adjoint system. 
The proofs of the first well-posedness result is given in Section~\ref{sec:exist} and 
the refined well-posedness is proved in Section~\ref{sec:refined}.
The last Section~\ref{sec:opt} 
is devoted to the optimal control problem and to the analysis of the adjoint system.


\section{General setting and main results}
\setcounter{equation}{0}
\label{sec:main}

Throughout the paper, $D\subset\erre^3$ is a smooth bounded domain
with smooth boundary $\Gamma$: we shall denote the space-time cylinder $(0,T)\times D$ by $Q$
and use the classical notation $Q_t$ for $(0,t)\times D$ with $t\in[0,T)$.

It is useful to introduce the spaces
\begin{equation*}
H:=L^2(D)\,, \qquad V:=H^1(D)\,, \qquad Z:=\left\{u\in H^2(D):\partial_{\bf n} u = 0 \text{ a.e.~on } \Gamma\right\} .
\end{equation*}
For a generic element $y \in V^*$ we denote by $y_D$ the average $y_D:= \frac{1}{D} \langle y,1 \rangle_V$ and we recall the Poincar\'e-Wirtinger inequality: there exists $M_D>0$, only depending on $D$, such that
\begin{equation}\label{Poin-wirt}
\| v- v_D \|_H \leq M_D \| \nabla v\|_H \quad \forall \, v \in V\,.
\end{equation}
Setting
\[V^*_0 := \{ y \in V^*: y_D = 0 \}\,, \qquad  V_0 := V \cap V^*_0\,,\]
we can define the operator $\cN: V^*_0 \to V_0$ as the unique solution of the generalized Neumann problem
\[ \int_D \nabla \cN y \cdot  \nabla \phi  = \langle y, \phi \rangle_V \quad \forall \, \phi \in V\,, \qquad (\cN y)_D = 0\,.\]
This also provides an equivalent norm on $V^*$ given by the following
\[ \| y\|_*^2 := \| \nabla \cN (y - y_D)\|_H^2  + |y_D|^2\,, \quad  y \in V^*\,.\]
In particular,  by the inequality \eqref{Poin-wirt} this guarantees that
there exists $M_D'>0$, only depending on $D$, such that
\begin{equation}\label{ineq:Ny}
\| \cN y\|_V \leq M'_D\| \nabla \cN y\|_H = \|y\|_* \quad\forall\,y\in V_0^*\,.
\end{equation}
We also recall that the following compactness inequality holds:
for every $\eps>0$ there exists $M_\eps>0$ such that 
\beq
  \label{comp_ineq}
  \norm{y}_H^2\leq \eps\norm{\nabla y}_H^2 + M_\eps\norm{y}_{V^*}^2 \quad\forall\,y\in V\,.
\eeq

Furthermore,
$(\Omega, \cF, (\cF_t)_{t\in[0,T]}, \P)$ is a filtered probability space
satisfying the usual conditions, where $T>0$ is a fixed final time, 
and $W_1$, $W_2$ are independent cylindrical Wiener processes on separable 
Hilbert spaces $U_1$ and $U_2$, respectively. 
We shall assume that the filtration $(\cF_t)_t$ is 
the one generated by $(W_1, W_2)$.
Let also $(e_n)_n$
be an arbitrary, but fixed, complete orthonormal system of $U_2$ and denote with $\cL^2(U_i, H)$ the space of Hilbert-Schmidt operators from $U_i$ to $H$, $i = 1,2$. 
In the following we frequently use the shorthand $F \cdot W_i := \int_0^T  F\, \d W_i$, for suitable stochastically integrable operators $F$. 

Throughout the work, we shall use the symbol $a\lesssim b$
for any $a,b\geq0$ to indicate that 
there exists a constant $c>0$, independendent of $a,b$, such that 
$a\leq cb$. When the implicit constant $c$ depends on some
relevant external quantities that we want to keep track of,
we shall write it as a subscript.

Fix now $p\geq 6$. The following assumptions will be in order throughout the paper:
\begin{itemize}
  \item[(A1)] $\mathcal P$, $a$, $\alpha$, $b$, $c$, $A$, $B$ are positive constants, 
  		and $h:\erre\to[0,1]$ is a Lipschitz-continuous function with Lipschitz constant $L_h>0$
		such that $h(-1)=0$ and $h(1)=1$;
  \item[(A2)] $\psi:\erre\to[0,+\infty)$ is of class $C^2$ and satisfies 
       		\begin{align*}
		|\psi''(x)|&\leq C_1(1+|\psi'(x)|)\,,\\
		\psi''(x)&\geq -C_2\,,\\
		|\psi'(x)-\psi'(y)|&\leq C_3\left(1 + |\psi''(x)| + |\psi''(y)|\right)|x-y|
		\end{align*}
		for some positive constants $C_i$, $i=1,2,3$, and for every $x,y\in\erre$;
  \item[(A3)] $G\in L^4(0,T; \cL^2(U_1,V)) \cap 
		L^\infty(0,T; \cL^2(U_1,V^*))$ is progressively measurable;
  \item[(A4)] the sequence $(\mathcalboondox{h}_n)_n\subset W^{1,\infty}(\erre)$ satisfies
  		\[
		L^2_{\mathcal H}:=\sum_{n=0}^\infty\norm{\mathcalboondox{h}_n'}^2_{L^\infty(\erre)}<+\infty\,,\qquad
		\mathcalboondox{h}_n(0)=\mathcalboondox{h}_n(1)=0 \quad\forall\,n\in\enne\,.
		\]
		This implies that
		it is well-defined and $L_{\mathcal H}$-Lipschitz-continuous the operator
		\[
		\mathcal H:H\to\cL^2(U_2, H)\,, \qquad \mathcal H(x): e_n\mapsto \mathcalboondox{h}_n(x) \quad \forall\,n\in\enne\,;
		\]
  \item[(A5)] $u, w$ are $H$-valued progressively measurable processes such that 
  		$u,w\in[0,1]$ almost everywhere in $\Omega\times(0,T)\times D$;
  \item[(A6)] $\varphi_0\in L^p(\Omega,\cF_0,\P; V)$ and $\psi(u_0)\in L^{p/2}(\Omega,\cF_0,\P; L^1(D))$;
  \item[(A7)] $\sigma_0 \in L^2(\Omega,\cF_0,\P; H)$ and $\sigma_0\in[0,1]$ a.e.~in $\Omega\times D$.
\end{itemize}

\begin{rmk}
  Let us comment on the assumptions above. First of all, note that assumption (A2) 
  includes both the classical polynomial double-well potential \eqref{double}
  as well as any double-well potential with at most first-order exponential growth at infinity.
  Secondly, a typical example of operator $G$ satisfying (A3) is given by any fixed
  nonrandom and time-independent $G\in\cL^2(U_1,V)$. Moreover, 
  in assumption (A4) we are requiring that the multiplicative noise $\mathcal H\,dW_2$
  acts along the directions $e_n$ through the Lipschitz-continuous function $\mathcalboondox{h}_n$
  for all $n$. The requirement that $\mathcalboondox{h}_n$ vanishes at $0$ and $1$ heuristically means
  that the noise switches off as soon as $\sigma$ touches the border-values $0$ and $1$.
  Finally, the assumptions (A6)--(A7) on the initial data are trivially satisfied 
  when $\varphi_0\in V$ and $\sigma\in H$ are nonrandom with 
  $\psi(\varphi_0)\in L^1(D)$ and $\sigma\in[0,1]$ almost everywhere in $D$.
\end{rmk}

The first three results concern the well-posedness of the state system.
First of all, we prove existence of solutions and continuous dependence on
the data in the most general framework (A1)--(A7) in Theorems~\ref{th:1}--\ref{th:2}.
Then, in Theorem~\ref{th:3} we show that under additional assumptions on the data the 
solution to the system inherits further regularity and a stronger dependence on the 
controls can be written: this will be necessary in order to 
tackle the optimal control problem.

\begin{thm}[Existence]
  \label{th:1}
  Under the assumptions {\em(A1)--(A7)} there exists a unique triplet $(\varphi,\mu,\sigma)$ of progressively
  measurable $V$-valued processes with
  \begin{gather}
    \label{phi1}
    \varphi \in L^p\left(\Omega; C^0([0,T]; H)\cap L^\infty(0,T; V)\cap L^2(0,T; Z)\right)\,, \qquad
    \varphi - G\cdot W_1 \in L^p(\Omega; H^1(0,T; V^*))\,,\\
    \label{phi2}
    \psi'(\varphi) \in L^{p/2}(\Omega; L^2(0,T; H))\,,\\
    \label{mu}
    \mu \in L^{p/2}(\Omega; L^2(0,T; V))\,, \qquad \nabla\mu \in L^p(\Omega; L^2(0,T; H))\,, \qquad
    \mu_D \in L^{p/2}(\Omega; L^\infty(0,T))\,,\\
    \label{sigma1}
    \sigma \in L^p\left(\Omega; C^0([0,T]; H) \cap L^2(0,T; V)\right)\,,\qquad
    \sigma - \mathcal H(\sigma)\cdot W_2 \in L^p(\Omega; H^1(0,T; V^*))\,,\\
    \label{sigma2}
    \qquad \sigma\in[0,1] \quad\text{a.e.~in } \Omega\times Q\,,\\
    \label{phi_sigma_0}
    \varphi(0)=\varphi_0\,, \quad \sigma(0)=\sigma_0
    \qquad\text{a.e.~in } \Omega\times D\,,
  \end{gather}
  and satisfying
  \begin{gather}
  \ip{\partial_t(\varphi-G\cdot W_1)(t)}{\zeta}_V 
  + \int_D\nabla\mu(t)\cdot\nabla\zeta =
  \int_D\left(\mathcal P\sigma(t) - a - \alpha u\right)h(\varphi(t))\zeta\,,\label{eq_phi_weak} \\
  \label{eq_mu_weak}
  \int_D\mu(t)\zeta = A\int_D\nabla\varphi(t)\cdot\nabla\zeta + B\int_D\psi'(\varphi(t))\zeta\,,\\
  \ip{\partial_t(\sigma-\mathcal H(\sigma)\cdot W_2)(t)}{\zeta}_V + \int_D\left[\nabla\sigma(t)\cdot\nabla\zeta
  +c\sigma(t)h(\varphi(t))\zeta + b(\sigma(t)-w(t))\zeta\right] = 0 \label{eq_sigma_weak} 
  \end{gather}
  for every $\zeta\in V$, for almost every $t\in(0,T)$, $\P$-almost surely.
\end{thm}

\begin{thm}[Continuous dependence]
  \label{th:2}
  Under assumptions {\em(A1)--(A4)}, let the data
  $(u_1,w_1,\varphi^1_0,\sigma^1_0)$ and $(u_2,w_2, \varphi^2_0,\sigma^2_0)$ 
  satisfy {\em(A5)--(A8)}, and let 
  $(\varphi_1,\mu_1,\sigma_1)$ and $(\varphi_2,\mu_2,\sigma_2)$ be two respective solutions satisfying
  \eqref{phi1}--\eqref{eq_sigma_weak}.
  Then there exists a sequence $(\tau^N)_{N \in \enne}$ of stopping times
  and a sequence $(M_N)_{N\in\enne}$ of positive real numbers,
  depending on $\varphi_1$ and $\varphi_2$, such that  
  $\tau^N\nearrow T$ $\P$-a.s.~as $N\to\infty$, and
  \begin{equation}\label{cont_dep_1}
  \begin{split}
  &\E\sup_{t\in[0,\tau^N]}\left(\|(\varphi_1  - \varphi_2)(t)\|_{V^*}^p + \|(\sigma_1  - \sigma_2)(t)\|_{H}^p\right)
  +\E\left(\int_0^{\tau^N} \| \nabla(\varphi_1-\varphi_2)(s) \|_H^2\,\d s\right)^{p/2}\\
  &\qquad+
  \E\left(\int_0^{\tau^N} \|(\sigma_1-\sigma_2)(s) \|_V^2\,\d s\right)^{p/2}\\
  &\leq M_N\left[\| \varphi_0^1 - \varphi_0^2 \|_{L^{p}(\Omega; V^*)}^p 
  +\| \sigma_0^1-\sigma_0^2 \|_{L^{p}(\Omega;H)}^p
  +\E\left(\int_0^{\tau^N} \|(u_1-u_2)(s) \|_{V^*}^2\,\d s\right)^{p/2}\right.\\
  &\qquad\left.+\E\left(\int_0^{\tau^N} \|(w_1-w_2)(s) \|_{H}^2\,\d s\right)^{p/2}\right]
  \qquad\forall\,N\in\enne\,.
  \end{split}
  \end{equation}
  In particular, the problem \eqref{phi1}--\eqref{eq_sigma_weak} admits a unique solution.
\end{thm}

\begin{thm}[Refined well-posedness]
  \label{th:3}
  Under the assumptions {\em(A1)--(A7)}, suppose also that
  \begin{gather}
  \label{ip_ref1}
  \exp\left(\beta\norm{\varphi_0}_H^2\right)\in L^1(\Omega) \quad\forall\,\beta\geq1\,,\\
  \label{ip_ref2}
  G\in L^\infty(\Omega\times(0,T); \cL^2(U_1,H))\,,\\
  \label{ip_ref3}
  \psi\in C^3(\erre)\,, \qquad \exists\, C_4>0:\quad |\psi'''(r)| \leq C_4(1+|r|) \quad\forall\,r\in\erre\,.
  \end{gather}
  Then, the unique solution $(\varphi,\mu,\sigma)$ to \eqref{phi1}--\eqref{eq_sigma_weak}
  also satisfies
  \begin{gather}
  \label{phi_H3}
  \varphi\in L^{p/3}(\Omega; L^2(0,T; H^3(D)))\,,\\
  \label{phi_reg}
  \exp\left(\beta\norm{\varphi}^2_{L^2(0,T; Z)}\right)
   \in L^1(\Omega) \quad\forall\,\beta\geq1\,.
  \end{gather}
  Furthermore, assume also that
  \beq\label{ip_H0}
  \mathcal H\equiv0\,.
  \eeq
  Let the data $(u_1,w_1,\varphi^1_0,\sigma^1_0)$ and 
  $(u_2,w_2, \varphi^2_0,\sigma^2_0)$ 
  satisfy {\em(A5)--(A7)} and \eqref{ip_ref1}, 
  and let $(\varphi_1,\mu_1,\sigma_1)$ and 
  $(\varphi_2,\mu_2,\sigma_2)$ be the two respective solutions satisfying
  \eqref{phi1}--\eqref{eq_sigma_weak} and \eqref{phi_reg}:
  then for every $q>p$ there 
  exists a positive constant $M_{q,p}=M_{q,p}(\varphi_1,\varphi_2)$ such that 
  \begin{equation}\label{cont_dep_2}
  \begin{split}
  &\| \varphi_1  - \varphi_2 \|_{L^{p}(\Omega; C^0([0,T]; V^*))} 
  + \| \nabla(\varphi_1-\varphi_2) \|_{L^{p}(\Omega; L^2(0,T;H))} 
  + \norm{\sigma_1-\sigma_2}_{L^p\left(\Omega; C^0([0,T]; H)\cap 
  L^2(0,T; V)\right)}\\
  &\qquad\leq M_{q,p}\Bigl(\| \varphi_0^1 - \varphi_0^2 \|_{L^q(\Omega; V^*)} 
  +\| \sigma_0^1-\sigma_0^2 \|_{L^q(\Omega;H)}\Bigr.\\
  &\quad\qquad\Bigl.+ \|u_1-u_2\|_{L^q(\Omega; L^2(0,T;V^*))} 
  +\norm{w_1-w_2}_{L^q(\Omega; L^2(0,T; H))}\Bigr)\,.
  \end{split}
  \end{equation}
  If also
  \beq
  \label{ip_ref4}
  \exists\,r>\frac{2pq}{q-p}:\qquad \varphi_0\in L^r(\Omega; V)\,, \quad G\in L^r(\Omega; L^2(0,T; \cL^2(U_1,V)))\,,
  \eeq
  then there exists a positive constant $M_{p,q,r}=M_{p,q,r}(\varphi_1, \varphi_2)$ such that 
  \begin{equation}\label{cont_dep_3}
  \begin{split}
  &\| \varphi_1  - \varphi_2 \|_{L^{p}(\Omega; C^0([0,T]; H)\cap L^2(0,T; Z))} 
  + \norm{\sigma_1-\sigma_2}_{L^p\left(\Omega; C^0([0,T]; H)\cap 
  L^2(0,T; V)\right)}\\
  &\qquad\leq M_{p,q,r}\Bigl(\| \varphi_0^1 - \varphi_0^2 \|_{L^q(\Omega; H)} 
  +\| \sigma_0^1-\sigma_0^2 \|_{L^q(\Omega;H)}\Bigr.\\
  &\qquad\quad\Bigl.+ \|u_1-u_2\|_{L^q(\Omega; L^2(0,T;H))} 
  +\norm{w_1-w_2}_{L^q(\Omega; L^2(0,T; H))}\Bigr)\,.
  \end{split}
  \end{equation}
\end{thm}

\begin{rmk}\label{Hzero}
  Note that the refined assumptions \eqref{ip_ref1}--\eqref{ip_ref3}  and \eqref{ip_ref4}
  are not much restrictive: indeed, they are still satisfied
  by the wide class of data $\varphi_0\in V$ nonrandom, $G\in\cL^2(U_1,V)$ fixed
  (nonrandom and time-independent) and by the classical double-well potential \eqref{double} of degree $4$.
  The reason why we need to assume \eqref{ip_H0} is that,
  in order to obtain the continuous dependence property on the whole time interval $[0,T]$,
  one needs to rely necessarily on pathwise estimates
  (i.e.~with $\omega\in\Omega$ fixed) 
  of $\sigma_1-\sigma_2$ in terms of
  $\varphi_1-\varphi_2$: this is possible only if the 
  noise in equation \eqref{eq:3} is of additive type. However, since it is also crucial 
  to prove that $\sigma\in[0,1]$ (hence also (A4) must be in order), the 
  only reasonable choice is to require that $\mathcal H\equiv0$ and to neglect 
  the noise in the equation for $\sigma$.
\end{rmk}

We are now ready to introduce the general setting 
for the optimal control problem of the state system \eqref{eq:1}--\eqref{eq:init}.
From now on, we shall assume (A1)--(A7)
\eqref{ip_ref1}--\eqref{ip_ref3}, \eqref{ip_H0} and \eqref{ip_ref4}, so that 
the strongest continuous dependence property \eqref{cont_dep_3} holds,
and we shall also fix an exponent $q>p$.

We introduce the set of admissible controls $\mathcal U$ as
\begin{align*}
  \mathcal U:=&\Big\{(u,w)\in
  L^2(\Omega; L^2(0,T; H))^2\text{ progr.~measurable:}
  \quad 0\leq u,w \leq 1 \text{ a.e.~in } \Omega\times(0,T)\times D\Big\}\,,
\end{align*}
which is clearly a closed convex bounded subset of $L^\infty(\Omega\times Q)^2$. 
It will be useful to embed $\mathcal U$ in an open bounded subset $\tilde{\mathcal U}  \subset L^q(\Omega; L^2(0,T; H))^2 \cap L^\infty(\Omega \times Q)^2$.

We define the cost functional $J$ as
\begin{align*}
  J&: L^2(\Omega; C^0([0,T]; H))\times L^2(\Omega; L^2(0,T; H))\times L^2(\Omega; L^2(0,T; H))
  \longrightarrow[0,+\infty)\,,\\
  J(\varphi,u,w)&:=\frac{\beta_1}{2}\E\int_Q|\varphi-\varphi_Q|^2
  +\frac{\beta_2}{2}\E\int_D|\varphi(T)-\varphi_T|^2
  +\frac{\beta_3}{2}\E\int_D(\varphi(T) + 1)\\
  &\quad+\frac{\beta_4}{2}\E\int_Q|u|^2
  +\frac{\beta_5}{2}\E\int_Q|w|^2\,,
\end{align*}
where $\beta_i\geq0$, $i=1,\ldots,5$, are fixed constants and we assume that
\[
   \beta_1\varphi_Q\in L^6(\Omega; L^2(0,T; H))\,,
   \qquad \beta_2\varphi_T\in L^6(\Omega,\cF_T; V)\,.
\]
Moreover, Theorems~\ref{th:1}--\ref{th:3} ensure that the control-to-state map
\[
  S:\tilde{\mathcal U}\to L^p(\Omega; C^0([0,T]; H)\cap L^\infty(0,T; V)\cap L^2(0,T; Z))\,,\qquad
  (u,w)\mapsto \varphi
\]
is well-defined and furthermore it is Lipschitz-continuous in the sense specified in \eqref{cont_dep_3}: 
the proofs of Theorems~\ref{th:1}--\ref{th:3} can be adapted 
with no difficulty to the more general case where $u,w\in L^\infty(\Omega\times Q)$ (not necessarily $0 \leq u,w \leq 1$). 
Consequently, it is natural to define the reduced cost functional
\[
  \tilde J: \tilde{\mathcal U} \to [0,+\infty)\,, \qquad
  \tilde J(u,w):=J(S(u,w), u,w)\,, \quad (u,w)\in\tilde{\mathcal U}\,.
\]
The optimal control problem associated to the state system \eqref{eq:1}--\eqref{eq:init}
consists in minimizing the map $\tilde J$ on $\mathcal U$.
Note that even if $J$ is convex, the nonlinearity of $S$ does not ensure that 
$\tilde J$ is convex as well, so that the minimization problem in nontrivial
and uniqueness of optimal controls
is difficult to prove. 
Furthermore, as already mentioned in the introduction, classical  compactness results cannot be directly applied to the stochastic setting and this gives rise 
to the need of a weaker concept of optimality.

We introduce below the notion of \textit{relaxed} optimal control, which is inspired by \cite[Def.~2.4]{barbu-rock-contr}.
\begin{defin}
  An optimal control is a pair $(u,w)\in\mathcal U$ such that 
  \[
  \tilde J(u,w)=\inf_{(v,z)\in\mathcal U}\tilde J(v,z)\,.
  \]
  A relaxed optimal control is a family 
  \[
  ((\Omega',\cF', (\cF'_t)_{t\in[0,T]}, \P'), W_1', \varphi_0', \sigma_0',u',w',
  \varphi', \mu', \sigma', 
  \varphi_Q', \varphi_T')\,,
  \]
  where $(\Omega',\cF',(\cF'_t)_{t\in[0,T]}, \P')$ is a complete filtered probability space,
  $W_1'$ is a $(\cF_t)_t$-cylindrical Wiener process, $(\varphi_0', \sigma'_0)$ is 
  a $\cF_0'$-measurable $(V\times H)$-valued random variable with the same law of 
  $(\varphi_0, \sigma_0)$, $(u',w')$ is a $(\cF'_t)_t$-progressively measurable $H^2$-valued
  process such that $0\leq u',w'\leq1$ almost everywhere in $\Omega'\times Q$,
  $(\varphi', \mu', \sigma')$ is the unique solution to the state system 
  \eqref{phi1}--\eqref{eq_sigma_weak}
  on the probability space $(\Omega'\cF',\P')$ with respect to the data 
  $(\varphi_0', \sigma_0', u', w', W_1')$ (and the choice $\mathcal H\equiv0$), 
  $\beta_1\varphi_Q' \in L^2(\Omega'; L^2(0,T; H))$ has the same law of $\beta_1\varphi_Q$,
  $\beta_2\varphi_T'\in L^6(\Omega',\cF_T'; V)$ has the same law of $\beta_2\varphi_T$, and 
  \begin{align*}
  \tilde J'(u',w')&:=\frac{\beta_1}{2}\E{}'\int_Q|\varphi'-\varphi_Q'|^2
  +\frac{\beta_2}{2}\E{}'\int_D|\varphi'(T)-\varphi_T'|^2
  +\frac{\beta_3}{2}\E{}'\int_D(\varphi'(T) + 1)\\
  &+\frac{\beta_4}{2}\E{}'\int_Q|u'|^2
  +\frac{\beta_5}{2}\E{}'\int_Q|w'|^2
  \leq \tilde J(v,z) \quad\forall\,(v,z)\in\mathcal U\,.
  \end{align*}
\end{defin}

Let us briefly comment on the notion of relaxed optimality.
First notice that the idea of relaxation mimics the concept of weak (or martingale) solutions to stochastic equations, thus involving in the definition the probability space itself along with the Wiener process and the parameters of the model.
This naturally allows for the application of the Skorohod representation theorem.
Moreover, let us observe that by setting $\beta_4 = \beta_5 = 0$ in the definition of the cost functional, the problem reduces to the existence of nearest points of a fixed target in uniformly convex Banach spaces (see \cite{barbu-rock-contr} for further details). 
Finally, given a relaxed optimal control it seems natural to wonder whether it admits or not a \textit{strong} formulation. 
A possible way to tackle the problem could be the combination of a Gy\"ongy-Krylov argument with a uniqueness result for optimal controls.

The first result that we prove in this context is very natural and ensures
that a relaxed optimal control for our problem always exists.
\begin{thm}
  \label{th:4}
  Assume {\em(A1)--(A7)}, \eqref{ip_ref1}--\eqref{ip_ref3}, \eqref{ip_H0} and \eqref{ip_ref4}. Then
  there exists a relaxed optimal control for problem \eqref{eq:1}--\eqref{eq:init}.
\end{thm}

Let us now concentrate on necessary conditions for optimality.
To this end, we first introduce the following linearized system
\begin{equation*}
\begin{split}
\partial_t x_k - \Delta y_k  &= h(\varphi)(\mathcal{P}z_k - \alpha k_u) + h'(\varphi)x_k(\mathcal P\sigma - a -\alpha u) \\
y_k &= A \Delta x_k + B \psi''(\varphi)x_k \\
\partial_t z_k - \Delta z_k &+ cz_kh(\varphi) + c\sigma h'(\varphi) x_k + b(z_k - k_w) = 0,
\end{split}
\end{equation*}
complemented with homogeneous Neumann boundary conditions for $x_k,y_k$ and $z_k$ with initial conditions given by $x_k(0) = z_k(0) = 0$.
Existence and uniqueness of variational solutions to the above system
is the content of the following theorem.
\begin{thm}
\label{th:5}
  Assume {\em(A1)--(A7)}, \eqref{ip_ref1}--\eqref{ip_ref3}, \eqref{ip_H0} and \eqref{ip_ref4}.
  Let $(u,w)\in\tilde{\mathcal U}$, $k:=(k_u,k_w)\in\tilde{\mathcal U}$ and set $\varphi:= S(u,w)$. Then
  there exists a unique triple $(x_{k},y_{k},z_{k})$ with
  \begin{gather}
    \label{xy}
    x_{k} \in L^p(\Omega; H^1(0,T; Z^*)\cap L^2(0,T; Z))\,, \qquad y_{k}\in L^p(\Omega; L^2(0,T; H))\,,\\
    \label{z}
    z_{k}\in L^p(\Omega; H^1(0,T; V^*)\cap L^2(0,T; V))\,,
  \end{gather}
  such that $x_{k}(0)=z_{k}(0)=0$ and
  \begin{gather}
  \label{lin1}
  \ip{\partial_t x_{k}}{\zeta}_V -\int_Dy_{k}\Delta\zeta = 
  \int_D\left[h(\varphi)(\mathcal Pz_k - \alpha k_u) + h'(\varphi)x_k(\mathcal P\sigma - a -\alpha u)\right]\zeta\,,\\
  \label{lin2}
  \int_Dy_k\zeta=A\int_D\nabla x_k\cdot \nabla\zeta +B \int_D\psi''(\varphi)x_k\zeta\,,\\
  \label{lin3}
  \ip{\partial_t z_k}{\zeta}_V + \int_D\nabla z_k\cdot\nabla\zeta
  +\int_D\left[cz_k h(\varphi) + c\sigma h'(\varphi)x_k + b(z_k-k_w)\right]\zeta=0
  \end{gather}
   for every $\zeta\in Z$, for almost every $t\in(0,T)$, $\P$-almost surely.
\end{thm}
Furthermore, the control-to-state map $\mathcal S$ is differentiable in a 
suitable sense and we characterize its derivative as the unique solution to the linearized system.
\begin{prop}
  \label{th:6}
  Assume {\em(A1)--(A7)}, \eqref{ip_ref1}--\eqref{ip_ref3}, \eqref{ip_H0} and \eqref{ip_ref4}.
  Then the control-to-state map $S:\tilde{\mathcal U}\to 
  L^p(\Omega; C^0([0,T]; H)\cap L^2(0,T; Z))$ is G\^ateaux-differentiable in the following sense:
  for every $(u,w)\in\tilde{\mathcal U}$ and $(k_u,k_w)\in L^q(\Omega; L^2(0,T; H))^2$ we have
  \begin{align*}
  \frac{S((u,w)+\eps(k_u,k_w))-S(u,w)}{\eps}\to x_k \qquad&\text{in }
  L^\ell(\Omega; L^2(0,T; V)) \quad\forall\,\ell\in[1,p)\,,\\
  \frac{S((u,w)+\eps(k_u,k_w))-S(u,w)}{\eps}\wto x_k \qquad&\text{in }
  L^p(\Omega; H^1(0,T; Z^*)\cap L^2(0,T; Z))
  \end{align*}
  as $\eps\to0$, where $x_k$ is the unique first solution component to the linearized system
  \eqref{xy}--\eqref{lin3}.
\end{prop}

The next result is a first version of necessary conditions for optimality.
\begin{prop}
  \label{th:7}
  Assume {\em(A1)--(A7)}, \eqref{ip_ref1}--\eqref{ip_ref3}, \eqref{ip_H0} and \eqref{ip_ref4}.
  Let $(\bar u, \bar w)\in \mathcal U$ be an optimal control and 
  let $\bar\varphi:=S(\bar u,\bar w)$. Then,
  for every $(u,w)\in\mathcal U$, setting 
  $k:=(k_u,k_w):=(u-\bar u, w-\bar w)$, we have 
  \[
  \beta_1\E\int_Q(\bar\varphi-\varphi_Q)x_k
  +\beta_2\E\int_D(\bar\varphi(T)-\varphi_T)x_k(T)
  +\frac{\beta_3}2\E\int_Dx_k(T)
  +\beta_4\E\int_Q\bar uk_u 
  +\beta_5\E\int_Q\bar w k_w
  \geq 0\,,
  \]
  where $x_k$ is the unique first solution component to the linearized system \eqref{xy}--\eqref{lin3}.
\end{prop}

In the last results of the paper we show how to remove 
the dependence on $x_k$ in the first-order necessary conditions
for optimality. To this end, we analyse the corresponding adjoint 
problem, which is a system of backward SPDEs of the form
\begin{equation*}
\begin{split}
  -\d\pi - A\Delta \tilde \pi\,\d t + B\psi''(\varphi)\tilde \pi\,\d t &= 
  h'(\varphi)(\mathcal P\sigma - a - \alpha u)\pi\,\d t
  - ch'(\varphi)\sigma \rho\,\d t 
  +\beta_1(\varphi-\varphi_Q)\,dt-\xi\,\d W_1\,,\\
  \tilde \pi &=-\Delta \pi\,,\\
  -\d\rho -\Delta \rho\,\d t &+ ch(\varphi)\rho\,\d t + b\rho\,\d t=\mathcal P h(\varphi)\pi\,\d t - \theta\,\d W_2\,,
\end{split}
\end{equation*}
complemented with homogeneous Neumann boundary conditions for $\pi$, $\tilde \pi$ and $\rho$,
with final conditions
\[
  \pi(T)=\beta_2(\varphi(T)-\varphi_T) + \frac{\beta_3}2\,, \qquad
  \rho(T)=0\,.
\]
The last two results 
deal with the existence and uniqueness of solutions to the adjoint system, 
and with the simplified version of the first-order necessary conditions
for optimality, respectively.

\begin{thm}
  \label{th:8}
  Assume {\em(A1)--(A7)}, \eqref{ip_ref1}--\eqref{ip_ref3}, \eqref{ip_H0} and \eqref{ip_ref4}.
  Let $(u,w)\in\tilde{\mathcal U}$ be fixed and set $\varphi:=S(u,w)$. 
  Then there exists a unique quintuplet $(\pi,\tilde \pi, \xi, \rho, \theta)$ with 
  \begin{gather}
    \label{pi}
    \pi \in L^2(\Omega;C^0([0,T]; V)\cap L^2(0,T; Z\cap H^3(D)))\,,\\
    \label{pi_tilde}
    \tilde \pi \in L^2(\Omega;C^0([0,T]; V^*)\cap L^2(0,T; V))\,,\\
    \label{rho}
    \rho \in L^2(\Omega;C^0([0,T]; H)\cap L^2(0,T; V))\,,\\
    \label{xi-sig}
    \xi \in L^2(\Omega; L^2(0,T; \cL^2(U_1,V)))\,, \qquad
    \theta\in L^2(\Omega; L^2(0,T; \cL^2(U_2,H)))\,,
  \end{gather}
  such that, for every $\zeta\in V$,
  \begin{gather}
  \label{ad1}
  \begin{split}
    &\int_D\pi(t)\zeta + A\int_t^T\!\!\int_D\nabla\tilde\pi(s)\cdot \nabla\zeta\,\d s 
    + B\int_t^T\!\!\int_D\psi''(\varphi(s))\tilde\pi(s)\zeta\,\d s
    =-\int_D\left(\int_t^T\xi(s)\,\d W_1(s)\right)\zeta\\
    &\qquad+\int_t^T\!\!\int_D\left[h'(\varphi(s))(\mathcal P\sigma(s)-a-\alpha u(s))\pi(s)
    -ch'(\varphi(s))\sigma(s)\rho(s)
    +\beta_1(\varphi-\varphi_Q)(s)\right]\zeta\,\d s\,,
  \end{split}\\
  \label{ad2}
  \ip{\tilde\pi(t)}{\zeta}_V=\int_D\nabla\pi(t)\cdot\nabla\zeta\,,\\
  \label{ad3}
  \begin{split}
  &\int_D\rho(t)\zeta + \int_t^T\!\!\int_D\nabla\rho(s)\cdot\nabla\zeta\,\d s
  +\int_t^T\!\!\int_D\left[ch(\varphi(s))\rho(s)+b\rho(s)\right]\zeta\,\d s\\
  &\qquad=
  \mathcal P\int_t^T\!\!\int_Dh(\varphi(s))\pi(s)\zeta\,\d s -
  \int_D\left(\int_t^T\theta(s)\,\d W_2(s)\right)\zeta
  \end{split}
  \end{gather}
  for every $t\in[0,T]$, $\P$-almost surely.
\end{thm}
The final version of the stochastic maximum principle reads as follows
\begin{thm}
  \label{th:9}
  Assume {\em(A1)--(A7)}, \eqref{ip_ref1}--\eqref{ip_ref3}, \eqref{ip_H0} and \eqref{ip_ref4}.
  Let $(\bar u, \bar w)\in \mathcal U$ be an optimal control and 
  let $\bar\varphi:=\mathcal S(\bar u,\bar w)$. Then
  \[
  \E\int_Q(\beta_4\bar u - \alpha h(\bar\varphi)\pi)(u-\bar u) + \E\int_Q(\beta_5\bar w +b\rho)(w-\bar w)\geq 0
  \qquad\forall\,(u,w)\in\mathcal U\,,
  \]
  where $(\pi,\rho)$ are the first-component solutions to the adjoint system \eqref{ad1}-\eqref{ad3}.
  In particular, if $\beta_4>0$ and $\beta_5>0$ then $(\bar u, \bar w)$ is the orthogonal 
  projection of $(\frac\alpha{\beta_4}h(\bar\varphi)\pi, -\frac{b}{\beta_5}\rho)$ on $\mathcal U$\,.
\end{thm}

\begin{rmk}
  Note that the first-order condition for optimality 
  can be equivalently rewritten pointwise in $\Omega\times[0,T]$
  by using standard localization techniques:
  see e.g.~\cite{YZ99}.
\end{rmk}


\section{Well-posedness of the state system}
\setcounter{equation}{0}
\label{sec:exist}
This section contains the proof of the Theorem~\ref{th:1}.
First of all, we consider an approximated problem where the nonlinearity $\psi'$
is smoothed out through a Yosida-type regularization. Secondly, 
we prove uniform estimates on the approximated solutions in suitable spaces, and
through the theorems of Prokhorov and Skorokhod we are then able to 
show  existence of (probabilistic) weak solutions. Finally,
we prove a pathwise uniqueness result for the original system, yielding 
existence and uniqueness also of strong solutions thanks to a well-known criterion.

\subsection{The approximated problem}
Thanks to assumption (A2), the function $\gamma(r):=\psi'(r) + C_2 r$, $r\in\erre$,
is nondecreasing, hence can be identified with a maximal monotone graph in $\erre\times\erre$.
For any $\lambda>0$, let $\gamma_\lambda$ be its Yosida approximation
and set 
\[
\psi_\lambda'(r):=\gamma_\lambda(r) - C_2 r\,, \quad r\in\erre\,.
\]
We consider the approximated problem 
\begin{align}
  \label{eq:1_app}
  \d\varphi_\lambda - \Delta\mu_\lambda\,\d t 
  = (\mathcal P\sigma_\lambda- a - \alpha u)h(\varphi_\lambda)\,\d t + G\,\d W_1 \qquad&\text{in } (0,T)\times D\,,\\
  \label{eq:2_app}
  \mu_\lambda = -A\Delta\varphi_\lambda + B\psi'_\lambda(\varphi_\lambda) \qquad&\text{in } (0,T)\times D\,,\\
  \label{eq:3_app}
  \d\sigma_\lambda - \Delta\sigma_\lambda \,\d t
  + c\sigma_\lambda h(\varphi_\lambda)\,\d t + b (\sigma_\lambda-w)\,\d t = 
  \mathcal H(\sigma_\lambda)\,\d W_2
  \qquad&\text{in } (0,T)\times D\,,\\
  \label{eq:bound_app}
  \partial_{\bf n}\varphi_\lambda = \partial_{\bf n}\mu_\lambda 
  = \partial_{\bf n}\sigma_\lambda = 0 \qquad&\text{in } (0,T)\times \partial D\,,\\
  \label{eq:init_app}
  \varphi_\lambda(0) = \varphi_0\,, \quad \sigma_\lambda(0)=\sigma_0 \qquad&\text{in } D\,.
\end{align}
We prove existence and uniqueness of a solution $(\varphi_\lambda, \mu_\lambda, \sigma_\lambda)$
through a fixed point argument: we fix a suitable $\phi$ in the third equation in place of $\varphi_\lambda$,
and we solve \eqref{eq:3_app} obtaining thus a solution component $\sigma_\lambda^\phi$
depending on $\phi$. We substitute $\sigma_\lambda^\phi$ into \eqref{eq:1_app}
and we solve \eqref{eq:1_app}--\eqref{eq:2_app}, getting the other two solution components 
$(\varphi_\lambda^\phi, \mu_\lambda^\phi)$ in terms on $\phi$. Finally, we show that the map
$\phi\mapsto\varphi_\lambda^\phi$ is well-defined in a suitable space and is a contraction.

Let us fix a progressively measurable $H$-valued process $\phi$ with
\[
  \phi \in L^2\left(\Omega; L^2(0,T; H)\right)\,.
\]
Since $h$ takes values in $[0,1]$ it is clear now that the initial-value problem
\[
  \begin{cases}
  \d\sigma_\lambda^\phi - \Delta\sigma_\lambda^\phi\,\d t
  + c\sigma_\lambda^\phi h(\phi)\,\d t + b (\sigma^\phi_\lambda-w)\,\d t = 
  \mathcal H(\sigma_\lambda)\,\d W_2
  \qquad&\text{in } (0,T)\times D\,,\\
  \partial_{\bf n}\sigma^\phi_\lambda = 0 \qquad&\text{in } (0,T)\times \partial D\,,\\
  \sigma_\lambda^\phi(0)=\sigma_0 \qquad&\text{in } D\,,
  \end{cases}
\]
admits a unique solution (see \cite{garcke-lam-roc, LiuRo})
\[
  \sigma_\lambda^\phi \in L^2\left(\Omega; C^0([0,T]; H)\cap L^2(0,T; V)\right)\,.
\]
Let us show now that $\sigma_\lambda\in[0,1]$ almost everywhere in $\Omega\times Q$.
Introduce first a smooth approximation $s_\eps:\erre\to\erre$ of the function 
$r\mapsto (r-1)_+$, $r\in\erre$: for example take
\[
  s_\eps(r):=
  \begin{cases}
  0 \quad&\text{if } r\leq 1\,,\\
  \frac1{2\eps}(r-1)^2 \quad&\text{if } r\in[1,1+\eps]\,,\\
  r-1-\frac\eps2\quad&\text{if } r>1+\eps\,,
  \end{cases}
\]
so that $\hat s_\eps:=\int_1^\cdot s_\eps(r)\,\d r \in C^2(\erre)$. 
Associate to $\hat s_\varepsilon$ an integral operator of the form $\mathcal{S}_\varepsilon:= \int_D \hat s_\varepsilon(\sigma(x)) \d x$, where $\sigma \in H$, and exploit the (sub-quadratic) growth of $\hat s_\varepsilon$ to show that for any $\sigma, v \in H$
\begin{equation}
\mathcal{S'}_\varepsilon(\sigma) v = \int_D s_\varepsilon(\sigma) v, \qquad \left[ \mathcal{S''}_\varepsilon(\sigma) v \right] = s'_\varepsilon(\sigma(x))v(x).
\end{equation}

This permit the application of  
It\^o's formula in the form of \cite[Thm.~4.1]{Pard}
\begin{align*}
  \E\int_D\hat s_\eps(\sigma_\lambda^\phi(t)) &+ 
  \E\int_{Q_t} s_\eps'(\sigma_\lambda^\phi)|\nabla\sigma_\lambda^\phi|^2 +
  c\E\int_{Q_t}h(\phi)\sigma_\lambda^\phi s_\eps(\sigma_\lambda^\phi)
  +b\E\int_{Q_t}(\sigma_\lambda^\phi - w) s_\eps(\sigma_\lambda^\phi)\\
  &=\E\int_D\hat s_\eps(\sigma_0)
  +\frac12\E\sum_{n=0}^\infty\int_{Q_t} s_\eps'(\sigma_\lambda^\phi)
  |\mathcalboondox{h}_n(\sigma_\lambda^\phi)|^2\,.
\end{align*}
Since $s_\eps$ is increasing and identically $0$ in $(-\infty,0]$, the second
and third term on the left-hand side are nonnegative. Moreover, 
using the fact that $s_\eps'\leq1$ and that $\mathcalboondox{h}_n$ is Lipschitz-continuous with $\mathcalboondox{h}_n(1)=0$, 
recalling that $\sigma_0\in[0,1]$ a.e.~in $\Omega\times D$ we get
\[
  \E\int_D\hat s_\eps(\sigma_\lambda^\phi(t)) + 
  b\E\int_{Q_t}(\sigma_\lambda^\phi - w)s_\eps(\sigma_\lambda^\phi)
  \lesssim\frac12\E\int_{\{\sigma_\lambda^\phi>1\}}|\sigma_\lambda^\phi-1|^2=
  \frac12\E\int_{Q_t}|(\sigma_\lambda^\phi-1)_+|^2\,.
\]
Letting $\eps\searrow0$ we deduce that 
\[
  \frac12\E\int_D|(\sigma_\lambda^\phi-1)_+(t)|^2
  +b\E\int_{Q_t}(\sigma_\lambda^\phi-w)(\sigma_\lambda^\phi-1)_+
  \leq\frac12\E\int_{Q_t}|(\sigma_\lambda^\phi-1)_+|^2\,,
\]
where, since $w\in[0,1]$ a.e.~in $\Omega\times Q$,
\[
  (\sigma_\lambda^\phi-w)(\sigma_\lambda^\phi-1)_+=
  (\sigma_\lambda^\phi-1)(\sigma_\lambda^\phi-1)_+ + 
  (1-w)(\sigma_\lambda^\phi-1)_+ \geq0 \quad\text{a.e.~in } \Omega\times Q\,.
\]
The Gronwall lemma yields then $(\sigma_\lambda^\phi-1)_+=0$,
hence $\sigma_\lambda^\phi\leq1$, a.e~in $\Omega\times Q$.
Arguing similarly with $-(\sigma_\lambda^\phi)_-$, one also deduces that 
$\sigma_\lambda^\phi\geq0$
almost everywhere, so that 
\[
  \sigma_\lambda^\phi \in [0,1] \qquad\text{a.e.~in } \Omega\times Q\,.
\]
Let us substitute now $\sigma_\lambda^\phi$ in \eqref{eq:1_app} and consider the Cahn-Hilliard system
\[
  \begin{cases}
  \d\varphi_\lambda^\phi - \Delta\mu_\lambda^\phi\,\d t 
  = (\mathcal P\sigma^\phi_\lambda-a- \alpha u)h(\varphi_\lambda^\phi)\,\d t + G\,\d W_1
  \qquad&\text{in } (0,T)\times D\,,\\
  \mu_\lambda^\phi = -A\Delta\varphi^\phi_\lambda + B\psi'_\lambda(\varphi^\phi_\lambda) 
  \qquad&\text{in } (0,T)\times D\,,\\
  \partial_{\bf n}\varphi^\phi_\lambda = \partial_{\bf n}\mu^\phi_\lambda 
  = 0 \qquad&\text{in } (0,T)\times \partial D\,,\\
  \varphi^\phi_\lambda(0) = \varphi_0 \qquad&\text{in } D\,.
  \end{cases}
\]
Arguing as in \cite{scar-SCH, scar-SVCH}, by the Lipschitz-continuity of $\psi_\lambda'$
and the fact that $h, u$ and $\sigma_\lambda^\phi$ are bounded almost everywhere,
there is a unique solution $(\varphi_\lambda^\phi, \mu_\lambda^\phi)$ with 
\begin{gather*}
  \varphi_\lambda^\phi \in L^2\left(\Omega; C^0([0,T]; H)\cap L^\infty(0,T; V)\cap L^2(0,T; Z)\right)\,,\\
  \mu_\lambda^\phi \in L^2(\Omega; L^2(0,T; V))\,.
\end{gather*}
For every $\lambda>0$, it is well-defined then the map
\[
  \Phi_\lambda:L^2\left(\Omega; L^2(0,T; H)\right)\to
  L^2\left(\Omega; C^0([0,T]; H)\cap L^\infty(0,T; V)\cap L^2(0,T; Z)\right)\,, \qquad
  \phi\mapsto \varphi_\lambda^\phi\,.
\]
Let us show that $\Phi_\lambda$ is a contraction: let $\phi_1, \phi_2 \in L^2(\Omega; L^2(0,T; H))$
progressively measurable.
First of all, from the third equation we have that 
\begin{align*}
  \d(\sigma_\lambda^{\phi_1}-\sigma_\lambda^{\phi_2}) 
  &-\Delta (\sigma_\lambda^{\phi_1}-\sigma_\lambda^{\phi_2}) \,\d t
  +c\left(\sigma_\lambda^{\phi_1}h(\phi_1)-\sigma_\lambda^{\phi_2}h(\phi_2)\right)\,\d t
  +b(\sigma_\lambda^{\phi_1}-\sigma_\lambda^{\phi_2})\,\d t \\
  &= \left(\mathcal H(\sigma_\lambda^{\phi_1}) - \mathcal H(\sigma_\lambda^{\phi_2})\right)\,\d W_2\,,
\end{align*}
with $(\sigma_\lambda^{\phi_1}-\sigma_\lambda^{\phi_2})(0)=0$.
Hence, It\^o's formula for the square of the $H$-norm yields
\begin{equation}\label{ito_sigma_lambda}
\begin{split}
  &\frac12\norm{(\sigma_\lambda^{\phi_1}-\sigma_\lambda^{\phi_2})(t)}_H^2
  +\int_{Q_t}|\nabla(\sigma_\lambda^{\phi_1}-\sigma_\lambda^{\phi_2})|^2
  +\int_{Q_t}\left(ch(\phi_1) + b\right)|\sigma_\lambda^{\phi_1}-\sigma_\lambda^{\phi_2}|^2\\
  &=c\int_{Q_t}\left(h(\phi_2)-h(\phi_1)\right)\sigma_\lambda^{\phi_2}
  (\sigma_\lambda^{\phi_1}-\sigma_\lambda^{\phi_2})
  +\frac12\int_0^t\norm{(\mathcal H(\sigma_\lambda^{\phi_1}) 
  - \mathcal H(\sigma_\lambda^{\phi_2}))(s)}^2_{\cL^2(U_2,H)}\,\d s\\
  &\qquad+\int_0^t\left((\sigma_\lambda^{\phi_1}-\sigma_\lambda^{\phi_2})(s),
  (\mathcal H(\sigma_\lambda^{\phi_1}) 
  - \mathcal H(\sigma_\lambda^{\phi_2}))(s)\right)_H\,\d W_2(s)\,.
\end{split}
\end{equation}
Now, employing the Lipschitz-continuity of $h$ and the fact that $\sigma_\lambda^{\phi_2}\in[0,1]$ 
a.e.~in $\Omega\times Q$ we have
\[
  c\int_{Q_t}\left(h(\phi_2)-h(\phi_1)\right)\sigma_\lambda^{\phi_2}
  (\sigma_\lambda^{\phi_1}-\sigma_\lambda^{\phi_2})\leq
  \frac12\int_{Q_t}|\sigma_\lambda^{\phi_1}-\sigma_\lambda^{\phi_2}|^2
  +\frac{c^2L_h^2}2\int_Q|\phi_1-\phi_2|^2\,.
\] 
Moreover, the Lipschitz continuity of $\mathcal H$ yields
\[
  \int_0^t\norm{(\mathcal H(\sigma_\lambda^{\phi_1}) 
  - \mathcal H(\sigma_\lambda^{\phi_2}))(s)}^2_{\cL^2(U_2,H)}\,\d s\leq
  L_{\mathcal H}^2
  \int_{Q_t}|\sigma_\lambda^{\phi_1}-\sigma_\lambda^{\phi_2}|^2
\]
while the Burkholder-Davis-Gundy (cf.\cite[Lemma~4.1]{fuhr-tes}) and Young inequalities
(see \cite[Lemma~4.3]{mar-scar-diss}) imply that,
for some positive constants $M$ and $M'$,
\begin{align*}
  &\E\sup_{r\in[0,t]}\left|\int_0^r\left((\sigma_\lambda^{\phi_1}-\sigma_\lambda^{\phi_2})(s),
  (\mathcal H(\sigma_\lambda^{\phi_1}) 
  - \mathcal H(\sigma_\lambda^{\phi_2}))(s)\right)_H\,\d W_2(s)\right|\\
  &\qquad\leq M\E\left(\int_0^t\norm{(\sigma_\lambda^{\phi_1}-\sigma_\lambda^{\phi_2})(s)}_H^2
  \norm{(\mathcal H(\sigma_\lambda^{\phi_1}) 
  - \mathcal H(\sigma_\lambda^{\phi_2}))(s)}_{\cL^2(U_2,H)}^2\,\d s\right)^{1/2}\\
  &\qquad\leq ML_{\mathcal H}
  \E\left(\norm{\sigma_\lambda^{\phi_1}-\sigma_\lambda^{\phi_2}}_{C^0([0,t]; H)}
  \norm{\sigma_\lambda^{\phi_1}-\sigma_\lambda^{\phi_2}}_{L^2(0,t; H)}\right)\\
  &\qquad\leq\frac14\E\norm{\sigma_\lambda^{\phi_1}-\sigma_\lambda^{\phi_2}}_{C^0([0,t]; H)}^2
  +M'\E\int_{Q_t}|\sigma_\lambda^{\phi_1}-\sigma_\lambda^{\phi_2}|^2\,.
\end{align*}
Hence, taking supremum in time and expectations, rearranging the terms
and using the Gronwall lemma, we deduce that 
\beq\label{ineq1}
  \norm{\sigma_\lambda^{\phi_1}-\sigma_\lambda^{\phi_2}}_{L^2\left(\Omega; C^0([0,T]; H)\cap 
  L^2(0,T; V)\right)} \leq M\norm{\phi_1-\phi_2}_{L^2(\Omega; L^2(0,T; H))}\,,
\eeq
where $M>0$ only depends on $c$, $L_h$, and $L_{\mathcal H}$.
Secondly, from the first two equations we have 
\begin{gather*}
  \partial_t(\varphi_\lambda^{\phi_1}-\varphi_\lambda^{\phi_2}) - \Delta(\mu_\lambda^{\phi_1}-\mu_\lambda^{\phi_2})
  =(\mathcal P\sigma_\lambda^{\phi_1} - a - \alpha u)h(\varphi_\lambda^{\phi_1}) - (\mathcal P\sigma_\lambda^{\phi_2} - a- \alpha u)h(\varphi_\lambda^{\phi_2})\,,\\
  \mu_\lambda^{\phi_1}-\mu_\lambda^{\phi_2} = -A\Delta(\varphi_\lambda^{\phi_1}-\varphi_\lambda^{\phi_2})
  +B\psi_\lambda'(\varphi_\lambda^{\phi_1})-B\psi_\lambda'(\varphi_\lambda^{\phi_2})
\end{gather*}
with initial condition $(\varphi_\lambda^{\phi_1}-\varphi_\lambda^{\phi_2})(0)=0$. Hence testing the first
by $\varphi_\lambda^{\phi_1}-\varphi_\lambda^{\phi_2}$ and integrating by parts yield
\begin{align*}
  \norm{(\varphi_\lambda^{\phi_1}-\varphi_\lambda^{\phi_2})(t)}^2_{H}
  &+A\int_{Q_t}|\Delta(\varphi_\lambda^{\phi_1}-\varphi_\lambda^{\phi_2})|^2\\
  &=\int_{Q_t}\left((\mathcal P\sigma_\lambda^{\phi_1} - a - \alpha u)h(\varphi_\lambda^{\phi_1}) 
  - (\mathcal P\sigma_\lambda^{\phi_2} - a - \alpha u)h(\varphi_\lambda^{\phi_2})\right)
  (\varphi_\lambda^{\phi_1}-\varphi_\lambda^{\phi_2})\\
  &+B\int_{Q_t}\left(\psi_\lambda'(\varphi_\lambda^{\phi_1})-\psi_\lambda'(\varphi_\lambda^{\phi_2})\right)
  \Delta(\varphi_\lambda^{\phi_1}-\varphi_\lambda^{\phi_2})\,.
\end{align*}
Hence, using the Young inequality on the right-hand side yields
\begin{align*}
  &\norm{(\varphi_\lambda^{\phi_1}-\varphi_\lambda^{\phi_2})(t)}^2_{H}
  +A\int_{Q_t}|\Delta(\varphi_\lambda^{\phi_1}-\varphi_\lambda^{\phi_2})|^2\\
  &\quad\leq\int_{Q_t}|\varphi_\lambda^{\phi_1}-\varphi_\lambda^{\phi_2}|^2
  +\int_{Q_t}|\mathcal P\sigma_\lambda^{\phi_1} - a - \alpha u|^2|h(\varphi_\lambda^{\phi_1}) 
  -h(\varphi_\lambda^{\phi_2})|^2
  +\mathcal P^2\int_{Q_t}|\sigma_\lambda^{\phi_1}-\sigma_\lambda^{\phi_2}|^2|h(\varphi_\lambda^{\phi_2})|^2\\
  &\quad+\frac{A}2\int_{Q_t}|\Delta(\varphi_\lambda^{\phi_1}-\varphi_\lambda^{\phi_2})|^2
  +\frac{B^2}{2A}\int_{Q_t}|\psi_\lambda'(\varphi_\lambda^{\phi_1})-\psi_\lambda'(\varphi_\lambda^{\phi_2})|^2\,,
\end{align*}
from the Lipschitz-continuity of $h$ and $\psi'_\lambda$,
and the fact that $\sigma_\lambda^\phi$ and $h$ are bounded in $[0,1]$ we get
\begin{align*}
  &\norm{(\varphi_\lambda^{\phi_1}-\varphi_\lambda^{\phi_2})(t)}^2_{H}
  +\frac{A}2\int_{Q_t}|\Delta(\varphi_\lambda^{\phi_1}-\varphi_\lambda^{\phi_2})|^2\\
  &\qquad\leq\left(1+(\mathcal PL_h)^2+\frac{B^2}{2A}\left|\frac1\lambda\vee C_2\right|^2\right)
  \int_{Q_t}|\varphi_\lambda^{\phi_1}-\varphi_\lambda^{\phi_2}|^2
  +\mathcal P^2\int_{Q_t}|\sigma_\lambda^{\phi_1}-\sigma_\lambda^{\phi_2}|^2\,.
\end{align*}
The Gronwall lemma implies that, possibly updating the value of the implicit constant $M$,
\beq\label{ineq2}
  \norm{\varphi_\lambda^{\phi_1}-\varphi_\lambda^{\phi_2}}_{L^2(\Omega; C^0([0,T]; H)\cap L^2(0,T; Z))}
  \leq M_{\mathcal P, L_h, A,B, C_2, \lambda} 
  \norm{\sigma_\lambda^{\phi_1} - \sigma_{\lambda}^{\phi_2}}_{L^2(\Omega; L^2(0,T; H))}\,.
\eeq
Combining now \eqref{ineq1} and \eqref{ineq2} we obtain that 
\[
  \norm{\varphi_\lambda^{\phi_1}-\varphi_\lambda^{\phi_2}}_{L^2(\Omega; C^0([0,T]; H)\cap L^2(0,T; Z))}
  \leq M_{c, \mathcal P, L_h, L_{\mathcal H}, A,B, C_2, \lambda} \sqrt T
  \norm{\phi_1-\phi_2}_{L^2(\Omega; L^2(0,T; H))}\,.
\]
We deduce that there is $T_0>0$ sufficiently small such that 
$\Phi_\lambda$ is a contraction on $L^2(\Omega; L^2(0,T_0; H))$, 
hence admits a fixed point in $L^2(\Omega; L^2(0,T_0; H))$.
A standard patching technique yields now the existence and uniqueness of a fixed point
$\varphi_\lambda$ in the whole space $L^2(\Omega; L^2(0,T; H))$.
Setting now $\mu_\lambda:=\mu_\lambda^{\varphi_\lambda}$ and 
$\sigma_\lambda:=\sigma_\lambda^{\varphi_\lambda}$,
the definition of $\Phi_\lambda$ itself ensures also that 
\begin{gather*}
  \varphi_\lambda \in L^2\left(\Omega; C^0([0,T]; H)\cap L^\infty(0,T; V)\cap L^2(0,T; Z)\right)\,,\qquad
  \mu_\lambda \in L^2(\Omega; L^2(0,T; V))\,,\\ 
  \sigma_\lambda \in L^2\left(\Omega; C^0([0,T]; H)\cap L^2(0,T; V)\right)\,,\qquad
  \sigma_\lambda \in [0,1] \quad\text{a.e.~in } \Omega\times Q\,,
\end{gather*}
and that $(\varphi_\lambda, \mu_\lambda, \sigma_\lambda)$ is the unique solution to 
the approximated problem \eqref{eq:1_app}--\eqref{eq:init_app}.

\subsection{Uniform estimates}
Here we perform uniform a-priori estimates with respect to the parameter $\lambda$.

\noindent {\bf First estimate.} We write It\^o's formula for $\frac12\norm{\sigma_\lambda}_H^2$,
getting
\begin{align}\nonumber
  &\frac12\norm{\sigma_\lambda(t)}_H^2
  +\int_{Q_t}|\nabla \sigma_\lambda|^2
  +\int_{Q_t}\left(ch(\varphi_\lambda) + b\right)|\sigma_\lambda|^2
  =\frac12\norm{\sigma_0}_H^2
  +b\int_{Q_t}w\sigma_\lambda\\
  &\qquad+\frac12\int_0^t\norm{\mathcal H(\sigma_\lambda(s))}^2_{\cL^2(U_2,H)}\,\d s
  +\int_0^t\left(\sigma_\lambda(s),
  \mathcal H(\sigma_\lambda(s))\right)_H\,\d W_2(s)\,.\label{unif}
\end{align}
Since $w\leq1$ almost everywhere we have, by the Young inequality,
\[
  b\int_{Q_t}w\sigma_\lambda\lesssim_b 1+\int_{Q_t}|\sigma_\lambda|^2\,.
\] 
Moreover, the Lipschitz continuity of $\mathcal H$ and the fact that $\mathcalboondox{h}_n(0)=0$ for all $n\in\enne$ yields
\[
  \int_0^t\norm{\mathcal H(\sigma_\lambda(s))}^2_{\cL^2(U_2,H)}\,\d s\leq
  L_{\mathcal H}^2
  \int_{Q_t}|\sigma_\lambda|^2\,.
\]
Finally, by the Burkholder-Davis-Gundy and Young inequalities (see e.g. \cite[Lem.~4.1]{mar-scar-ref}),
for every $\eps>0$ and a certain $M_\eps>0$ we get
\begin{align*}
  \norm{\sup_{r\in[0,t]}\int_0^r\left(\sigma_\lambda(s),
  \mathcal H(\sigma_\lambda(s))\right)_H\, \d W_2(s)}_{L^{p/2}(\Omega)}
  &\lesssim \norm{\left(\int_0^t\norm{\sigma_\lambda(s)}_H^2
  \norm{\mathcal H(\sigma_\lambda(s))}_{\cL^2(U_2,H)}^2\,\d s\right)^{1/2}}_{L^{p/2}(\Omega)}\\
  &\leq\eps\norm{\sigma_\lambda}^2_{L^{p}(\Omega; C^0([0,t]; H))}
  +M_\eps\norm{\mathcal H(\sigma_\lambda)}^2_{L^{p}(\Omega; L^2(0,t; \cL^2(U_2,H))}\\
  &\leq\eps\norm{\sigma_\lambda}^2_{L^{p}(\Omega; C^0([0,t]; H))}
  +M_\eps L_{\mathcal H}^2\|\sigma_\lambda\|^2_{L^p(\Omega;L^2(0,T;H))}\,.
\end{align*}
Taking supremum in time and $L^{p/2}(\Omega)$-norm in \eqref{unif}, recalling that $\sigma_\lambda\in[0,1]$, 
there exists $M>0$, independent of $\lambda$ and depending only on 
$T$ and $\norm{\sigma_0}_{L^p(\Omega;H)}^2$, such that 
\beq\label{est1}
  \norm{\sigma_\lambda}_{L^p(\Omega; C^0([0,T]; H)\cap L^2(0,T; V))\cap L^\infty(\Omega\times Q)} \leq M\,,
\eeq
hence, by the Lipschitz continuity of $\mathcal H$ and by comparison in the equation itself, 
\beq
  \label{est1_bis}
  \norm{\mathcal H(\sigma_\lambda)}_{L^p(\Omega; C^0([0,T]; \cL^2(U_2,H)))} + 
  \norm{\sigma_\lambda-\mathcal H(\sigma_\lambda)\cdot W_2}_{L^p(\Omega; H^1(0,T; V^*))}\leq M\,.
\eeq

\noindent {\bf Second estimate.} We write It\^o's formula for the square of the $H$-norm of $\varphi_\lambda$:
taking into account the boundary conditions for $\varphi_\lambda$ and $\mu_\lambda$,
integrating by parts we have
\begin{align*}
  \frac12\norm{\varphi_\lambda(t)}_H^2 &+ A\int_{Q_t}|\Delta\varphi_\lambda|^2
  -B\int_{Q_t}\psi'_\lambda(\varphi_\lambda)\Delta\varphi_\lambda
  =\frac12\norm{\varphi_0}^2_H + \int_{Q_t}(\mathcal P\sigma_\lambda - a- \alpha u)h(\varphi_\lambda)\varphi_\lambda\\
  &+\frac12\int_0^t\norm{G(s)}^2_{\cL^2(U_1,H)}\,\d s
  +\int_0^t\left(\varphi_\lambda(s),G(s)\right)_H\,\d W_1(s)\,.
\end{align*}
Recalling the definition of $\psi'_\lambda$ and the fact that $\sigma_\lambda$, $h$ and $u$
are bounded in $[0,1]$, we deduce
\begin{align*}
  \norm{\varphi_\lambda(t)}_H^2 &+ 2A\int_{Q_t}|\Delta\varphi_\lambda|^2
  +2B\int_{Q_t}\gamma''_\lambda(\varphi_\lambda)|\nabla\varphi_\lambda|^2
  \leq \norm{\varphi_0}^2_H + 2\mathcal P\int_{Q_t}|\varphi_\lambda|\\
  &-C_2\int_{Q_t}\varphi_\lambda\Delta\varphi_\lambda
  +\norm{G}^2_{L^2(0,T; \cL^2(U_1,H))}
  +\int_0^t\left(\varphi_\lambda(s),G(s)\right)_H\,\d W_1(s)\,.
\end{align*}
The Young inequality yields then
\begin{align*}
  &\norm{\varphi_\lambda(t)}_H^2 + \int_{Q_t}|\Delta\varphi_\lambda|^2
  +\int_{Q_t}\gamma''_\lambda(\varphi_\lambda)|\nabla\varphi_\lambda|^2\\
  &\quad\lesssim_{A,B,\mathcal P, a,\alpha, C_2}1+ \norm{\varphi_0}^2_H + \int_{Q_t}|\varphi_\lambda|^2
  +\norm{G}^2_{L^2(0,T; \cL^2(U,H))}
  +\int_0^t\left(\varphi_\lambda(s),G(s)\right)_H\,\d W_1(s)\,,
\end{align*}
where the implicit constant is independent of $\lambda$.
It is a standard matter to check that the Burkholder-Davis-Gundy and Young
inequalities on the right-hand side ensure that, for every $\delta>0$,
\[
  \norm{\sup_{r\in[0,t]}\int_0^r\left(\varphi_\lambda(s),G(s)\right)_H\,\d W(s)}_{L^{p/2}(\Omega)}^{p/2}
  \lesssim\delta\E\sup_{r\in[0,t]}\norm{\varphi_\lambda(r)}^p_H
  + \frac1{4\delta}\norm{G}^p_{L^p(\Omega; L^2(0,T; \cL^2(U,H)))}\,.
\]
Hence, taking supremum in time and $L^{p/2}(\Omega)$-norm, choosing 
$\delta$ sufficiently small and rearranging the terms, 
we infer that there exists a positive constant $M$, independent of $\lambda$, such that 
\beq\label{est2}
  \norm{\varphi_\lambda}_{L^p(\Omega; C^0([0,T]; H)\cap L^2(0,T; Z))} \leq M\,.
\eeq

\noindent {\bf Third estimate.} The idea is to write It\^o's formula for the free-energy functional
\[
  \mathcal E_\lambda(\varphi_\lambda):=\frac{A}2\int_D|\nabla \varphi_\lambda|^2 
  + B\int_D\psi_\lambda(\varphi_\lambda)\,.
\]
Note that the regularity of $\varphi_\lambda$ and $\psi_\lambda$ are not enough in order 
to do so, as $\mathcal E_\lambda$ may not be 
twice Fr\'echet-differentiable in $V$, and $\varphi_\lambda$ is not necessarily 
continuous in $V$: consequently, a rigorous approach would require a further approximation
on the problem. However, 
since this is not restrictive in our direction, we shall proceed formally 
in order to avoid heavy notations, and refer to 
\cite{scar-SCH, scar-SVCH} for a rigorous approach instead.
Noting that, by \eqref{eq:2_app}, for every $x,y, y_1,y_2 \in V$
\begin{equation*}
\begin{split}
D\mathcal E_\lambda(x)[y]&=\langle -A\Delta(x) + B\psi'_\lambda(x), y \rangle_V=\langle \mu_\lambda\,, y \rangle \,,\\
D_2\mathcal{E}_\lambda(x)[y_1,y_2]&= \int_D \nabla y_1 \cdot\nabla y_2 + \int_D \psi''(x)y_1y_2 \,,
\end{split}
\end{equation*}
we have that 
\begin{align*}
  \frac{A}2\int_D|\nabla \varphi_\lambda(t)|^2 &+ B\int_D\psi_\lambda(\varphi_\lambda(t)) 
  + \int_{Q_t}|\nabla \mu_\lambda|^2\\
  &=\frac{A}2\int_D|\nabla \varphi_0|^2 + B\int_D\psi_\lambda(\varphi_0)
  +\int_{Q_t}(\mathcal P\sigma_\lambda -a- \alpha u)h(\varphi_\lambda)\mu_\lambda\\
  &+\frac12\int_0^t\operatorname{Tr}\left(G^*(s)D^2\mathcal E_\lambda(\varphi_\lambda(s))G(s)\right)\,\d s
  +\int_0^t\left(\mu_\lambda(s), G(s)\right)_H\,\d W_1(s)\,.
\end{align*}
On the right-hand side, the fact that $\sigma_\lambda$, $h$ and $u$ are bounded in $[0,1]$
together with the Poincar\'e  and Young inequalities imply that, for every $\delta>0$,
\begin{align*}
  \int_{Q_t}(\mathcal P\sigma_\lambda -a - \alpha u)h(\varphi_\lambda)\mu_\lambda&\leq
  (\mathcal P + a + \alpha)\int_{Q_t}|\mu_\lambda-(\mu_\lambda)_D| 
  + (\mathcal P+a + \alpha)\int_0^t|(\mu_\lambda(s))_D|\,\d s\\
  &\lesssim_{\mathcal P, a, \alpha} \frac1\delta+ 
  \delta\int_{Q_t}|\nabla\mu_\lambda|^2 + \int_0^t|(\mu_\lambda(s))_D|\,\d s\,,
\end{align*}
where the implicit constant is independent of $\lambda$.
Now, formally writing 
\[
  \left(\mu_\lambda, G\right)_H =\left(\mu_\lambda-(\mu_\lambda)_D, G\right)_H 
  +|D|(\mu_\lambda)_DG_D \lesssim
  \norm{\nabla\mu_\lambda}_H\norm{G}_{V^*} + (\mu_\lambda)_D\norm{G}_{V^*}\,,
\]
the quadratic variation of the stochastic integral on the right-hand side
can be bounded employing the Young inequality by
\[
\begin{split}
&\left(\int_0^t(\norm{\nabla\mu_\lambda(s)}_H^2+|(\mu_\lambda(s))_D|^2) \norm{G(s)}_{\cL^2(U_1,V^*)}^2\,\d s\right)^{1/2} \\
&\quad \lesssim \delta\int_{Q_t}|\nabla \mu_\lambda|^2 + \frac1\delta\norm{G}^2_{L^\infty(0,T; \cL^2(U_1,V^*))}
  +\norm{G}_{L^\infty((0,T); \cL^2(U_1,V^*))}\norm{(\mu_\lambda)_D}_{L^2(0,t)}\,.
\end{split}
\]
Hence, using the Burkholder-Davis-Gundy inequality we end up with
\begin{align*}
  &\left\|\sup_{r\in[0,t]}\int_0^r\left(\mu_\lambda(s), G(s)\right)_H\,\d W_1(s)\right\|_{L^{p/2}(\Omega)}^{p/2}\\
  &\lesssim \delta^{p/2}
  \E\left(\int_{Q_t}|\nabla \mu_\lambda|^2 \right)^{p/2} 
  + \frac1{\delta^{p/2}}\norm{G}^p_{L^p(\Omega; L^\infty(0,T; \cL^2(U_1,V^*)))}\\
  &\qquad+\norm{G}_{L^\infty(\Omega\times(0,T); \cL^2(U_1,V^*))}^{p/2}
  \E\norm{(\mu_\lambda)_D}^{p/2}_{L^2(0,t)}\\
  &\lesssim
  \delta^{p/2}\E\left(\int_{Q_t}|\nabla \mu_\lambda|^2 \right)^{p/2} + 
  \frac1{\delta^{p/2}}
  +\sqrt{t}\E\sup_{r \in [0,t]}|(\mu_\lambda(r))_D|^{p/2}\,,
\end{align*}
where the implicit constant depends on $\norm{G}_{L^\infty(\Omega\times(0,T); \cL^2(U_1,V^*))}$
but is independent of $\lambda$.
Finally, given a complete orthonormal system $(e_j)_j$ of $U_1$, 
since $\psi''\lesssim1+|\psi'|$
and $V\embed L^6(D)$, by the H\"older inequality we have
\begin{align*}
&\int_0^t\operatorname{Tr}\left(G^*(s)D^2\mathcal E_\lambda(\varphi_\lambda(s))G(s)\right)\,\d s
=\int_0^t\sum_{j=0}^\infty\left(\norm{\nabla G(s)e_j}_H^2 
+ \int_D\psi''_\lambda(\varphi_\lambda(s))|G(s)e_j|^2\right)\,\d s\\
&\qquad\leq\norm{G}^2_{L^2(0,T; \cL^2(U_1,V))}
+\int_0^t\norm{\psi''_\lambda(\varphi_\lambda(s))}_{L^{3/2}(D)}
\norm{G(s)}^2_{\cL^2(U_1,V)}\,\d s\\
&\qquad\lesssim1+\norm{G}^2_{L^2(0,T; \cL^2(U_1,V))}
+\norm{\psi'_\lambda(\varphi_\lambda)}_{L^2(0,t;H)}
\norm{G}^2_{L^4(0,T;\cL^2(U_1,V))}\,.
\end{align*}
Hence, exploiting \eqref{eq:2_app}, we have that, for every $\eps>0$ and 
a certain $C_\eps>0$,
\begin{align*}
\int_0^t\operatorname{Tr}\left(G^*(s)D^2\mathcal E_\lambda(\varphi_\lambda(s))G(s)\right)\,\d s
&\lesssim1+ \norm{\Delta \varphi_\lambda}^2_{L^2(0,T; H)} +
\eps\int_{Q_t}|\nabla\mu_\lambda|^2 + C_\eps\norm{G}^4_{L^4(0,T;\cL^2(U_1,V))}\\
&\quad+\norm{G}^2_{L^\infty(\Omega; L^4(0,T; \cL^2(U_1,V)))}\sup_{r\in[0,t]}\sqrt{t}|(\mu_\lambda(r))_D|.
\end{align*}
Consequently, taking supremum in time and $L^{p/2}(\Omega)$ norm in It\^o's formula, choosing
$\delta,\eps$ sufficiently small and rearranging the terms yield, for every $T_0\in(0,T]$,
\begin{align*}
  &\left\|\sup_{r\in[0,T_0]}\int_D|\nabla \varphi_\lambda(r)|^2 \right\|_{L^{p/2}(\Omega)}
  +\left\|\sup_{r\in[0,T_0]}\int_D\psi_\lambda(\varphi_\lambda(r))\right\|_{L^{p/2}(\Omega)}
  + \left\|\int_{Q_{T_0}}|\nabla \mu_\lambda|^2 \right\|_{L^{p/2}(\Omega)}\\
  &\lesssim1+
  \norm{\varphi_0}^2_{L^p(\Omega; V)} + \norm{\psi(\varphi_0)}_{L^{p/2}(\Omega; L^1 (D))}+
  \norm{\Delta \varphi_\lambda}^2_{L^{p}(\Omega; L^2(0,T; H))}
  +\sqrt{T_0}\|(\mu_\lambda)_D\|_{L^{p/2}(\Omega;L^\infty(0,T_0))}\,,
\end{align*}
where the implicit constant only depends on 
$\norm{G}_{L^\infty(\Omega\times(0,T); \cL^2(U_1,H))}$, $\norm{G}_{L^\infty(\Omega; L^4(0,T; \cL^2(U_1,V)))}$ and
on the data $A$, $B$, $\mathcal P$, $a$, $\alpha$, $C_3$.
Now, thanks to \eqref{est2}, the fact that $|\psi'|\lesssim1+\psi$ on the left-hand side yields,
by comparison in \eqref{eq:2_app},
\begin{align*}
  &\left\|\sup_{r\in[0,T_0]}\int_D|\nabla \varphi_\lambda(r)|^2 \right\|_{L^{p/2}(\Omega)}
  + \left\|\sup_{r\in[0,T_0]}|(\mu_\lambda(r))_D|\right\|_{L^{p/2}(\Omega)} 
  + \left\|\int_{Q_{T_0}}|\nabla \mu_\lambda|^2 \right\|_{L^{p/2}(\Omega)}\\
  &\lesssim1+
  \norm{\varphi_0}^2_{L^p(\Omega; V)} + \norm{\psi(\varphi_0)}_{L^{p/2}(\Omega; L^1(D))}
  +\sqrt{T_0}\|(\mu_\lambda)_D\|_{L^{p/2}(\Omega;L^\infty(0,T_0))}\,,
\end{align*}
so that, choosing first $T_0$ small enough and then 
using a classical patching argument, we infer that
\beq
  \label{est3}
  \norm{\varphi_\lambda}_{L^p(\Omega; L^\infty(0,T; V))} 
  + \norm{\nabla \mu_\lambda}_{L^p(\Omega; L^2(0,T; H))}
  +\norm{(\mu_\lambda)_D}_{L^{p/2}(\Omega; L^\infty(0,T))} \leq M\,.
\eeq
By comparison in \eqref{eq:2_app} we deduce then also that
\beq
  \label{est4}
  \norm{\mu_\lambda}_{L^{p/2}(\Omega; L^2(0,T; V))}
  +\norm{\psi'_\lambda(\varphi_\lambda)}_{L^{p/2}(\Omega; L^2(0,T; H))} \leq M\,.
\eeq

\subsection{Continuous dependence on the data}
\label{ssec:cont_dep}
Here we prove the continuous dependence contained in Theorem~\ref{th:2}.
Given $(\varphi_1,\mu_1, \sigma_1)$ and $(\varphi_2,\mu_2,\sigma_2)$ two solutions 
to \eqref{phi1}--\eqref{eq_sigma_weak} with sources $(u_1,w_1)$, $(u_2,w_2)$,
and initial data $(\varphi^1_0,\sigma^1_0)$, 
$(\varphi^2_0,\sigma^2_0)$, respectively,
we denote by $\varphi := \varphi_1 - \varphi_2$, $\mu:=\mu_1-\mu_2$, $\sigma := \sigma_1- \sigma_2$,
$u:= u_1 - u_2$, and $w:= w_1 - w_2$. 

The equation for the differences in the Cahn-Hilliard system reads 
\begin{align}
\partial_t\varphi - \Delta\mu  &= \left( \mathcal P\sigma - \alpha u \right)h(\varphi_1) +  
(\mathcal P\sigma_2 - a - \alpha u_2)\left( h(\varphi_1) -  h(\varphi_2) \right) \label{eq_phi12}\\
\mu &= - A\Delta \varphi + B(\psi'(\varphi_1) - \psi'(\varphi_2)), \label{eq_mu12}
\end{align}
with initial data given by $\varphi_0:= \varphi_0^1 - \varphi_0^2$
and $\sigma_0:= \sigma_0^1 - \sigma_0^2$.
To get the required estimate we use the same strategy as in \cite{mir-rocca-sch}.
Firstly, we integrate \eqref{eq_phi12} over $D$
\beq\label{phi_D}
\partial_t \varphi_D = \int_D \left( \calP \sigma  + \alpha u \right)h(\varphi_1)
+ \int_D \left( \calP \sigma_2 - a  - \alpha u_2\right)  ( h(\varphi_1) - h(\varphi_1) )
\eeq
then we test by $\varphi_D$ to get
\begin{equation*}
\norm{\varphi_D}_{C^0([0,t])}^2 \lesssim |(\varphi_0)_D|^2+
\| \sigma \|_{L^2(0,t;H)}^2 + \| u \|_{L^2(0,t;V^*)}^2 + \| \varphi \|_{L^2(0,t;H)}^2,
\end{equation*}
thanks to the boundedness of $h, \sigma_2, u_2$ and the Gronwall lemma.
Taking the difference between \eqref{eq_phi12} and \eqref{phi_D} 
and testing by $\cN(\varphi - \varphi_D)$ we have
\[
\langle \partial_t (\varphi - \varphi_D), \cN(\varphi - \varphi_D) \rangle_V 
= \frac{1}{2}\frac{\d }{\d t} \| \varphi - \varphi_D \|_{*}^2\,, \qquad
\langle -\Delta \mu, \cN(\varphi - \varphi_D) \rangle_V = \left( \mu, \varphi - \varphi_D \right)_H
\]
so that, integrating in time,
\begin{equation*}
\begin{split}
&\frac{1}{2} \| (\varphi - \varphi_D)(t) \|_{*}^2 + \int_{Q_t} \mu (\varphi - \varphi_D) \\
&\lesssim 
\norm{\varphi_0-(\varphi_0)_D}_{V^*}^2+
\norm{\sigma}_{L^2(0,t; H)}^2 
+\norm{u}^2_{L^2(0,t; V^*)}
+\norm{\varphi-\varphi_D}^2_{L^2(0,t; V^*)}
+\norm{\varphi}^2_{L^2(0,t; H)}\,,
\end{split}
\end{equation*}
where we used the boundedness of $h(\varphi_1),\sigma_1$, the Lipschitz continuity of $h$ and the characterization \eqref{ineq:Ny}.
We test now equation \eqref{eq_mu12} by $\varphi - \varphi_D$ and we employ (A2) to obtain
\begin{equation}\label{eq:zero_cont_dep}
\begin{split}
A\int_{Q_t}|\nabla \varphi|^2 &= \int_{Q_t} \mu(\varphi - \varphi_D) 
-B \int_{Q_t} \left( \psi'(\varphi_1) - \psi'(\varphi_2) \right)(\varphi - \varphi_D) \\
&\leq\int_{Q_t} \mu(\varphi - \varphi_D) + B \int_{Q_t} \varphi_D\left(\psi'(\varphi_1) 
- \psi'(\varphi_2) \right) + C_2 \norm{\varphi}_{L^2(0,t;H)}^2\\
&\lesssim  \int_{Q_t} \mu(\varphi - \varphi_D) 
+  \int_{Q_t}\varphi_D |\varphi |\left( 1 + |\psi''(\varphi_1)| + |\psi''(\varphi_2)| \right)
+ \norm{\varphi}_{L^2(0,t;H)}^2.
\end{split}
\end{equation}
From the previous estimates we have
\begin{equation}\label{eq:first_cont_dep}
\begin{split}
&\| (\varphi - \varphi_D)(t) \|_{V^*}^2 + | \varphi_D(t)|^2 + \int_{Q_t} | \nabla \varphi(s)|^2 \\
&\lesssim \| \varphi_0 - (\varphi_0)_D \|_{V^*}^2 + | (\varphi_0)_D|^2 
+ \int_0^t \left( \varphi_D(s) \int_D|\varphi(s)|\left(1+|\psi''(\varphi_1(s))| +|\psi''(\varphi_2(s))|\right) \right)\,\d s \\
&\qquad+ \int_0^t \left( \| (\varphi(s) - \varphi_D(s)) \|^2_{V^*}  
+ \|\sigma(s) \|_H^2  + \|u(s)\|_{V^*}^2+ \|\varphi(s)  - \varphi_D(s)\|_H^2  + |\varphi_D(s)|^2 \right) \, \d s\\
&\lesssim \| \varphi_0 - (\varphi_0)_D \|_{V^*}^2 + | (\varphi_0)_D|^2 +  \int_0^t |\varphi_D(s)|^2\int_{D} 
\left(1 +|\psi''(\varphi_1(s))|^2 + |\psi''(\varphi_2(s))|^2  \right)\,\d s\\ 
&\qquad+ \int_0^t \left( \| (\varphi(s) - \varphi_D(s)) \|^2_{V^*}  
+ \|\sigma(s) \|_H^2 + \|u(s)\|_{V^*}^2 + \varepsilon\|\nabla \varphi(s) \|_H^2 + |\varphi_D(s)|^2\right) \d s\,,
\end{split}
\end{equation}
where we used Young inequality and the compactness inequality \eqref{comp_ineq} with $\varepsilon < 1$. 
We introduce now a sequence of stopping times 
$\left(\tau^N\right)_{N \in \mathbb{N}}$ as 
\[ \tau^N := \inf \left\lbrace t \geq 0: 
\|\psi''(\varphi_1)\|_{L^2(0,t;H)}^2 + \|\psi''(\varphi_2)\|_{L^2(0,t;H)}^2  \geq N \right\rbrace \wedge T\,. \]
Notice that $\tau^N = \tau^N(\varphi_1,\varphi_2)$ depends on $\varphi_1$ and $\varphi_2$ but we will simply write $\tau^N$ not to weight too much the notation.  
For every $N \in \enne$, the application of Gronwall's lemma (recall that $\varepsilon$ can be
chosen arbitrarily small) in the time horizon $[0,\tau^N]$ gives, for $\P$-almost every $\omega\in\Omega$,
\[
\begin{split}
&\| (\varphi - \varphi_D)(t) \|_{V^*}^2 + | \varphi_D(t)|^2 + \int_0^t \| \nabla \varphi \|_H^2\, \d s \\
&\lesssim\left(1+ e^{T+N}(T+N)\right) \left(\| \varphi_0 - (\varphi_0)_D \|_{V^*}^2 
+ | (\varphi_0)_D|^2 
+ \| \sigma\|_{L^2(0,\tau^N;H)}^2  
+ \| u\|_{L^2(0,\tau^N;V^*)}^2 \right)
\end{split}
\] 
for every $t\in[0,\tau^N(\omega)]$.
Taking the power $p/2$, the supremum in time and expectation on both sides 
we deduce that 
\begin{equation}\label{eq:cont_dip_3}
\begin{split}
&\E \sup_{t \in [0,\tau^N]}\| \varphi(t) - \varphi_D(t) \|_{V^*}^{p} 
+ \E \sup_{t \in [0,\tau^N]}| \varphi_D(t)|^{p} 
+ \E \left( \int_0^{\tau^N} \| \nabla \varphi(s) \|_H^2\, \d s \right)^{p/2}\\
&\lesssim_N 
\E\| \varphi_0 - (\varphi_0)_D \|_{V^*}^{p} + \E| (\varphi_0)_D|^{p}   
+ \E \left( \int_0^{\tau^N} \| u(s)\|_{V^*}^2\, \d s \right)^{p/2}
+ \E \left( \int_0^{\tau^N} \| \sigma(s) \|_H^2\, \d s \right)^{p/2}.
\end{split}
\end{equation} 
For what concerns $\sigma$, it holds that 
\begin{equation*}
\begin{split}
\d\sigma - \Delta\sigma \d t  + c\sigma h(\varphi_1)\d t + c\sigma_2\left( h(\varphi_1) - h(\varphi_2) \right) \d t + b(\sigma - w) \d t = \left( \mathcal{H}(\sigma_1) - \mathcal{H}(\sigma_2)\right) \d W_2
\end{split}
\end{equation*}
and the same strategy as in \eqref{ito_sigma_lambda} (taking into account also the term $w$) gives
\begin{equation}\label{eq:sigma_diff}
  \begin{split}
  &\E\sup_{t\in[0,\tau^N]}\norm{\sigma(t)}_H^p + \E \left( \int_0^{\tau^N} \| \sigma(s) \|_V^2\, \d s \right)^{p/2}\\
  &\lesssim_{c,L_h, L_{\mathcal H}}
   \|\sigma_0 \|_{L^p(\Omega;H)}  
   +\E \left( \int_0^{\tau^N} \| w(s) \|_H^2\, \d s \right)^{p/2}
   +\E \left( \int_0^{\tau^N} \| \varphi(s) \|_H^2\, \d s \right)^{p/2}\,.
  \end{split}
\end{equation}
Note again that, for every $t \in [0,T]$
\begin{equation*}
\begin{split}
\norm{\varphi}^p_{L^{p}(\Omega; L^2(0,t; H))} &\lesssim 
\norm{\varphi - \varphi_D}_{L^{p}(\Omega; L^2(0,t; H))}^2 + \norm{\varphi_D}_{L^{p}(\Omega; L^2(0,t))}^2 \\
&\lesssim \norm{\varphi - \varphi_D}^2_{L^{p}(\Omega; L^2(0,t; V^*))} + \varepsilon\| \nabla \varphi\|^2_{L^{p}(\Omega; L^2(0,t; H))} + \norm{\varphi_D}_{L^{p}(\Omega; L^2(0,t))}^2\,.
\end{split}
\end{equation*}
Combining \eqref{eq:sigma_diff} and \eqref{eq:cont_dip_3}, choosing $\eps$
sufficiently small,
and applying the stochastic Gronwall's Lemma (see \cite[Lem.~29.1]{metivier}),
we end up with \eqref{cont_dep_1}.

For what concerns uniqueness, if we take the same initial conditions and set $u, w =0$ in \eqref{cont_dep_1}, a consistency result holds: for every $N \in \enne$
\begin{equation}\label{uniquness_tau_N}
\begin{split}
&\| \varphi - \varphi_D \|_{C^0([0,t]; V^*))} 
+ \| \varphi_D\| + \| \nabla \varphi \|_{L^2(0,t;H)}^2
+ \norm{\sigma}_{C^0([0,t]; H)\cap 
  L^2(0,T; V)} = 0
\end{split}
\end{equation} 
for every $t\in[0,\tau^N(\omega)]$ and for $\P$-almost
every $\omega\in\Omega$. Hence,
and the validity of \eqref{uniquness_tau_N} can be extended to 
the stochastic interval $[\![0,\tau]\!]$ where
\[\tau:= \lim_{N \to +\infty} \tau^N = \sup_{n \in \enne} \tau^N.\]
 It remains to show that $\tau = T$ $\P$-almost surely:
 by contradiction, if $\P\{\tau < T\}>0$,
then by definition of $\tau$ we would have
$\P\{\| \psi'(\varphi_1)\|_{L^2(\tau,T; H)} +
\| \psi'(\varphi_2)\|_{L^2(\tau,T; H)}= + \infty\}>0$, 
which clearly contradicts that $\psi'(\varphi_1), \psi'(\varphi_2) \in L^1(\Omega; L^2(0,T; H))$.

\subsection{Stochastic compactness and passage to the limit}
\label{ssec:stoch_comp}
In this section we prove Theorem \ref{th:1}. 
From the uniform estimates \eqref{est1}--\eqref{est4} obtained above, 
there exists a constant $M$, independent of $\lambda$, such that
\begin{equation}\label{uniform_bounds_lambda}
\begin{split}
\norm{\sigma_\lambda}_{L^p(\Omega; C^0([0,T]; H)\cap L^2(0,T; V))\cap L^\infty(\Omega\times Q)} &\leq M\,, \\
\norm{\varphi_\lambda}_{L^p(\Omega; C^0([0,T]; H)\cap L^\infty(0,T; V)\cap L^2(0,T; Z))} &\leq M\, , \\
\norm{\mu_\lambda}_{L^{p/2}(\Omega; L^2(0,T; V))}+
\norm{\nabla \mu_\lambda}_{L^p(\Omega; L^2(0,T; H))}
+\norm{(\mu_\lambda)_D}_{L^{p/2}(\Omega; L^\infty(0,T))} &\leq M\, ,\\
\norm{\psi'_\lambda(\varphi_\lambda)}_{L^{p/2}(\Omega; L^2(0,T; H))} &\le M.
\end{split}
\end{equation} 
Moreover, given $s \in (0,1/2)$, from Hypothesis (A3) and \cite[Lem.~2.1]{flan-gat} 
\[
G \cdot W_1 \in L^p(\Omega; W^{s,2}(0,T;V)) \cap L^\kappa(\Omega; W^{s,\kappa}(0,T);V^*)) 
\quad \forall \, \kappa \geq 2\,,
\]
and by comparison in \eqref{eq_phi_weak} it holds that 
\begin{equation*}
\| \varphi_\lambda \|_{L^p(\Omega; W^{s,\kappa}(0,T;V^*))} \leq M\,.
\end{equation*}
This is crucial when dealing with the method of compactness:
fixing $\kappa >\frac{1}{s}$, by \cite[Cor.~4-5]{simon} the following inclusions are compact
\begin{equation*}
\begin{split}
L^\infty(0,T;V) \cap W^{s,\kappa}(0,T;V^*) &\hookrightarrow C^0([0,T];H)\,, \\
L^2(0,T; Z) \cap W^{s,2}(0,T;V^*) &\hookrightarrow L^2(0,T;V)\,.
\end{split}
\end{equation*}
We are now in position to show that the family of laws 
$(\pi_\lambda)_{\lambda} := (\cL(\varphi_\lambda))_\lambda$ is a tight family of
probability measures on $C([0,T]; H) \cap L^2(0,T; V)$. 
Denoting $\mathcal{X}:= L^\infty(0,T;V) \cap W^{s,\kappa}(0,T;V^*) \cap L^2(0,T; Z)$, 
we know that $\mathcal{X} \hookrightarrow C([0,T]; H) \cap L^2(0,T; V)$ compactly and
\[ \| \varphi_\lambda \|_{L^p(\Omega; \mathcal{X})} \leq M\,.\] 
Hence, denoting by $B_n$ the closed ball of radius $n$ in $\mathcal X$,
we have that $B_n$ is compact in $C([0,T]; H) \cap L^2(0,T; V)$ and
the Markov inequality implies that
\[
  \sup_{\lambda>0}\pi_\lambda(B_n^c)=
  \sup_{\lambda>0}\P\{\norm{\varphi_\lambda}_{\mathcal X}^p>n^p\}
  \leq \frac1{n^p}\E\norm{\varphi_\lambda}_{\mathcal X}^p\leq \frac{M^p}{n^p}\to 0
\]
as $n\to\infty$. We deduce that for every $\varepsilon >0$
there is $n\in\enne$ sufficiently large such that the compact 
set $B_n$ of $C([0,T]; H) \cap L^2(0,T; V)$ satisfies
$\pi_\lambda(B_n) > 1-\varepsilon$ for every $\lambda>0$, and this proves the tightness of the laws 
$(\pi_\lambda)_\lambda$ on $C([0,T]; H) \cap L^2(0,T; V)$.

In order to reconstruct a process from the limit law, we need to exhibit a 
random variable $\varphi: (\Omega,\cF,\P) \to C([0,T]; H) \cap L^2(0,T; V)$ 
such that for every $\delta > 0$
\[\P\{\|\varphi_\lambda - \varphi\|_{C([0,T]; H) \cap L^2(0,T; V)} > \delta\}
\longrightarrow 0 \quad \text{ as } \lambda \to 0\,. \]
Thanks to a method of Gy\"ongy-Krylov \cite[Lem.~1.1]{gyongy-krylov}, 
this is equivalent to the following:
for any subsequences $(\varphi_{k})_k:= (\varphi_{\lambda_k})_k$ and 
$(\varphi_{j})_j := (\varphi_{\lambda_j})_j$ of the original sequence 
$(\varphi_{\lambda})_\lambda$, there exists a joint subsequence 
$(\varphi_{k_i}, \varphi_{j_i})_i$ converging in law to a probability 
measure $\nu$ on $(C([0,T]; H) \cap L^2(0,T; V))^2$ with the property
\begin{equation}\label{crit_G-K}
\nu \left(  \left\lbrace (g^1,g^2) \in (C([0,T]; H) \cap L^2(0,T; V))^2 : 
g^1 = g^2 \right\rbrace \right) = 1. 
\end{equation} 
Let us now prove the validity of \eqref{crit_G-K}.
Using the tightness of the laws $(\pi_\lambda)_\lambda$ and 
Skorokhod theorem (see \cite[Thm.~2.7]{ike-wata}) we can find a new probability space 
$(\tilde \Omega, \tilde \cF, \tilde \P)$ along with a sequence of 
random variables $(\tilde \varphi_{k_i},\tilde \varphi_{j_i}) : 
\tilde \Omega \to (C([0,T]; H) \cap L^2(0,T; V))^2$ such that 
\begin{equation}\label{conv_skorohod}
(\tilde \varphi_{k_i},\tilde \varphi_{j_i}) \longrightarrow 
(\varphi^1, \varphi^2) \qquad \text{ in } (C([0,T]; H) \cap L^2(0,T; V))^2, \quad \tilde \P\text{-a.s. }
\end{equation}
and $\mathscr{L}(\varphi_{k_i}, \varphi_{j_i}) = 
\mathscr{L}(\tilde \varphi_{k_i}, \tilde \varphi_{j_i})$, for every $i \in \erre$.
Moreover, there exists a sequence of measurable maps 
$\Upsilon_i: (\tilde \Omega, \tilde \cF) \to (\Omega, \cF)$ 
such that for every $i \in \erre$ it holds that $\P = \tilde\P\circ(\Upsilon_i)^{-1}$, 
$(\tilde \varphi_{k_i},\tilde \varphi_{j_i}):= (\varphi_{k_i}, \varphi_{j_i}) \circ \Upsilon_i$ 
and \eqref{conv_skorohod} takes place (c.f.~\cite[Thm.~1.10.4, Add.~1.10.5]{vaa-well}).

\noindent If we define $(\tilde \sigma_{k_i},\tilde \sigma_{j_i}):= (\sigma_{k_i}, \sigma_{j_i}) \circ \Upsilon_i$ 
and  $(\tilde \mu_{k_i},\tilde \mu_{j_i}):= (\mu_{k_i}, \mu_{j_i}) \circ \Upsilon_i$, 
the uniform bounds in \eqref{uniform_bounds_lambda} still holds  in the form
\begin{align*}
\norm{(\tilde \sigma_{k_i},\tilde \sigma_{j_i})}_{(L^p(\tilde\Omega; C^0([0,T]; H)\cap L^2(0,T; V))\cap L^\infty(\Omega\times Q))^2} &\leq M\,, \\
\norm{(\tilde \varphi_{k_i},\tilde \varphi_{j_i})}_{(L^p(\tilde\Omega; C^0([0,T]; H)\cap
L^\infty(0,T; V)\cap L^2(0,T; Z)))^2} &\leq M\, , \\
\norm{(\tilde \mu_{k_i},\tilde \mu_{j_i})}_{(L^{p/2}(\tilde\Omega; L^2(0,T; V)))^2}+
\norm{\nabla (\tilde \mu_{k_i},\tilde \mu_{j_i})}_{(L^p(\tilde\Omega; L^2(0,T; H)))^2}
&\leq M\, ,\\
\norm{(\psi'_\lambda(\tilde \varphi_{k_i}),\psi'_\lambda(\tilde \varphi_{j_i}))}_{(L^{p/2}(\tilde\Omega; L^2(0,T; H)))^2} &\le M\,.
\end{align*}
This guarantees that,
up to passing to subsequences, 
recalling the strong convergence \eqref{conv_skorohod} and 
employing the strong-weak closure of maximal monotone operators
as well as the weak lower semicontinuity of the norms,
\begin{align*}
(\tilde \sigma_{k_i},\tilde \sigma_{j_i}) \rightharpoonup (\sigma^1, \sigma^2) 
\qquad &\text{ in } (L^p(\tilde\Omega; L^2(0,T; V)))^2\,, \\
(\tilde \sigma_{k_i},\tilde \sigma_{j_i}) \wstarto (\sigma^1, \sigma^2) 
\qquad &\text{ in } (L^p(\tilde\Omega; L^1(0,T; H))^*\cap L^\infty(\tilde\Omega\times Q))^2\,,\\
(\tilde \varphi_{k_i},\tilde \varphi_{j_i}) \to (\varphi^1, \varphi^2) 
\qquad &\text{ in }  (L^\ell(\tilde\Omega; C^0([0,T]; H)\cap L^2(0,T; V)))^2 
\quad\forall\,\ell\in[1,p)\,, \\
(\tilde \varphi_{k_i},\tilde \varphi_{j_i}) \wto (\varphi^1, \varphi^2) 
\qquad &\text{ in }  (L^p(\tilde\Omega; L^2(0,T; Z)))^2, \\
(\tilde \mu_{k_i},\tilde \mu_{j_i}) \rightharpoonup (\mu^1, \mu^2) \qquad &\text{ in } (L^{p/2}(\tilde\Omega; L^2(0,T; V)))^2, \\
(\psi'_\lambda(\tilde \varphi_{k_i}),\psi'_\lambda(\tilde \varphi_{j_i})) \rightharpoonup (\psi'(\varphi^1), \psi'(\varphi^2)) \qquad &\text{ in } (L^{p/2}(\tilde \Omega; L^2(0,T; H)))^2,
\end{align*} 
where the limit objects belong  to the same spaces of the respective sequences.

\noindent Furthermore, for every $t \in [0,T]$ and $\tilde \P$-a.s. the following holds in $V^*$, 
\begin{align*}
\partial_t (\tilde \varphi_{k_i} - \tilde \varphi_{j_i}) - \Delta(\tilde \mu_{k_i} - \tilde \mu_{j_i})  &= \mathcal P(\tilde \sigma_{k_i} - \tilde \sigma_{j_i}) h(\tilde \varphi_{k_i}) +  (\mathcal P\tilde \sigma_{j_i} - a - \alpha u)\left( h(\tilde \varphi_{k_i}) -  h(\tilde \varphi_{j_i}) \right)\\
\tilde \mu_{k_i} - \tilde \mu_{j_i} &= - \Delta (\tilde \varphi_{k_i} - \tilde \varphi_{j_i}) + \psi'(\tilde \varphi_{k_i}) - \psi'(\tilde \varphi_{j_i}),
\end{align*}
and passing  to the limit as $i \to +\infty$ we get
\begin{align*}
\partial_t (\varphi^1 - \varphi^2) - \Delta(\mu^1 - \mu^2)  &= \mathcal P(\sigma^1 - \sigma^2) h(\varphi^1) +  (\mathcal P \sigma^1 - a - \alpha u)\left( h(\varphi^1) - h(\varphi^2) \right)\\
\mu^1 - \mu^2 &= - \Delta (\varphi^1 - \varphi^2) + \psi'( \varphi^1) - \psi'(\varphi^2).
\end{align*}
From the uniqueness of the solution proved in Section~\ref{ssec:cont_dep} 
(see \eqref{uniquness_tau_N}, with $t \in [0,T]$) we deduce that $\varphi^1(t) = \varphi^2(t)$ 
for every $t \in [0,T]$, $\tilde \P$-a.s.~and this readily implies that 
\[ \nu \left(  \left\lbrace (g^1,g^2) \in (C^0([0,T]; H) \cap L^2(0,T; V))^2 : g^1 = g^2 \right\rbrace \right) = \tilde \P (\varphi^1= \varphi^2) = 1\,,\]
and \eqref{crit_G-K} holds.
From Gy\"ongy-Krylov's criterium, it holds that 
\[ \varphi_\lambda \to \varphi \qquad  \text{ in } C^0([0,T]; H) \cap L^2(0,T; V), \quad  \text{ in } \P\text{-measure on } \Omega.\] 
Moreover, using the uniform bounds \eqref{uniform_bounds_lambda},
the lower semicontinuity of the norms and the strong-weak closure 
of maximal monotone graphs we also have that
\begin{align*}
\sigma_\lambda \wstarto \sigma \qquad &\text{ in } L^p(\Omega; L^2(0,T; V))\cap L^\infty(\Omega\times Q)\,, \\
\varphi_\lambda \to \varphi \qquad &\text{ in }  L^\ell(\Omega; C^0([0,T]; H)\cap L^2(0,T; V))
\quad\forall\,\ell\in[1,p)\,, \\
\varphi_\lambda \wto \varphi \qquad &\text{ in }  L^p(\Omega;  L^2(0,T; Z))\,,\\
\mu_\lambda \rightharpoonup \mu \qquad &\text{ in } L^{p/2}(\Omega; L^2(0,T; V))\,,\\ 
\psi'_\lambda(\varphi_\lambda) \rightharpoonup \psi'(\varphi) 
\qquad &\text{ in } L^{p/2}(\Omega; L^2(0,T; H))\,,
\end{align*}
where 
\begin{gather*}
\varphi \in L^p(\Omega; C^0([0,T];H) \cap L^\infty(0,T; V)\cap L^2(0,T; Z))\,, \\
\sigma \in L^p(\Omega; C^0([0,T]; H)\cap L^2(0,T; V))\cap L^\infty(\Omega\times Q)\,, \\
\mu \in L^{p/2}(\Omega; L^2(0,T;V))\,, \qquad
\nabla \mu \in L^p(\Omega; L^2(0,T;H)), \qquad \;\,(\mu)_D \in L^{p/2}(\Omega; L^\infty(0,T))\,, \\
\psi'(\varphi) \in L^{p/2}(\Omega; L^2(0,T;H)) \,.
\end{gather*}
Let us show a further convergence for $(\sigma_\lambda)_\lambda$.
To do it, we exploit the continuous dependence on data given in 
\eqref{eq:sigma_diff}. For every subsequences $\lambda_i, \lambda_j$, $i,j \in \enne$
and every $\ell\in[1,p)$, it holds 
\[
   0 \leq 
   \norm{\sigma_{\lambda_i} - \sigma_{\lambda_j}}_{L^\ell\left(\Omega; C^0([0,T]; H)
   \cap L^2(0,T; V)\right)} \lesssim 
   \norm{\varphi_{\lambda_i} - \varphi_{\lambda_j}}_{L^\ell(\Omega; L^2(0,T; H))} 
   \longrightarrow 0, \quad \text{ as } i \to +\infty\,.
\]
This implies that $\sigma_{\lambda} \to \sigma$ in $L^\ell\left(\Omega; C^0([0,T]; H)\cap 
L^2(0,T; V)\right)$, which allows to pass to the limit in the stochastic integrals, getting
by the Burkholder-Davis-Gundy inequality that
\[
  \mathcal{H}(\sigma_{\lambda}) \cdot W_2 \to
   \mathcal{H}(\sigma) \cdot W_2\qquad\text{in } L^\ell(\Omega; C^0([0,T]; H))
   \quad\forall\,\ell\in[1,p)\,.
\] 
Putting this information together and letting $\lambda\searrow0$ in the approximated
problem, it is now straightforward to see that 
$(\varphi,\mu,\sigma)$ is a solution to \eqref{phi1}--\eqref{eq_sigma_weak}.

Finally, the uniqueness result proved in Proposition \ref{cont_dep_1}
(see estimate \eqref{uniquness_tau_N}, with $t \in [0,T]$) implies
that $(\varphi, \mu, \sigma)$ is the unique solution to 
\eqref{eq:1}-\eqref{eq:init}. This concludes the proof of Theorem \ref{th:1}.


\section{Refined well-posedness}
\setcounter{equation}{0}
\label{sec:refined}
This section is devoted to the proof of the refined well-posedness result
contained in Theorem~\ref{th:3}.

\subsection{Refined existence}
In order to prove the additional regularity \eqref{phi_H3}--\eqref{phi_reg},
we show that under assumptions \eqref{ip_ref1}--\eqref{ip_ref3}
some additional estimates hold on the approximated solutions
$(\varphi_\lambda, \mu_\lambda, \sigma_\lambda)$ obtained in the previous section.

\noindent{\bf Fourth estimate.} Thanks to \eqref{ip_ref3} we have that
\[
  \norm{\psi'_\lambda(\varphi_\lambda)}_H^2\lesssim 1+ \norm{\varphi_\lambda}_{L^6(D)}^6\,,
\]
hence also, by the embedding $V\embed L^6(D)$,
\[
  \norm{\psi'_\lambda(\varphi_\lambda)}_{L^2(0,T; H)} \lesssim 
  1+\norm{\varphi_\lambda}^3_{L^\infty(0,T; V)}\,.
\]
Similarly, since $\nabla\psi_\lambda'(\varphi_\lambda)=\psi_\lambda''(\varphi_\lambda)\nabla\varphi_\lambda$,
we have again by \eqref{ip_ref3} that
\[
  \norm{\nabla\psi_\lambda'(\varphi_\lambda)}_H^2\lesssim
  \norm{\nabla\varphi_\lambda}_H^2 + \int_D|\varphi_\lambda|^4|\nabla\varphi_\lambda|^2
  \leq \norm{\nabla\varphi_\lambda}_H^2 + 
  \norm{\varphi_\lambda}_{L^6(D)}^4\norm{\nabla\varphi_\lambda}_{L^6(D)}^2\,,
\]
from which, thanks again to the embedding $V\embed L^6(D)$,
\[
  \norm{\nabla\psi_\lambda'(\varphi_\lambda)}_{L^2(0,T; H)} \lesssim \norm{\varphi}_{L^2(0,T; Z)}
  + \norm{\varphi_\lambda}_{L^\infty(0,T;V)}^2\norm{\varphi_\lambda}_{L^2(0,T; Z)}\,.
\]
We infer that 
\[
  \norm{\psi'_\lambda(\varphi_\lambda)}_{L^2(0,T; V)}\lesssim 
  1+ \norm{\varphi_\lambda}^3_{L^\infty(0,T; V)} + 
   \norm{\varphi_\lambda}_{L^\infty(0,T;V)}^2\norm{\varphi_\lambda}_{L^2(0,T; Z)}\,,
\]
hence from \eqref{est2}--\eqref{est3} we deduce that
\[
  \norm{\psi'_\lambda(\varphi_\lambda)}_{L^{p/3}(\Omega; L^2(0,T; V))}\leq M\,.
\]
By comparison in \eqref{eq:2_app} and estimate \eqref{est4} we infer that 
\[
  \norm{\Delta\varphi_\lambda}_{L^{p/3}(\Omega; L^2(0,T; V))}\leq M\,,
\]
which yields \eqref{phi_H3} by elliptic regularity.

\noindent{\bf Fifth estimate.} Arguing as in the {\bf Second estimate}, we easily get that
\begin{align*}
  &\norm{\varphi_\lambda(t)}_H^2 + \int_{Q_t}|\Delta\varphi_\lambda|^2
  \lesssim_{A,B,\mathcal P, a,\alpha, C_2}
  1+\norm{\varphi_0}^2_H 
  +\norm{G}^2_{L^2(0,T; \cL^2(U_1,H))}\\
&
  + \int_{Q_t}|\varphi_\lambda|^2
  +\int_0^t\left(\varphi_\lambda(s),G(s)\right)_H\,\d W_1(s)
\end{align*}
for every $t\in[0,T]$. Now, we add and subtract the term
\[
  \frac r2\int_0^t|(\varphi_\lambda(s), G(s))_H|^2\,\d s
\]
on the right-hand side, with $r>0$ to be fixed later, getting
\begin{align*}
  \norm{\varphi_\lambda(t)}_H^2 &+ \int_{Q_t}|\Delta\varphi_\lambda|^2
  \lesssim
  1+\norm{\varphi_0}^2_H 
  +\norm{G}^2_{L^2(0,T; \cL^2(U_1,H))}
  + \int_{Q_t}|\varphi_\lambda|^2 \\
  &+\left(\int_0^t\left(\varphi_\lambda(s),G(s)\right)_H\,\d W_1(s)
  -\frac r2\int_0^t|(\varphi_\lambda(s), G(s))_H|^2\,\d s\right)
  + \frac r2\int_0^t|(\varphi_\lambda(s), G(s))_H|^2\,\d s\,,
\end{align*}
where
\[
  \int_0^t|(\varphi_\lambda(s), G(s))_H|^2\,\d s\leq
  \norm{G}^2_{L^\infty(\Omega\times(0,T); \cL^2(U_1,H))}\int_{Q_t}|\varphi_\lambda|^2\,.
\]
Hence, we infer that 
\begin{align*}
  \norm{\varphi_\lambda(t)}_H^2 &+ \int_{Q_t}|\Delta\varphi_\lambda|^2
  \lesssim_{\norm{G}^2_{L^\infty(\Omega\times(0,T; \cL^2(U_1,H))}}
  1+\norm{\varphi_0}^2_H 
  + \int_{Q_t}|\varphi_\lambda|^2 \\
  &+\left(\int_0^t\left(\varphi_\lambda(s),G(s)\right)_H\,\d W_1(s)
  -\frac r2\int_0^t|(\varphi_\lambda(s), G(s))_H|^2\,\d s\right)\,,
\end{align*}
and consequently, for every $\beta\geq1$,
\begin{align*}
  &\exp\left(\beta\norm{\varphi_\lambda(t)}_H^2\right) + \exp\left(\beta\int_{Q_t}|\Delta\varphi_\lambda|^2\right)\\
  &\lesssim
  \exp\left(\beta\norm{\varphi_0}^2_H \right)\cdot
  \exp\left(\beta\int_{Q_t}|\varphi_\lambda|^2\right)\cdot
  \exp\left(\beta\int_0^t\left(\varphi_\lambda(s),G(s)\right)_H\,\d W_1(s)
  -\frac{\beta r}2\int_0^t|(\varphi_\lambda(s), G(s))_H|^2\,\d s\right)\\
  &\lesssim\exp\left(3\beta\norm{\varphi_0}^2_H \right)
  +\exp\left(3\beta\int_{Q_t}|\varphi_\lambda|^2\right)\\
  &\qquad+\exp\left(3\beta\int_0^t\left(\varphi_\lambda(s),G(s)\right)_H\,\d W_1(s)
  -\frac{3\beta r}2\int_0^t|(\varphi_\lambda(s), G(s))_H|^2\,\d s\right)
\end{align*}
for every $t\in[0,T]$, $\P$-almost surely.
Now, choosing $r:=3\beta$, it is well-know that the process
\[
  t\mapsto\exp\left(3\beta\int_0^t\left(\varphi_\lambda(s),G(s)\right)_H\,\d W_1(s)
  -\frac{9\beta^2}2\int_0^t|(\varphi_\lambda(s), G(s))_H|^2\,\d s\right)
\]
is a real positive local martingale, hence also a real supermartingale, so that 
\[
  \E\exp\left(3\beta\int_0^t\left(\varphi_\lambda(s),G(s)\right)_H\,\d W_1(s)
  -\frac{9\beta^2}2\int_0^t|(\varphi_\lambda(s), G(s))_H|^2\,\d s\right) \leq 1 
  \qquad\forall\,t\in[0,T]\,.
\]
Consequently, taking expectations and supremum in time yields, for all $T_0\in(0,T]$,
\begin{align*}
  &\sup_{t\in[0,T_0]}\E\exp\left(\beta\norm{\varphi_\lambda(t)}_H^2\right) 
  + \E\exp\left(\beta\norm{\Delta\varphi_\lambda}_{L^2(0,T_0; H)}^2\right)
  \lesssim1+\E\exp\left(3\beta\norm{\varphi_0}^2_H \right)\\
&
  +\E\exp\left(3\beta\int_{Q_{T_0}}|\varphi_\lambda|^2\right)\,.
\end{align*}
Noting that, by the Jensen inequality and Fubini's theorem,
\begin{align*}
  \E\exp\left(\frac\beta{T_0}\norm{\varphi_\lambda}^2_{L^2(0,T_0; H)}\right)
  &\leq\E\frac1{T_0}\int_0^{T_0}\exp\left(\beta\norm{\varphi_\lambda(t)}_H^2\right)\,\d t
  =\frac1{T_0}\int_0^{T_0}\E\exp\left(\beta\norm{\varphi_\lambda(t)}_H^2\right)\,\d t\\
  &\leq\sup_{t\in[0,T_0]}\E\exp\left(\beta\norm{\varphi_\lambda(t)}_H^2\right)\,,
\end{align*}
rearranging the terms yields
\begin{align*}
  \E\exp\left(\frac\beta{T_0}\norm{\varphi_\lambda}^2_{L^2(0,T_0; H)}\right)
  &+ \E\exp\left(\beta\norm{\Delta\varphi_\lambda}_{L^2(0,T_0; H)}^2\right)\\
  &\lesssim1+\E\exp\left(3\beta\norm{\varphi_0}^2_H \right)
  +\E\exp\left(3\beta\norm{\varphi_\lambda}^2_{L^2(0,T_0;H)}\right)\,.
\end{align*}
Hence, choosing $T_0\in(0,T]$ such that $\frac\beta{T_0}>3\beta$, for example $T_0=\frac16\wedge T$,
and using the Young inequality yields, for every $\eps>0$ and for $C_\eps>0$,
\begin{align*}
  \E\exp\left(6\beta\norm{\varphi_\lambda}^2_{L^2(0,T_0; H)}\right)
  &+ \E\exp\left(\beta\norm{\Delta\varphi_\lambda}_{L^2(0,T_0; H)}^2\right)\\
  &\lesssim1+\E\exp\left(3\beta\norm{\varphi_0}^2_H \right)
  +\E\exp\left(3\beta\norm{\varphi_\lambda}^2_{L^2(0,T_0;H)}\right)\\
  &\lesssim1+\E\exp\left(3\beta\norm{\varphi_0}^2_H \right)
  +\eps\E\exp\left(6\beta\norm{\varphi_\lambda}^2_{L^2(0,T_0;H)}\right) + C_\eps\,.
\end{align*}
Choosing $\eps>0$ sufficiently small, rearranging the terms and using 
a standard patching argument implies that there exists $M_\beta>0$, 
independent of $\lambda$, such that
\beq
  \label{est5}
  \sup_{t\in[0,T]}\E\left(\beta\norm{\varphi_\lambda(t)}_H^2\right)+
  \E\exp\left(\beta\norm{\varphi_\lambda}_{L^2(0,T; Z)}^2\right) \leq M_\beta\,.
\eeq
This concludes the proof of \eqref{phi_reg}.

\subsection{Refined continuous dependence}\label{ssec:ref_cont}
We prove the refined continuous dependence property \eqref{cont_dep_2}
by using the additional regularity \eqref{phi_reg}.
We use the same notation of Section~\ref{ssec:cont_dep}.
From \eqref{eq:zero_cont_dep}--\eqref{eq:first_cont_dep} and
the mean-value theorem we get that 
\beq\label{aux_ref_ineq}
\begin{split}
\| \varphi(t)& - \varphi_D(t) \|_{V^*}^2 + |\varphi_D(t)|^2 
+ \int_0^t \| \nabla \varphi(s) \|_H^2\, \d s \\
&\lesssim \| \varphi_0 - (\varphi_0)_D \|_{V^*}^2 
+ | (\varphi_0)_D|^2 +  \int_0^t |\varphi_D(s)|\int_D(\psi'(\varphi_1(s))-\psi'(\varphi_2(s))) \, \d s\\ 
&+ \int_0^t \left( \|\varphi(s) - \varphi_D(s) \|^2_{V^*}  + |\varphi_D(s)|^2\right)\,\d s
+ \|\sigma \|_{L^2(0,t;H)}^2 + \|u\|_{L^2(0,t;V^*)}^2\\
&\lesssim \| \varphi_0 - (\varphi_0)_D \|_{V^*}^2 
+ | (\varphi_0)_D|^2 +  
\int_0^t |\varphi_D(s)|\int_D\varphi(s)\int_0^1\psi''(\iota\varphi_1(s) + (1-\iota)\varphi_2(s))\,\d\iota\, \d s\\ 
&+ \int_0^t \left( \| (\varphi(s) - \varphi_D(s)) \|^2_{V^*}  + |\varphi_D(s)|^2\right)\,\d s
+ \|\sigma \|_{L^2(0,t;H)}^2 + \|u\|_{L^2(0,t;V^*)}^2 \,.
\end{split}
\eeq
Note that, by the growth assumption \eqref{ip_ref3}, for every $\iota\in[0,1]$ we have that 
\[
|\psi''(\iota\varphi_1 + (1-\iota)\varphi_2)|\lesssim
1 + |\varphi_1|^2 + |\varphi_2|^2
\]
and 
\begin{align*}
&|\nabla\psi''(\iota\varphi_1 + (1-\iota)\varphi_2)|=
|\psi'''(\iota\varphi_1 + (1-\iota)\varphi_2)(\iota\nabla\varphi_1 + (1-\iota)\nabla\varphi_2)|\\
&\lesssim(1+ |\varphi_1| + |\varphi_2|)(|\nabla\varphi_1| + |\nabla\varphi_2|)\\
&\lesssim 1+ |\varphi_1|^2 + |\varphi_2|^2 + |\nabla\varphi_1|^2 + |\nabla\varphi_2|^2\,,
\end{align*}
from which we deduce, thanks to the continuous embedding $V\embed L^6(D)$, that
\[
  \norm{\psi''(\iota\varphi_1 + (1-\iota)\varphi_2)}_V^2\lesssim
  1+ \norm{\varphi_1}_{L^4(D)}^4 + \norm{\varphi_2}_{L^4(D)}^4
  +\norm{\nabla\varphi_1}_{L^4(D)}^4 + \norm{\nabla\varphi_2}_{L^4(D)}^4
  \lesssim 1 + \norm{\varphi_1}_Z^4 + \norm{\varphi_2}_Z^4\,.
\]
Hence, we infer that, for every $s\in[0,T]$,
\[
\norm{\int_0^1\psi''(\iota\varphi_1(s) + (1-\iota)\varphi_2(s))\,\d\iota}_V\lesssim
1 + \norm{\varphi_1(s)}_Z^2 + \norm{\varphi_2(s)}_Z^2\,.
\]
Hence, substituting in \eqref{aux_ref_ineq} we have 
\[
\begin{split}
\| \varphi(t)& - \varphi_D(t) \|_{V^*}^2 + |\varphi_D(t)|^2 
+ \int_0^t \| \nabla \varphi(s) \|_H^2\, \d s \\
&\lesssim \| \varphi_0 - (\varphi_0)_D \|_{V^*}^2 
+ | (\varphi_0)_D|^2 +  
\int_0^t \norm{\varphi(s)}_{V^*}^2\left(1+ \norm{\varphi_1(s)}_Z^2 + \norm{\varphi_2(s)}_Z^2\right)\,\d s\\ 
&+ \int_0^t \left( \| (\varphi(s) - \varphi_D(s)) \|^2_{V^*}  + |\varphi_D(s)|^2\right)\,\d s
+ \|\sigma \|_{L^2(0,t;H)}^2 + \|u\|_{L^2(0,t;V^*)}^2 \,.
\end{split}
\]

Now, since \eqref{ip_H0} is in order, we have that 
\[
\partial_t\sigma - \Delta\sigma + c\sigma h(\varphi_1) + c\sigma_2\left( h(\varphi_1) - h(\varphi_2) \right)  
+ b(\sigma - w)  = 0\,, \qquad \sigma(0)=\sigma_0\,,
\]
so that it is not difficult to prove that, for every $\eps>0$ and a certain $C_\eps>0$
\beq\label{sigm_phi_cont}
\begin{split}
  \norm{\sigma}^2_{C^0([0,t]; H)\cap 
  L^2(0,t; V))} &\lesssim_{c,L_h,}
   \|\sigma_0 \|_{H}^2
  + \norm{\varphi}_{L^2(0,t; H)}^2 + \| w\|_{L^2(0,t;H)}^2\\
  &\leq\|\sigma_0 \|_{H}^2
  + \eps\norm{\nabla \varphi}_{L^2(0,t; H)}^2
  +C_\eps\norm{\varphi}^2_{L^2(0,t; V^*)} + \| w\|_{L^2(0,t;H)}^2
  \qquad\P\text{-a.s.}
  \end{split}
\eeq
Substituting in the estimate above and choosing $\eps$ small enough we get that 
\[
\begin{split}
&\| \varphi(t) - \varphi_D(t) \|_{V^*}^2 + | \varphi_D(t)|^2 
+ \int_0^t \| \nabla \varphi(s) \|_H^2\, \d s \\
&\lesssim \| \varphi_0-(\varphi_0)_D\|_{V^*}^2 +|(\varphi_0)_D|^2
+\|\sigma_0 \|_{H}^2
+ \|u\|_{L^2(0,T;V^*)}^2 + \| w\|_{L^2(0,T;H)}^2\\
&\quad+  \int_0^t \norm{\varphi(s)}_{V^*}^2
\left(1+ \norm{\varphi_1(s)}_Z^2 + \norm{\varphi_2(s)}_Z^2\right) \, \d s
+ \int_0^t \| (\varphi(s) - \varphi_D(s)) \|^2_{V^*}\,\d s
\end{split}
\]
for every $t\in[0,T]$, $\P$-almost surely. The Gronwall lemma implies then that 
\beq\label{cont_dep_aux2}
  \begin{split}
  \norm{\varphi}^2_{C^0([0,T]; V^*)\cap L^2(0,T; V)}
  &\lesssim
  \left(\| \varphi_0 \|_{V^*}^2 
  +\|\sigma_0 \|_{H}^2
  + \|u\|_{L^2(0,T;V^*)}^2 + \| w\|_{L^2(0,T;H)}^2\right)\\
  &\times\exp\left(\norm{\varphi_1}_{L^2(0,T; Z)}^2
  +\norm{\varphi_2}_{L^2(0,T; Z)}^2\right)
  \end{split}
\eeq
so that, taking $p/2$-power, expectations,
applying the H\"older inequality on the right-hand side with exponents
$\frac{q}p>1$ and $\frac{q/p}{(q/p) - 1}$, and recalling condition \eqref{phi_reg}
holds with the choice $\beta=\frac{q/p}{(q/p) - 1}$, we infer that 
\[
  \norm{\varphi}^p_{L^p(\Omega;C^0([0,T]; V^*)\cap L^2(0,T; V))}
  \lesssim
  \| \varphi_0 \|_{L^q(\Omega;V^*)}^p
  +\|\sigma_0 \|_{L^q(\Omega; H)}^p
  + \|u\|_{L^q(\Omega; L^2(0,T;V^*))}^p + \| w\|_{L^q(\Omega; L^2(0,T;H))}^p\,,
\]
so that \eqref{cont_dep_2} follows from
\eqref{sigm_phi_cont}.

Now, testing \eqref{eq_phi12} by $\varphi$, \eqref{eq_mu12} by $-\Delta\varphi$
and taking the difference, yields
\begin{align*}
  &\frac12\norm{\varphi(t)}_H^2 + \int_{Q_t}|\Delta\varphi|^2 
  -\int_{Q_t}\left(\psi'(\varphi_1) - \psi'(\varphi_2)\right)\Delta\varphi\\
  &=\frac12\norm{\varphi_0}_H^2 + \int_{Q_t}
  \left( \mathcal P\sigma - \alpha u \right)h(\varphi_1)\varphi +  
 \int_{Q_t}(\mathcal P\sigma_2 - a - \alpha u_2)\left( h(\varphi_1) -  h(\varphi_2) \right)\varphi
 \end{align*}
for every $t\in[0,T]$.
Using the boundedness of
$\sigma_2$, $u_2$, $h$ and the Lipschitz-continuity of $h$ we get
\[
  \norm{\varphi}_{C^0([0,T]; H)\cap L^2(0,T; Z)}^2
  \lesssim\norm{\varphi_0}_H^2 + \norm{\varphi}_{L^2(0,T; H)}^2 +
 \norm{\sigma}^2_{L^2(0,T; H)} + \norm{u}^2_{L^2(0,T; H)}
 + \int_{Q}\left(1+|\varphi_1|^4 + |\varphi_2|^4\right)|\varphi|^2\,.
\]
Noting now that, by the H\"older inequality and the embedding $V\embed L^6(D)$,
\begin{align*}
  \int_{Q}\left(1+|\varphi_1|^4 + |\varphi_2|^4\right)|\varphi|^2&\lesssim
  \int_0^T\norm{\varphi(s)}^2_{L^6(D)}\left(1+\norm{\varphi_1(s)}_{L^6(D)}^4 +
  \norm{\varphi_2(s)}^4_{L^6(D)}\right)\,\d s\\
  &\lesssim
  \norm{\varphi}^2_{L^2(0,T; V)}\left(1+\norm{\varphi_1}^4_{L^\infty(0,T; V)}
  +\norm{\varphi_2}^4_{L^\infty(0,T; V)}\right)\,,
\end{align*}
taking power $p/2$ and using \eqref{cont_dep_aux2} again to the power $p/2$ we infer that 
\begin{align*}
  \norm{\varphi}_{C^0([0,T]; H)\cap L^2(0,T; Z)}^p
  &\lesssim\left(
  \norm{\varphi_0}_H^p +
  \|\sigma_0 \|_{H}^p
  + \|u\|_{L^2(0,T;H)}^p + \| w\|_{L^2(0,T;H)}^p\right)\\
  &\times
  \left(1+\norm{\varphi_1}^{2p}_{L^\infty(0,T; V)}
  +\norm{\varphi_2}^{2p}_{L^\infty(0,T; V)}\right)\\
  &\times\exp\left(\frac{p}2\norm{\varphi_1}_{L^2(0,T;Z)}^2 + 
  \frac{p}2\norm{\varphi_2}_{L^2(0,T;Z)}^2\right)\,.
\end{align*}
Now, note that by \eqref{ip_ref4} it is easy to check that
\[
  1-\frac1{\beta_0}:=\frac{p}{q} + \frac{2p}{r} <1\,,
\]
hence
we can take expectations and use the H\"older inequality on the right-hand side
with exponents $q/p$, $r/(2p)$ and $\beta_0$, respectively,
getting
\begin{align*}
  \norm{\varphi}_{L^p(\Omega;C^0([0,T]; H)\cap L^2(0,T; Z))}^p &\lesssim
  \left(\| \varphi_0 \|_{L^q(\Omega;V^*)}^p
  +\|\sigma_0 \|_{L^q(\Omega; H)}^p
  + \|u\|_{L^q(\Omega; L^2(0,T;H))}^p + \| w\|_{L^q(\Omega; L^2(0,T;H))}^p\right)\\
  &\times
  \left(1+\norm{\varphi_1}^{2p}_{L^r(\Omega;L^\infty(0,T; V))}
  +\norm{\varphi_2}^{2p}_{L^r(\Omega;L^\infty(0,T; V))}\right)\\
  &\times\norm{\exp\left(\frac{\beta_0p}2\norm{\varphi_1}_{L^2(0,T;Z)}^2
  + \frac{\beta_0p}2\norm{\varphi_2}_{L^2(0,T;Z)}^2\right)}_{L^1(\Omega)}^{1/{\beta_0}}\,.
\end{align*}
Noting that
$\varphi_1,\varphi_2\in L^r(\Omega; L^\infty(0,T;V))$ by \eqref{ip_ref4}, \eqref{phi1}, and Theorem~\ref{th:1},
recalling also \eqref{phi_reg} the last two factors are finite, and
we can conclude the proof of Theorem~\ref{th:3}.


\section{Optimal control problem}
\label{sec:opt}

This section is devoted to the analysis of the optimal control problem associated
to the state system \eqref{eq:1}--\eqref{eq:init} and the cost functional $J$.

\subsection{Existence of an optimal control}
We prove here Theorem~\ref{th:4}, showing that a relaxed optimal control always exists.

Let $(u_n,w_n)_n\subset\mathcal U$ be a minimizing sequence for $\tilde J$ in $\mathcal U$, i.e.~such that
\[
  \lim_{n\to\infty}\tilde J(u_n,w_n)=\inf_{(v,z)\in\mathcal U}\tilde J(v,z)\,.
\]
By definition of $\mathcal U$ the sequence $(u_n,w_n)_n$ is uniformly bounded in 
$L^\infty(\Omega\times Q)$.
In particular, if we denote by $L^2_w(Q)$ the space 
$L^2(Q)$ equipped with its weak topology, 
it is immediate to see that
the sequence of laws of $(u_n,w_n)$ is tight on $L^2_w(Q)^2$.
Furthermore,
for every $n\in\enne$, let 
$(\varphi_n, \mu_n, \sigma_n)_n$ be the corresponding solution to \eqref{phi1}--\eqref{eq_sigma_weak}
with respect to the data $(\varphi_0,\sigma_0,u_n,w_n)$:
recalling the proof of Theorem~\ref{th:1} (see section \ref{sec:exist}),
we know that there exists a positive constant $M$, 
depending only on the initial data $(\varphi_0,\sigma_0)$, but not on $n$, such that 
\begin{align*}
  \norm{\varphi_n}_{L^p(\Omega; W^{s,\kappa}(0,T; V^*)\cap 
  C^0([0,T]; H)\cap L^\infty(0,T; V)\cap L^2(0,T; Z))} &\leq M\,,\\
  \norm{\mu_n}_{L^{p/2}(\Omega; L^2(0,T; H))}+
  \norm{\nabla\mu_n}_{L^p(\Omega; L^2(0,T; H))} +
  \norm{(\mu_n)_D}_{L^{p/2}(\Omega; L^\infty(0,T))}&\leq M\,,\\
  \norm{\psi'(\varphi_n)}_{L^{p/2}(\Omega; L^2(0,T; H))}&\leq M\,,\\
  \norm{\sigma_n}_{L^p(\Omega; H^1(0,T; V^*)\cap L^2(0,T; V))\cap L^\infty(\Omega\times Q)} &\leq M\,,
\end{align*}
where $s\in(0,1/2)$ and $\kappa>1/s$ are fixed.
In particular, the sequence of laws of $(\varphi_n)_n$ and $(\sigma_n)_n$ are tight on the spaces 
$L^2(0,T; V)\cap C^0([0,T]; H)$ and $L^2(0,T; H)$, respectively
(see again section \ref{sec:exist}).

Hence, the sequence 
$(W_1, G\cdot W_1,\varphi_0, \sigma_0, u_n, w_n, \varphi_n, \mu_n, \sigma_n, \varphi_Q, \varphi_T)_n$
is tight on the product space
\[
  C^0([0,T]; U_1)\times C^0([0,T]; H)
  \times V \times H \times L^2_w(Q)^2 \times C^0([0,T]; H) \times L^2_w(0,T; V) \times
  L^2(0,T; H) \times L^2(0,T; H)\times V\,.
\]
Recalling that $L^2_w(Q)$ and $L^2_w(0,T; V)$ are a sub-Polish spaces,
by Jakubowski-Skorokhod theorem (see e.~g.~\cite[Thm.~2.7.1]{fei-hof}) 
there is a probability space $(\Omega',\cF',\P')$
and a sequence of measurable mappings $\phi_n:(\Omega',\cF')\to(\Omega,\cF)$
such that $\P:=\P'\circ\phi_n^{-1}$ for every $n\in\enne$ and 
\begin{align*}
  (W_{1,n}', I_n'):=(W_1, G\cdot W_1)\circ \phi_n \to (W_1', I') \qquad&\text{in } 
  C^0([0,T]; U_1)\times C^0([0,T]; H)\,,\\
  (\varphi_{0,n},\sigma_{0,n})':=
  (\varphi_0,\sigma_0)\circ\phi_n \to (\varphi_0',\sigma_0') \qquad&\text{in } V\times H\,,\\
  (u_n', w_n'):=(u_n,w_n)\circ\phi_n \wto (u',w') \qquad&\text{in } L^2(0,T; H)^2\,,\\
  (\varphi_n', \sigma_n'):=(\varphi_n,\sigma_n)\circ\phi_n \to (\varphi',\sigma') 
  \qquad&\text{in } C^0([0,T]; H)\times L^2(0,T; H)\,,\\
  \mu_n':=\mu_n\circ\phi_n \wto \mu' \qquad&\text{in } L^2(0,T; V)\,,\\
  (\varphi_{Q,n}', \varphi_{T,n}'):=(\varphi_Q, \varphi_T)\circ\phi_n \to (\varphi_Q', \varphi_T')
  \qquad&\text{in } L^2(0,T; H) \times V\,,
\end{align*}
$\P'$-almost surely on $\Omega'$. Since $(W_1, G\cdot W_1, 
\varphi_0, \sigma_0, \varphi_Q, \varphi_T)_n$
is constant, it follows immediately that the law of 
$(W_1', I', \varphi_0', \sigma_0', \varphi_Q', \varphi_T')$
coincides with the law of $(W_1, \varphi_0, \sigma_0, \varphi_Q, \varphi_T)$. Moreover, 
by weak lower semicontinuity we also have that $0\leq u',w'\leq 1$ almost everywhere 
in $\Omega'\times Q$. Finally, 
using a classical procedure based on martingale representation theorems
(for a detailed argument the reader can refer to \cite[\S~4]{vall-zimm}), 
it is possible to show that $W_{1,n}'$ is a $(\cF'_{n,t})_t$-cylindrical Wiener process in $U_1$ and
$W_1'$ is a $(\cF'_t)_t$-cylindrical Wiener process in $U_1$, where
\[
  \cF'_{n,t}:=\sigma(W_{1,n}'(s))_{s\in[0,t]}\,, \qquad
  \cF'_t:=\sigma(W_1'(s), I'(s), \varphi(s), \sigma(s))_{s\in[0,t]}\,,
\]
and that $I'=G\cdot W'$. Since the uniform estimates on
$(\varphi_n, \mu_n, \sigma_n, \psi'(\varphi_n))_n$ are also satisfied by
$(\varphi_n', \mu_n', \sigma_n', \psi'(\varphi_n'))_n$, 
passing to the weak limit as $n\to\infty$ in the variational formulation
of the problem on $(\Omega',\cF', \P')$, by the strong-weak closure
of maximal monotone operators it follows that 
$(\varphi', \mu', \sigma')$ is the unique solution to 
\eqref{phi1}--\eqref{eq_sigma_weak} on $\Omega'$
with respect to $(\varphi_0', \sigma_0', u',w')$.
Consequently, by weak lower 
semicontinuity, the properties of $(\phi_n)_n$, and
the definition of minimizing sequence, we have
\begin{align*}
  \tilde J'(u',w')
  &\leq\liminf_{n\to\infty}
  \frac{\beta_1}{2}\E{}'\int_Q|\varphi_n'-\varphi_{Q,n}'|^2
  +\frac{\beta_2}{2}\E{}'\int_D|\varphi'_n(T)-\varphi_{T,n}'|^2
  +\frac{\beta_3}{2}\E{}'\int_D(\varphi_n'(T) + 1)\\
  &\qquad+\frac{\beta_4}{2}\E{}'\int_Q|u_n'|^2
  +\frac{\beta_5}{2}\E{}'\int_Q|w_n'|^2\\
  &=\liminf_{n\to\infty}J(\varphi_n,u_n,w_n)=
  \liminf_{n\to\infty}\tilde J(u_n,w_n)=\inf_{(v,z)\in\mathcal U}\tilde J(v,z)\,,
\end{align*}
so that $(u',w')$ is a relaxed optimal control.

\subsection{The linearized system}
In this section we prove Theorems~\ref{th:5}--\ref{th:6}.
First of all, we show that uniqueness of solution holds for 
the linearized system \eqref{xy}--\eqref{lin3}.
Secondly, we prove Theorem~\ref{th:6} and (hence) existence of solutions
for the linearized system.

\noindent{\bf Uniqueness.} Let us show that the linearized system \eqref{xy}--\eqref{lin3}
admits a unique solution. Let $(x_k^i,y_k^i, z_k^i)$ solve \eqref{xy}--\eqref{lin3}
for $i=1,2$. Then we have 
\begin{align*}
  \partial_t (x_k^1-x_k^2) - \Delta (y_k^1-y_k^2) = 
  h(\varphi)\mathcal P(z_k^1-z_k^2) + h'(\varphi)(x_k^1-x_k^2)(\mathcal P\sigma -a - \alpha u)
   &\qquad\text{in } (0,T)\times D\,,\\
  y_k^1-y_k^2=-A\Delta (x_k^1-x_k^2) + B\psi''(\varphi)(x_k^1-x_k^2) &\qquad\text{in } (0,T)\times D\,,\\
  \partial_t (z_k^1-z_k^2) - \Delta (z_k^1-z_k^2) + c(z_k^1-z_k^2) h(\varphi) 
  + c\sigma h'(\varphi)(x_k^1-x_k^2)
  +b(z_k^1-z_k^2)=0 &\qquad\text{in } (0,T)\times D\,,\\
  \partial_{\bf n}(x_k^1-x_k^2) = \partial_{\bf n}(z_k^1-z_k^2) = 0 &\qquad\text{in } (0,T)\times \partial D\,,\\
  (x_k^1-x_k^2)(0)=(z_k^1-z_k^2)(0)=0 &\qquad\text{in } D\,.
\end{align*}
Testing the first equation by $\frac1{|D|}$, using the boundedness of $h$, $h'$ and $\sigma$, we 
deduce that there exists $M>0$ such that
\[
  \norm{(x_k^1-x_k^2)_D}_{C^0([0,t])}^2 \leq M\left(\norm{z_k^1-z_k^2}_{L^1(Q_t)}^2
  +\norm{x_k^1-x_k^2}_{L^1(Q_t)}^2\right) \qquad \forall\,t\in[0,T]\,,\quad\P\text{-a.s.}
\]
Testing the first equation by $\mathcal N((x_k^1-x_k^2)-(x_k^1-x_k^2)_D)$, the second one by 
$x_k^1-x_k^2$ and taking the difference yields
\begin{align*}
&\frac12\norm{(x_k^1-x_k^2-(x_k^1-x_k^2)_D)(t)}_{V^*}^2 + A\int_{Q_t}|\nabla (x_k^1-x_k^2)|^2 
+B \int_{Q_t}\psi''(\varphi)|(x_k^1-x_k^2)|^2\\
&\qquad=B\int_{Q_t}\psi''(\varphi)(x_k^1-x_k^2)(x_k^1-x_k^2)_D\\
&\qquad+
\int_{Q_t}\left[h(\varphi)\mathcal P(z_k^1-z_k^2)
+h'(\varphi)(x_k^1-x_k^2)(\mathcal P\sigma -a - \alpha u)\right]\mathcal N(x_k^1-x_k^2-(x_k^1-x_k^2)_D)\,.
\end{align*}
Summing the two inequalities and recalling that $\psi''\geq-C_2$ we infer then that
\begin{align*}
  &\norm{(x_k^1-x_k^2)(t)}_{V^*}^2 + \int_{Q_t}|\nabla(x_k^1-x_k^2)|^2\lesssim_{M,C_2,\mathcal P}
  \int_{Q_t}|x_k^1-x_k^2|^2 + \int_{Q_t}|z_k^1-z_k^2|^2\\
  &\qquad+\int_0^t(x_k^1-x_k^2)_D(s)\norm{(x_k^1-x_k^2)(s)}_{V^*}\norm{\psi''(\varphi(s))}_{V}\,\d s
  +\int_0^t\norm{(x_k^1-x_k^2)(s)}_{V^*}^2\,\d s\,,
\end{align*}
where a direct computation based on \eqref{ip_ref3}, the
embedding $V\embed L^6(D)$ and the Young inequality yields 
(as already performed in Section~\ref{ssec:ref_cont})
\[
  \norm{\psi''(\varphi)}_{V}\lesssim 1 + \norm{\varphi}_Z^2\,.
\]
Testing the third equation by $z_k^1-z_k^2$ it follows easily by the Gronwall lemma and
the boundedness of $h$, $h'$ and $\sigma$ that 
\[
  \norm{z_k^1-z_k^2}_{C^0([0,t]; H)\cap L^2(0,t; V)}
   \leq M \norm{x_k^1-x_k^2}_{L^2(0,t; H)}\,.
\]
Hence, substituting in the previous inequality and using a compactness inequality in the form
\[
  \norm{x_k^1-x_k^2}_{L^2(0,t; H)}^2\leq \eps\norm{\nabla(x_k^1-x_k^2)}^2_{L^2(0,t; H)}
  +C_\eps\norm{x_k^1-x_k^2}^2_{L^2(0,t; V^*)}\,,
\]
choosing $\eps$ sufficiently small and rearranging the terms we have 
\[
  \norm{(x_k^1-x_k^2)(t)}_{V^*}^2 + \int_{Q_t}|\nabla(x_k^1-x_k^2)|^2\lesssim
  \int_0^t\left(1+\norm{\varphi(s)}_{Z}^2\right)\norm{(x_k^1-x_k^2)(s)}^2_{V^*}\,\d s\,,
\]
yielding $x_k^1(t)=x_k^2(t)$ for every $t\in[0,T]$ thanks to the Gronwall lemma
and recalling $\varphi\in L^2(0,T; Z)$.
It follows as a consequence the uniqueness $y_k^1=y_k^2$ and $z_k^1=z_k^2$.

\noindent{\bf G\^ateaux-differentiability and existence.} Let $(u,w), (k_u,k_w)\in\tilde{\mathcal U}$
and let us set $\varphi:=\mathcal S(u,w)$. 
Let now $\eps\in(-\eps_0,\eps_0)$ where $\eps_0>0$ is chosen sufficiently small so that 
$(u,w)+\eps(k_u,k_w)\in\tilde{\mathcal U}$ for all $\eps\in(-\eps_0,\eps_0)$ (note that 
this is possibly since $\tilde{\mathcal U}$ is an open subset of $L^2(\Omega\times Q)$).
Defining $\varphi_\eps:=\mathcal S((u,w)+\eps(k_u,k_w))$, we have then that 
\begin{gather*}
  \partial_t \left(\frac{\varphi_\eps-\varphi}{\eps}\right) - \Delta\left(\frac{\mu_\eps-\mu}{\eps}\right)=
  \frac{h(\varphi_\eps)-h(\varphi)}{\eps}(\mathcal P\sigma_\eps - a - u-\eps k_u) 
   +h(\varphi)\left(\mathcal P\frac{\sigma_\eps-\sigma}{\eps} - k_u\right)\,,\\
   \frac{\mu_\eps-\mu}{\eps}=-A\Delta\left(\frac{\varphi_\eps-\varphi}{\eps}\right)
   +B\frac{\psi'(\varphi_\eps)-\psi'(\varphi)}{\eps}\,,\\
   \partial_t\left(\frac{\sigma_\eps-\sigma}{\eps}\right)-\Delta\left(\frac{\sigma_\eps-\sigma}{\eps}\right)
   +c\frac{\sigma_\eps-\sigma}\eps h(\varphi_\eps)+c\frac{h(\varphi_\eps)-h(\varphi)}{\eps}\sigma
   +b\left(\frac{\sigma_\eps-\sigma}{\eps}-k_w\right)=0\,,
\end{gather*}
where $\frac{\varphi_\eps-\varphi}{\eps}(0)=\frac{\sigma_\eps-\sigma}{\eps}(0)=0$.
Now, from the continuous dependence property \eqref{cont_dep_3}, we have 
\begin{align*}
  &\| \varphi_\eps  - \varphi\|_{L^{p}(\Omega; C^0([0,T]; H)\cap L^2(0,T; Z))} 
  + \norm{\sigma_\eps-\sigma}_{L^p\left(\Omega; C^0([0,T]; H)\cap 
  L^2(0,T; V)\right)}\\
  &\qquad\leq M\eps\Bigl(\|k_u\|_{L^q(\Omega; L^2(0,T;H))} 
  +\norm{k_w}_{L^q(\Omega; L^2(0,T; H))}\Bigr)\,,
\end{align*}
so that we deduce the uniform estimate (updating $M$)
\beq\label{est_lin1}
  \norm{\frac{\varphi_\eps-\varphi}{\eps}}_{L^{p}(\Omega; C^0([0,T]; H)\cap L^2(0,T; Z))}
  +\norm{\frac{\sigma_\eps-\sigma}{\eps}}_{L^p\left(\Omega; C^0([0,T]; H)\cap 
  L^2(0,T; V)\right)}\leq M.
\eeq
Moreover, thanks to the growth condition \eqref{ip_ref3}, the H\"older inequality and
the embedding $V\embed L^6(D)$, we also have 
\begin{align*}
  &\int_Q\left|\frac{\psi'(\varphi_\eps)-\psi'(\varphi)}{\eps}\right|^2=
  \int_Q\int_0^1|\psi''(\varphi + \tau(\varphi_\eps-\varphi))|^2\left|\frac{\varphi_\eps-\varphi}\eps\right|^2\,d\tau
  \lesssim\int_Q(1+|\varphi|^4+|\varphi_\eps|^4)\left|\frac{\varphi_\eps-\varphi}\eps\right|^2\\
  &\leq\int_0^T\left(1+\norm{\varphi(s)}_{L^6(D)}^4 + \norm{\varphi_\eps(s)}^4_{L^6(D)}\right)
  \norm{\frac{\varphi_\eps-\varphi}\eps(s)}^2_{L^6(D)}\,\d s\\
  &\lesssim\left(1+\norm{\varphi}_{L^\infty(0,T; V)}^4 + \norm{\varphi_\eps}^4_{L^\infty(0,T; V)}\right)
  \norm{\frac{\varphi_\eps-\varphi}\eps}^2_{L^2(0,T; V)}\,,
\end{align*}
yielding
\[
\norm{\frac{\psi'(\varphi_\eps)-\psi'(\varphi)}{\eps}}_{L^2(0,T; H)}\lesssim
\left(1+\norm{\varphi}_{L^\infty(0,T; V)}^2 + \norm{\varphi_\eps}^2_{L^\infty(0,T; V)}\right)
\norm{\frac{\varphi_\eps-\varphi}\eps}_{L^2(0,T; V)}\,,
\]
where by \eqref{cont_dep_aux2} we have that 
\[
  \norm{\frac{\varphi_\eps-\varphi}\eps}_{L^2(0,T; V)}\lesssim
  \left(\norm{k_u}_{L^2(0,T; V^*)} + \norm{k_w}_{L^2(0,T; H)}\right)
  \exp\left(\norm{\varphi}_{L^2(0,T; Z)}^2 + \norm{\varphi_\eps}_{L^2(0,T; Z)}^2\right)\,.
\]
Now, thanks to \eqref{ip_ref4} we have that 
$\norm{\varphi}_{L^\infty(0,T; V)}^2 + \norm{\varphi_\eps}^2_{L^\infty(0,T; V)}$
is uniformly bounded (w.r.t.~$\eps$) in $L^{r/2}(\Omega)$, where $\frac{r}2>\frac{pq}{q-p}>p$.
Moreover, $\norm{k_u}_{L^2(0,T; V^*)} + \norm{k_w}_{L^2(0,T; H)}\in L^q(\Omega)$
by definition of $\tilde{\mathcal U}$, where $q>p$ by assumption,
and by \eqref{ip_ref1}--\eqref{ip_ref3} we also have
\[
  \norm{\exp\left(\norm{\varphi}^2_{L^2(0,T; Z)} 
  + \norm{\varphi_\eps}_{L^2(0,T; Z)}^2\right)}_{L^\beta(\Omega)} \leq M_\beta
  \qquad\forall\,\beta>1\,.
\]
In particular, noting that 
\[
  \frac1q + \frac2r <\frac{1}{q} + \frac{q-p}{pq} = \frac1p < 1\,,
\]
choosing $\frac1\beta:=\frac1q + \frac2r$ in the estimates above 
and using the H\"older inequality yields
\beq\label{est_lin2}
\norm{\frac{\psi'(\varphi_\eps)-\psi'(\varphi)}{\eps}}_{L^p(\Omega; L^2(0,T; H))} \leq M\,,
\eeq
hence also, by comparison in the equations, 
\beq\label{est_lin3}
\norm{\partial_t\left(\frac{\varphi_\eps-\varphi}{\eps}\right)}_{L^p(\Omega; L^2(0,T; Z^*))}+
\norm{\frac{\mu_\eps-\mu}{\eps}}_{L^p(\Omega; L^2(0,T; H))}
+\norm{\partial_t\left(\frac{\sigma_\eps-\sigma}{\eps}\right)}_{L^p(\Omega; L^2(0,T; V^*))}\leq M\,.
\eeq
By the uniform estimates \eqref{est_lin1}--\eqref{est_lin3} we deduce that 
there are 
\begin{gather*}
  x_k\in L^p(\Omega; H^1(0,T; Z^*)\cap L^2(0,T; Z))\,,\\
  y_k\in L^p(\Omega; L^2(0,T; H))\,, \qquad
  z_k\in L^p(\Omega; H^1(0,T; V^*)\cap L^2(0,T; V))
\end{gather*}
such that, as $\eps\to0$,
\begin{align*}
  \frac{\varphi_\eps-\varphi}{\eps} \wto x_k \qquad&\text{in } L^p(\Omega; H^1(0,T; Z^*)\cap L^2(0,T; Z))\,,\\
  \frac{\mu_\eps-\mu}{\eps} \wto y_k \qquad&\text{in } L^p(\Omega; L^2(0,T; H))\,,\\
  \frac{\sigma_\eps-\sigma}{\eps} \wto z_k \qquad&\text{in } L^p(\Omega; H^1(0,T; V^*)\cap L^2(0,T; V))
\end{align*}
and
\begin{align*}
  \varphi_\eps\to\varphi \qquad&\text{in } L^p(\Omega; C^0([0,T]; H)\cap L^2(0,T; Z))\,,\\
  \sigma_\eps\to\sigma \qquad&\text{in } L^p(\Omega; C^0([0,T]; H)\cap L^2(0,T; V))\,.
\end{align*}
Since we have the compact inclusions
\[
H^1(0,T; Z^*)\cap L^2(0,T; Z)\cembed L^2(0,T, V)\,, \qquad
H^1(0,T; V^*)\cap L^2(0,T; V)\cembed L^2(0,T; H)\,,
\]
by Skorokhod theorem there exists a probability space $(\Omega',\cF',\P')$
and measurable mappings 
$$\phi_\eps:(\Omega',\cF')\to(\Omega,\cF)$$
with $\P=\P'\circ\phi_\varepsilon^{-1}$ such that 
\begin{gather*}
  \varphi\circ\phi_\eps \to \varphi' \quad\text{in } L^2(0,T; V)\,,\\
  \frac{\varphi_\eps-\varphi}{\eps}\circ\phi_\eps \to x_k' \quad\text{in } L^2(0,T; V)\,,\qquad
  \frac{\sigma_\eps-\sigma}{\eps}\circ\phi_\eps \to z_k' \quad\text{in } L^2(0,T; H)\,,
\end{gather*}
$\P'$-almost surely, and
\begin{align*}
  \frac{\varphi_\eps-\varphi}{\eps}\circ\phi_\eps
   \wto x'_k \qquad&\text{in } L^p(\Omega'; H^1(0,T; Z^*)\cap L^2(0,T; Z))\,,\\
  \frac{\mu_\eps-\mu}{\eps}\circ\phi_\eps \wto y'_k \qquad&\text{in } L^p(\Omega'; L^2(0,T; H))\,,\\
  \frac{\sigma_\eps-\sigma}{\eps}\circ\phi_\eps
   \wto z'_k \qquad&\text{in } L^p(\Omega'; H^1(0,T; V^*)\cap L^2(0,T; V))
\end{align*}
for some $\varphi'$, $x_k'$, $y_k'$, and $z_k'$.
Moreover, up to extracting a subsequence, the continuity of $\psi''$ guarantees that
\[
\frac{\psi'(\varphi_\eps)-\psi'(\varphi)}\eps\circ\phi_\eps=
\int_0^1\psi''(\varphi\circ\phi_\eps+\tau(\varphi_\eps-\varphi)\circ\phi_\eps)
\frac{\varphi_\eps-\varphi}{\eps}\circ\phi_\eps\,\d\tau\to \psi''(\varphi')x'_k
\]
a.e. in $\Omega\times Q$. 
Recalling then that the left-hand side is
uniformly bounded in $L^p(\Omega'; L^2(0,T; H))$ by \eqref{est_lin2}, we deduce also the convergence of the whole sequence  
\[
  \frac{\psi'(\varphi_\eps)-\psi'(\varphi)}\eps\circ\phi_\eps\wto \psi''(\varphi')x'_k
  \qquad\text{in } L^p(\Omega'; L^2(0,T; H))\,.
\]
Similarly, by the Lipschitz-continuity of $h$ it is immediate to show that
\[
  \frac{h(\varphi_\eps)-h(\varphi)}\eps\circ\phi_\eps\to h'(\varphi')x'_k
  \qquad\text{in } L^\ell(\Omega'; L^2(0,T; H)) \quad\forall\,\ell\in[1,p)\,.
\]
Passing then to the weak limit in the
variational formulation of the equations on $\Omega'$ we deduce that 
$(x_k',y_k',z_k')$ solves the linearized system \eqref{xy}--\eqref{lin3} on
$(\Omega',\cF',\P')$ with respect to $\varphi'$. Since e have already proved uniqueness
for such system, the well-known results by Gy\"ongy and Krylov \cite[Lem~1.1.]{gyongy-krylov}
ensures that the strong convergences hold in the original probability space $(\Omega,\cF,\P)$, i.e.~that
\[
\frac{\varphi_\eps-\varphi}{\eps} \to x_k \quad\text{in } L^2(0,T; V)\,,\qquad
  \frac{\sigma_\eps-\sigma}{\eps}\to z_k \quad\text{in } L^2(0,T; H)
\]
$\P$-almost surely in $\Omega$. Hence, repeating the 
same argument on $(\Omega,\cF,\P)$, we have that 
$(x_k,y_k,z_k)$ is the unique solution to the linearized system in the sense 
of \eqref{xy}--\eqref{lin3}. 
This completes the proof of existence of Theorem~\ref{th:5}.

Finally, as a consequence of estimates \eqref{est_lin1} and \eqref{est_lin3}, we also have that 
\[
  \norm{x_k}_{L^p(\Omega; H^1(0,T; Z^*)\cap L^2(0,T; Z))}\leq 
  M\norm{(k_u,k_w)}_{L^q(\Omega; L^2(0,T; H))^2}\,,
\]
so that the map $k\mapsto x_k$ is linear and continuous 
from $L^q(\Omega; L^2(0,T; H))^2$ to $L^p(\Omega; H^1(0,T; Z^*)\cap L^2(0,T; Z))$,
and the G\^ateaux-differentiability of Theorem~\ref{th:6} is also proved.

\subsection{The adjoint system}
We prove here existence (and uniqueness) of solutions for the adjoint system.
We firstly introduce a suitable approximation of the system so that classical variational theory for Backward SPDEs can be applied.
Then we derive uniform estimates on the solution by using a duality argument and pass to the limit exploiting the linear character of the equations.
As regards uniqueness we use again a duality relation.

\noindent{\bf The approximated problem.}
For every $n\in\enne$, let
\[
  \psi''_n:\erre\to\erre\,, \qquad
  \psi''_n(r):=\begin{cases}
  n \quad&\text{if } \psi''(r)>n\,,\\
  \psi''(r) \quad&\text{if } |\psi''(r)|\leq n\,,\\
  -n \quad&\text{if } \psi''(r)<-n\,,
  \end{cases}
  \quad r\in\erre\,.
\]
We consider the approximated problem
\begin{equation*}
\begin{split}
  -\d\pi_n - A\Delta \tilde \pi_n\,\d t + B\psi''_n(\varphi)\tilde \pi_n\,\d t &= 
  h'(\varphi)(\mathcal P\sigma - a - \alpha u)\pi_n\,\d t
  - ch'(\varphi)\sigma \rho_n\,\d t 
  +\beta_1(\varphi-\varphi_Q)\,	d t-\xi_n\,\d W_1\,,\\
  \tilde \pi_n &=-\Delta \pi_n\,,\\
  -d\rho_n -\Delta \rho_n\,\d t &+ ch(\varphi)\rho_n\,\d t + b\rho_n\,\d t=\mathcal P h(\varphi)\pi_n\,\d t - \theta_n\,	\d W_2\,,\\
  \pi_n(T)&=\beta_2(\varphi(T)-\varphi_T)+\frac{\beta_3}2\,, \quad
  \rho(T)=0\,.
  \end{split}
\end{equation*}
From the boundedness of $\psi''_n(\varphi), h', \sigma$ and $u$ and the linear character of the system, we can infer existence and uniqueness of a variational solution due to the classical theory for backward SPDEs (see e.g. \cite[\S~3]{du-meng2}). 
By rewriting the system as a unique equation in the corresponding product spaces (or by using a fixed point technique) it can be easily shown that that 
\begin{gather*}
    \pi_n \in L^2\left(\Omega; C^0([0,T];H)\cap L^2(0,T; Z)\right)\,,\\
    \tilde \pi_n \in L^2\left(\Omega; C^0([0,T]; Z^*)\cap L^2(0,T; H)\right)\,,\\
    \rho_n \in L^2\left(\Omega; C^0([0,T];  H)\cap L^2(0,T; V)\right)\,,\\
    \xi_n \in L^2(\Omega; L^2(0,T; \cL^2(U_1,H)))\,, \qquad
    \theta_n\in L^2(\Omega; L^2(0,T; \cL^2(U_2,H))).
  \end{gather*}
Furthermore, by assumption on $\varphi_Q$ and $\varphi_T$, we have $\beta_1(\varphi-\varphi_Q)\in L^6(\Omega; L^2(0,T; H))$ and
$\beta_2(\varphi(T)-\varphi_T)\in L^6(\Omega,\cF_T; V)$
(recall that $p\geq6$).
Hence, 
by computing It\^o formula for $\| \nabla \pi_n\|^2_H$ and subsequently derive $L^6_\Omega$-estimates, it can be shown that the variational solution $(\pi_n,\tilde\pi_n,\xi_n,\rho_n,\theta_n)$ given by \eqref{ad1}--\eqref{ad3} (where $\psi''$ is replaced by $\psi''_n$), is actually more regular:   
\begin{gather*}
    \pi_n \in L^6\left(\Omega; C^0([0,T];V)\cap L^2(0,T; Z\cap H^3(D))\right)\,,\\
    \tilde \pi_n \in L^6\left(\Omega; C^0([0,T]; V^*)\cap L^2(0,T; V)\right)\,,\\
    \rho_n \in L^6\left(\Omega; C^0([0,T];  H)\cap L^2(0,T; V)\right)\,,\\
    \xi_n \in L^6(\Omega; L^2(0,T; \cL^2(U_1,V)))\,, \qquad
    \theta_n\in L^6(\Omega; L^2(0,T; \cL^2(U_2,H))).
  \end{gather*}
We omit here the details not to weight to much the readability of the paper, and we refer to \cite[Lem.~4.2]{fuhr-tes}= for what concerns $L^p_\Omega$ estimates on backward SPDEs and to \cite{Scarpa} for the improved regularity in space.

In order to compute uniform estimates 
on the approximated solutions to the adjoint problem, we need some auxiliary results.
First of all, we show that the corresponding approximated linearized system
is well-posed in a more general setting, where the forcing terms
in the equations are represented by an arbitrary term $\gamma:=(\gamma_1,\gamma_2)$: 
 \begin{align*}
  \partial_t x_n^\gamma - \Delta y_n^\gamma= 
  h(\varphi)\mathcal P z_n^\gamma + h'(\varphi)x_n^\gamma
  (\mathcal P\sigma -a - \alpha u) +\gamma_1
   &\qquad\text{in } (0,T)\times D\,,\\
  y_n^\gamma=-A\Delta x_n^\gamma + B\psi''_n(\varphi)x_n^\gamma &\qquad\text{in } (0,T)\times D\,,\\
  \partial_t z_n^\gamma - \Delta z_n^\gamma + c z_n^\gamma h(\varphi) 
  + c\sigma h'(\varphi)x_n^\gamma
  +bz_n^\gamma=\gamma_2 &\qquad\text{in } (0,T)\times D\,,\\
  \partial_{\bf n}x_n^\gamma = \partial_{\bf n} z_n^\gamma = 0 &\qquad\text{in } (0,T)\times \partial D\,,\\
  x_n^\gamma (0)=z_n^\gamma(0)=0 &\qquad\text{in } D\,.
\end{align*}
Then we introduce the linear map $\tau: \gamma \mapsto (x_n^\gamma, x_n^\gamma(T))$ assigning to the forcing terms the solution and the solution at final time $t = T$ of the first equation.
Notice that a more general map can be studied, also involving $z_n^\gamma$ and allowing for stochastic perturbation in the linearized system (this is not necessary in our situation due to the additive character of the noise in \eqref{eq:1}).
By carefully choosing the functional spaces, the map $\tau$ turns out to be continuous along with its adjoint: 
$\tau^*: (f,\zeta) \to (\pi_n^\gamma,\rho_n^\gamma)$.
Observe that the dual operator $\tau^*$ maps the forcing term $f$ and the final condition $\zeta$ to the corresponding solution $(\pi_n^\gamma,\rho_n^\gamma)$ of the backward equation.   
Hence,  the solution $(\pi_n,\rho_n)$ we are interested in, can be obtained by evaluating $\tau^*$ at $f = \beta_1(\varphi - \varphi_Q)$ and $\zeta = \beta_2(\varphi(T)-\varphi_T)+\frac{\beta_3}{2}$.

Let us start by showing the continuity of the map $\tau$ and the subsequent duality formula.
\begin{lem}
  \label{lin_gamma}
  Assume {\em(A1)--(A7)}, \eqref{ip_ref1}--\eqref{ip_ref3}, \eqref{ip_H0} and \eqref{ip_ref4}.
  Let $(u,w)\in\tilde{\mathcal U}$ and set $\varphi:=\mathcal S(u,w)$. Then
  for every $\gamma:=(\gamma_1,\gamma_2)\in L^{6/5}(\Omega; L^1(0,T; H))^2$ and for every $n\in\enne$
  there exists a unique triple $(x_n^{\gamma},y_n^{\gamma},z_n^{\gamma})$ with
  \begin{gather*}
    x_n^{\gamma} \in L^{1}(\Omega; H^1(0,T; Z^*)\cap L^2(0,T; Z))\cap 
    L^{6/5}(\Omega; C^0([0,T]; V^*)\cap L^2(0,T; V))\,, \\
    y_n^{\gamma}\in L^{1}(\Omega; L^2(0,T; H))\,,\\
    z_n^{\gamma}\in L^{6/5}(\Omega; H^1(0,T; V^*)\cap L^2(0,T; V))\,,
  \end{gather*}
  such that 
 \begin{gather*}
  \ip{\partial_t x_{k}}{\zeta}_V -\int_Dy_{k}\Delta\zeta = 
  \int_D\left[h(\varphi)\mathcal Pz_k + h'(\varphi)x_k(\mathcal P\sigma - a -\alpha u) + \gamma_1\right]\zeta\,,\\
  \int_Dy_k\zeta=A\int_D\nabla x_k\cdot \nabla\zeta +B \int_D\psi''(\varphi)x_k\zeta\,,\\
  \ip{\partial_t z_k}{\zeta}_V + \int_D\nabla z_k\cdot\nabla\zeta
  +\int_D\left[cz_k h(\varphi) + c\sigma h'(\varphi)x_k + bz_k + \gamma_2\right]\zeta=0
  \end{gather*}
   for every $\zeta\in Z$, for almost every $t\in(0,T)$, $\P$-almost surely.  
Moreover, there exists a positive constant $M>0$, independent of $\gamma$ and $n$, such that 
\begin{align}
  \label{aux_x}
  \norm{x_n^\gamma}_{L^{6/5}(\Omega; C^0([0,T]; V^*)\cap L^2(0,T; V))}&\leq M
  \left(\norm{\gamma_1}_{L^{6/5}(\Omega; L^1(0,T; V^*))}+
  \norm{\gamma_2}_{L^{6/5}(\Omega; L^1(0,T; H))}\right)\,,\\
  \label{aux_x'}
  \norm{x_n^\gamma}_{L^1(\Omega; C^0([0,T]; H)\cap L^2(0,T; Z))}&\leq M
  \left(\norm{\gamma_1}_{L^{6/5}(\Omega; L^1(0,T; H))}+
  \norm{\gamma_2}_{L^{6/5}(\Omega;L^1(0,T; H))}\right)\,.
\end{align}
Finally, it holds that 
  \beq\label{eq:duality}
  \E\int_Q\pi_n\gamma_1 + \E\int_Q\rho_n\gamma_2 = \beta_1\E\int_Q(\varphi-\varphi_Q)x^\gamma_n +
   \E\int_D\left(\beta_2(\varphi(T)-\varphi_T)+\frac{\beta_3}{2}\right)x^\gamma_n(T)\,.
  \eeq
\end{lem}
\begin{proof}
  The existence and uniqueness of $(x_n^\gamma, y_n^\gamma, z_n^\gamma)$ follows 
  from the fact that $\psi''_n(\varphi)\in L^\infty(\Omega\times Q)$ and $\gamma\in L^{6/5}(\Omega; L^1(0,T; H))^2$.
  Let us show the two estimates. 
  We integrate the first equation on $D$ and test it by $(x_n^k)_D$,
then we also test the first equation by 
  $\mathcal N(x_n^\gamma-(x_n^\gamma)_D)$, 
  the second by $x_n^\gamma-(x_n^\gamma)_D$,
  the third by $z_n^\gamma$. 
 Summing up all the contributions we obtain
  \begin{align*}
  &\frac12|(x_n^\gamma)_D(t)|^2 + \frac12\norm{(x_n^\gamma-(x_n^\gamma)_D)(t)}^2_{V^*}
  +\frac12\norm{z_n(t)}_H^2 + A\int_{Q_t}|\nabla x_n^\gamma|^2
  +\int_{Q_t}|\nabla z_n^\gamma|^2\\
  &\leq-B\int_{Q_t}\psi''_n(\varphi)x_n^\gamma(x_n^\gamma-(x_n^\gamma)_D)
  +\int_{Q_t}\big(h(\varphi)\mathcal Pz_n^\gamma + h'(\varphi)x_n^\gamma
  (\mathcal P\sigma -a - \alpha u)+\gamma_1\big)_D(x_n^\gamma)_D\\
  &+\int_{Q_t}\left[h(\varphi)\mathcal Pz_n^\gamma + h'(\varphi)x_n^\gamma
  (\mathcal P\sigma -a - \alpha u)+\gamma_1\right]\mathcal N(x_n^\gamma-(x_n^\gamma)_D)
  +\int_{Q_t}\left[\gamma_2-c\sigma h'(\varphi)x_n^\gamma\right]z_n^\gamma\,,
  \end{align*}
  which yields then by the assumptions on $\psi$, the Young and H\"older inequalities, and
  the boundedness of $h$, $h'$, $\sigma$ and $u$,
  \begin{align*}
  &|(x_n^\gamma)_D(t)|^2 + \norm{(x_n^\gamma-(x_n^\gamma)_D)(t)}^2_{V^*}
  +\norm{z_n(t)}_H^2 +\int_{Q_t}|\nabla x_n^\gamma|^2 +\int_{Q_t}|\nabla z_n^\gamma|^2\\
  &\lesssim_{A,B, C_2}
  \int_0^t\left(|(x_n^\gamma)_D(s)|^2+\norm{x_n^\gamma(s)}^2_H+\norm{z_n^\gamma(s)}^2_H\right)\,	\d s\\
  &\qquad+\int_0^t\norm{\gamma_1(s)}_{V^*}\left(\norm{(x_n^\gamma-(x_n^\gamma)_D)(s)}_{V^*}+
  |(x_n^\gamma)_D(s)|^2\right)\,\d s
  +\int_0^t\norm{\gamma_2(s)}_H\norm{z_n^\gamma(s)}_H\,\d s\\
  &\leq\delta\int_{Q_t}|\nabla x_n^\gamma|^2 
  + C_\delta\int_0^t\left(\norm{(x_n^\gamma-(x_n^\gamma)_D)(s)}_{V^*}^2 
  +|(x_n^\gamma)_D(s)|^2
  +\norm{z_n^\gamma(s)}_H^2\right)\,\d s\\
  &\qquad+
  \int_0^t\left(\norm{\gamma_1(s)}_{V^*}\norm{x_n^\gamma(s)}_{V^*}+
  \norm{\gamma_2(s)}_H\norm{z_n^\gamma(s)}_H\right)\,\d s
  \end{align*}
  for every $\delta >0$. Choosing the $\delta>0$ sufficiently small, rearranging the terms
  and applying the Gronwall lemma 
  in the version \cite[Lem.~A4--A5]{brezis} we infer that 
  \[
  \norm{x_n^\gamma}_{L^\infty(0,T; V^*)\cap L^2(0,T; V)}
  +\norm{z_n^\gamma}_{L^\infty(0,T; H)\cap L^2(0,T; V)}\leq M\left(
  \norm{\gamma_1}_{L^1(0,T; V^*)}+\norm{\gamma_2}_{L^1(0,T; H)}\right)
  \qquad\P\text{-a.s.}
  \]
  for a certain $M>0$ independent of $n$ and $\gamma$, from which \eqref{aux_x} follows.\\
  Now, we test the first equation by $x_n^\gamma$, the second by $-\Delta x_n^\gamma$ and take 
  the difference, getting
  \begin{align*}
  \frac12\norm{x_n^\gamma(t)}_H^2+A\int_{Q_t}|\Delta x_n^\gamma|^2&=
  B\int_{Q_t}\psi''_n(\varphi)x_n^\gamma\Delta x_n^\gamma\\
  &+\int_{Q_t}\left[h(\varphi)\mathcal Pz_n^\gamma + h'(\varphi)x_n^\gamma
  (\mathcal P\sigma -a - \alpha u)+\gamma_1\right]x_n^\gamma\,.
  \end{align*}
  The Young and H\"older inequalities together with the growth assumption on $\psi$, 
  the continuous embedding $V\embed L^6(D)$ and
  the boundedness of $h$, $\sigma$ and $u$ yield then
  \begin{align*}
  \norm{x_n^\gamma(t)}_H^2 &+ \int_{Q_t}|\Delta x_n^\gamma|^2\lesssim \norm{z_n^\gamma}^2_{L^2(0,T; H)}
  +\int_{Q_t}|\psi''(\varphi)|^2|x_n^\gamma|^2+
  \int_{Q_t}|x_n^\gamma|^2 + \int_0^t\norm{\gamma_1(s)}_H\norm{x_n^\gamma(s)}_H\,\d s\\
  &\lesssim\norm{z_n^\gamma}^2_{L^2(0,T; H)} + 
  \norm{\varphi}^4_{L^\infty(0,T; V)}\norm{x_n^\gamma}^2_{L^2(0,T; V)}
  +\int_{Q_t}|x_n^\gamma|^2 + \int_0^t\norm{\gamma_1(s)}_H\norm{x_n^\gamma(s)}_H\,\d s\,.
  \end{align*}
  The Gronwall lemma and the fact that 
  $\norm{z_n^\gamma}_{L^2(0,T; H)}\lesssim\norm{\gamma}_{L^1(0,T; H)^2}$
  imply again that 
  \[
  \norm{x_n^\gamma}_{L^\infty(0,T; H)\cap L^2(0,T; Z)}\leq M
  \left(
  \norm{\varphi}^2_{L^\infty(0,T; V)}\norm{x_n^\gamma}_{L^2(0,T; V)}
  +\norm{\gamma}_{L^1(0,T; H)^2}\right)\,.
  \]
  Now, note that since $q>p\geq6$ and $\varphi\in L^r(\Omega; L^\infty(0,T; V))$
  with $r>\frac{2pq}{q-p}>12$, we have in particular that $\norm{\varphi}^2_{L^\infty(0,T; V)}\in L^6(\Omega)$:
  hence, taking expectations in the last inequality and using the 
  the H\"older inequality and \eqref{aux_x}, we deduce also \eqref{aux_x'}.\\
  Finally, in order to prove the duality relation \eqref{eq:duality} 
  we test the equation for $x_n^\gamma$ by $\pi_n$, the equation for $z_n^\gamma$ by $\rho_n$,
  and we subtract the equation for $\pi_n$ tested by $x_n^\gamma$ and the equation for $\rho_n$
  tested by $z_n^\gamma$. Integrating from $0$ to $T$ and taking expectations, 
  using the initial conditions for $(x_n^\gamma, z_n^\gamma)$ and the final conditions 
  for $(\pi_n, \rho_n)$, the duality 
  relations follows from usual computations involving integration by parts.
\end{proof}

We are now ready to show uniform estimates on the approximated solutions
to the adjoint problem. The main tool we have at our disposal is the duality relation \eqref{eq:duality}.

\noindent{\bf First estimate.} For every $\gamma\in L^{6/5}(\Omega; L^1(0,T; H))^2$, the duality 
relation \eqref{eq:duality}
implies that 
\begin{align*}
\E\int_Q\pi_n\gamma_1 + \E\int_Q\rho_n\gamma_2&\leq
\norm{\beta_1(\varphi-\varphi_Q)}_{L^6(\Omega; L^2(0,T; H))}
\norm{x_n^\gamma}_{L^{6/5}(\Omega;L^2(0,T; H))}\\
&+\norm{\beta_2(\varphi(T)-\varphi_T)+\frac{\beta_3}2}_{L^6(\Omega; V)}
\norm{x_n^\gamma(T)}_{L^{6/5}(\Omega; V^*)}\,.
\end{align*}
Since $\varphi\in L^p(\Omega; L^\infty(0,T; V)\cap C^0([0,T]; H))$ and $p\geq6$, 
by Lemma~\ref{lin_gamma} we infer that 
\[
\E\int_Q\pi_n\gamma_1 + \E\int_Q\rho_n\gamma_2\leq M
\left(\norm{\gamma_1}_{L^{6/5}(\Omega; L^1(0,T; V^*))}+
  \norm{\gamma_2}_{L^{6/5}(\Omega; L^1(0,T; H))}\right)
\]
for a positive constant $M$ independent of $\gamma$ and $n$. 
Since $H$ is dense in $V^*$, taking 
supremum over $\gamma\in L^{6/5}(\Omega; L^1(0,T; H))^2$ such that 
$\norm{\gamma_1}_{L^{6/5}(\Omega; L^1(0,T; V^*))}\leq 1$ and
$\norm{\gamma_2}_{L^{6/5}(\Omega; L^1(0,T; H))}\leq 1$,
we deduce that 
for every $\ell\in[1,+\infty)$
\[
  \norm{\pi_n}_{L^6(\Omega; L^\ell(0,T; V))}+\norm{\rho_n}_{L^6(\Omega; L^\ell(0,T; H))}\leq M_\ell\,,
\] 
so that in particular
\beq
  \label{ad_est1}
  \norm{\pi_n}_{L^6(\Omega\times(0,T); V)}+\norm{\rho_n}_{L^6(\Omega\times(0,T); H)}\leq M\,.
\eeq

\noindent{\bf Second estimate.} We write It\^o's formula for the sum $\frac12\norm{\pi_n}_V^2+
\frac12\norm{\rho_n}_H^2$ (it is crucial not to do it seprately). 
By noting that $\frac{1}{2} D\|\pi_n\|_V^2 = \pi_n + \Delta \pi_n = \pi_n + \tilde \pi_n$, we have
\begin{align*}
  &\frac12\norm{\pi_n(t)}_V^2+\frac12\norm{\rho_n(t)}_H^2 + A\int_t^T\left(
  \norm{\nabla\tilde\pi_n(s)}_H^2+\norm{\tilde\pi_n(s)}_H^2\right)\,\d s
  +\int_t^T\norm{\nabla\rho_n(s)}_H^2\,\d s\\
  &\qquad+\int_t^T\!\!\int_D(ch(\varphi(s))+b)|\rho_n(s)|^2
  +\frac12\int_t^T\norm{\xi_n(s)}^2_{\cL^2(U_1,V)}\,\d s
  +\frac12\int_t^T\norm{\theta_n(s)}^2_{\cL^2(U_2,H)}\,\d s\\
  &=\frac12\norm{\beta_2(\varphi(T)-\varphi_T)}_V^2
  +\int_t^T\!\!\int_D\beta_1(\varphi-\varphi_Q)(s)\left(\tilde\pi_n(s)+\pi_n(s)\right)\,\d s
  +\int_t^T\!\!\int_D\mathcal P h(\varphi(s))\pi_n(s)\rho_n(s)\,\d s\\
  &\qquad+\int_t^T\!\!\int_D\left[h'(\varphi(s))(\mathcal P\sigma(s) - a - \alpha u(s))\pi_n(s)
  - ch'(\varphi(s))\sigma(s) \rho_n(s)\right]\left(\tilde\pi_n(s)+\pi_n(s)\right)\,\d s\\
  &\qquad-B\int_t^T\!\!\int_D\psi''_n(\varphi(s))|\tilde\pi_n(s)|^2\,\d s -
  B\int_t^T\!\!\int_D\psi''_n(\varphi(s))\tilde\pi_n(s)\pi_n(s)\,\d s\\
  &\qquad-\int_t^T\left(\xi_n(s), \tilde\pi_n(s)+\pi_n(s)\right)_H\,\d W_1(s)
  -\int_t^T\left(\theta_n(s), \rho_n(s)\right)_H\,\d W_2(s)\,.
\end{align*}
Taking expectations, using the boundedness of $h$, $\sigma$ and $u$
together with the Young inequality,
the first four terms on the right-hand side are estimated by 
\begin{align*}
  &\norm{\beta_2(\varphi(T)-\varphi_T)}_{L^2(\Omega; V)}^2 
  +C_\eps \norm{\beta_1(\varphi-\varphi_Q)}^2_{L^2(\Omega; L^2(0,T; H))} 
  +\eps\E\int_t^T\norm{\tilde\pi_n(s)}_H^2\,\d s\\
  &\qquad+C_\eps\norm{\pi_n}^2_{L^2(\Omega; L^2(0,T; H))} + 
  C_\eps\norm{\rho_n}^2_{L^2(\Omega; L^2(0,T; H))}
\end{align*}
for every $\eps>0$.
Secondly, note that by definition of $\tilde\pi_n$ and the fact that $\psi''\geq-C_2$,
\begin{align*}
-B\E\int_t^T\!\!\int_D\psi''_n(\varphi(s))|\tilde\pi_n(s)|^2\,\d s&\leq
BC_2\int_t^T \int_D|\tilde\pi_n|^2\leq
\eps\E\int_t^T\norm{\nabla\tilde\pi_n(s)}_H^2\,\d s + C_\eps\E\int_t^T\norm{\tilde\pi_n(s)}_{V^*}^2\,\d s\\
&\leq\eps\E\int_t^T\norm{\nabla\tilde\pi_n(s)}_H^2\,\d s + C_\eps\E\int_t^T\norm{\nabla\pi_n(s)}_{H}^2\,\d s.
\end{align*}
Finally, by the H\"older inequality, the growth assumption on $\psi''$
and the continuous inclusion $V\embed L^6(D)$ we have
\begin{align*}
 &-B\E\int_t^T\!\!\int_D\psi''_n(\varphi(s))\tilde\pi_n(s)\pi_n(s)\,\d s\lesssim
 \E\int_t^T\!\!\int_D\left(1+|\varphi(s)|^2\right)|\tilde\pi_n(s)||\pi_n(s)|\\
 &\qquad\leq\eps\E\int_t^T\norm{\tilde\pi_n(s)}_H^2\,\d s +
 C_\eps\E\int_t^T\!\!\int_D\left(1+|\varphi(s)|^4\right)|\pi_n(s)|^2\\
 &\qquad\leq\eps\E\int_t^T\norm{\tilde\pi_n(s)}_H^2\,\d s +
 C_\eps\norm{\varphi}^6_{L^6(\Omega\times(0,T); V)} + 
 C_\eps\norm{\pi_n}^6_{L^6(\Omega\times(0,T); V)}\,.
\end{align*}
Taking into account the positivity of $ch(\varphi)+b$, rearranging the terms 
and choosing $\eps$ sufficiently small, the estimate \eqref{ad_est1} yields
\begin{align}
\label{ad_est2}
  \norm{\pi_n}_{L^\infty(0,T; L^2(\Omega; V))} + \norm{\tilde\pi_n}_{L^2(\Omega; L^2(0,T; V))}
  +\norm{\rho_n}_{L^\infty(0,T; L^2(\Omega; H))\cap L^2(\Omega; L^2(0,T; V))}&\leq M\,,\\
  \label{ad_est3}
  \norm{\xi_n}_{L^2(\Omega; L^2(0,T; \cL^2(U_1,V)))}+ 
  \norm{\theta_n}_{L^2(\Omega; L^2(0,T; \cL^2(U_2,H)))}&\leq M
\end{align}
for a positive constant $M$ independent of $n$.

Going back now to It\^o's formula, taking first supremum in time and then expectations,
using \eqref{ad_est3} with Burkholder-Davis-Gundy inequality,
a classical procedure yields also by elliptic regularity
\begin{align}
  \label{ad_est4}
  \norm{\pi_n}_{L^2(\Omega; L^\infty(0,T; V)\cap L^2(0,T; Z\cap H^3(D)))} 
  +\norm{\tilde\pi_n}_{L^2(\Omega; L^\infty(0,T; V^*)\cap L^2(0,T; V))}&\leq M\,,\\
  \label{ad_est5}
  \norm{\rho_n}_{L^2(\Omega; L^\infty(0,T; H)\cap L^2(0,T; V))}&\leq M\,.
\end{align}

\noindent{\bf Passage to the limit.}
From \eqref{ad_est4}--\eqref{ad_est5} we obtain that there exist
\begin{gather*}
  \pi \in L^\infty(0,T; L^2(\Omega; V))\cap L^2(\Omega; L^2(0,T; Z\cap H^3(D)))\,,\\
  \tilde\pi\in L^\infty(0,T; L^2(\Omega; V^*))\cap L^2(\Omega; L^2(0,T; V))\,, \\
  \rho\in L^\infty(0,T; L^2(\Omega; H))\cap L^2(\Omega; L^2(0,T; V))\,,\\
  \xi\in L^2(\Omega; L^2(0,T; \cL^2(U_1,V)))\,, \qquad
  \theta\in L^2(\Omega; L^2(0,T; \cL^2(U_2,H)))\,,
\end{gather*}
such that 
\begin{align*}
  \pi_n \wstarto \pi \qquad&\text{in } L^\infty(0,T; L^2(\Omega; V))\cap L^2(\Omega; L^2(0,T; Z\cap H^3(D)))\,,\\
  \tilde\pi_n \wstarto \tilde \pi \qquad&\text{in }
  L^\infty(0,T; L^2(\Omega; V^*))\cap L^2(\Omega; L^2(0,T; V))\,,\\
  \rho_n \wstarto\rho \qquad&\text{in }
  L^\infty(0,T; L^2(\Omega; H))\cap L^2(\Omega; L^2(0,T; V))\,,\\
  \xi_n\wto\xi \qquad&\text{in } L^2(\Omega; L^2(0,T;\cL^2(U_1,V)))\,,\\
  \theta_n\wto\theta \qquad&\text{in } L^2(\Omega; L^2(0,T;\cL^2(U_2,H)))\,.
\end{align*}
Notice that, in order to extract the above (weakly converging) sequences, we modified on purpose the functional setting so to guarantee reflexivity of the functional spaces in consideration. 
The weak convergences are enough to pass to the limit in each term 
of the variational formulation of the approximated problem. Let us show 
in detail only the term involving $\psi''_n$. First of all, since
$\psi''(\varphi)\in L^{p/2}(\Omega; L^\infty(0,T; L^3(D)))$ and $p\geq6$, 
by the properties of the truncation operator we have 
\[
  \psi''_n(\varphi)\to \psi''(\varphi) \qquad\text{in } L^3(\Omega\times(0,T)\times D)\,.
\]
Hence, by strong-weak convergence we infer that 
\[
  \psi''_n(\varphi)\tilde\pi_n \wto \psi''(\varphi)\tilde \pi \qquad\text{in } L^{6/5}(\Omega\times(0,T)\times D)\,.
\]
By linearity of the approximated adjoint system, letting $n\to\infty$ we obtain exactly 
conditions \eqref{ad1}--\eqref{ad3}. The further 
regularity in \eqref{pi}--\eqref{xi-sig} is recovered a posteriori 
in the limit equation by It\^o's formula.
This completes the proof of existence 
for the adjoint system.

\noindent{\bf Uniqueness.}
For every $\gamma \in L^{6/5}(\Omega; L^1(0,T; H))^2$,
using the estimates \eqref{aux_x}--\eqref{aux_x'}
and arguing as in the proof of Lemma~\ref{lin_gamma},
it is straightforward to 
prove the existence of $(x^\gamma, y^\gamma, z^\gamma)$ such that 
  \begin{gather*}
    x^{\gamma} \in L^{1}(\Omega; H^1(0,T; Z^*)\cap L^2(0,T; Z))\cap 
    L^{6/5}(\Omega; C^0([0,T]; V^*)\cap L^2(0,T; V))\,, \\
    y^{\gamma}\in L^{1}(\Omega; L^2(0,T; H))\,,\\
    z^{\gamma}\in L^{6/5}(\Omega; H^1(0,T; V^*)\cap L^2(0,T; V))\,,
  \end{gather*}
  such that 
  \begin{align*}
  \partial_t x^\gamma - \Delta y^\gamma= 
  h(\varphi)\mathcal P z^\gamma + h'(\varphi)x^\gamma
  (\mathcal P\sigma -a - \alpha u) +\gamma_1
   &\qquad\text{in } (0,T)\times D\,,\\
  y^\gamma=-A\Delta x^\gamma + B\psi''(\varphi)x^\gamma &\qquad\text{in } (0,T)\times D\,,\\
  \partial_t z^\gamma - \Delta z^\gamma + c z^\gamma h(\varphi) 
  + c\sigma h'(\varphi)x^\gamma
  +bz^\gamma=\gamma_2 &\qquad\text{in } (0,T)\times D\,,\\
  \partial_{\bf n}x^\gamma = \partial_{\bf n} z^\gamma = 0 &\qquad\text{in } (0,T)\times \partial D\,,\\
  x^\gamma (0)=z^\gamma(0)=0 &\qquad\text{in } D\,.
\end{align*}

Assume now that $(\pi_i,\tilde\pi_i, \rho_i,\xi_i,\theta_i)$ solve \eqref{pi}--\eqref{ad3}
for $i=1,2$. Testing the equation for $x^\gamma$ by $\pi_i$, the equation for $z^\gamma$ by $\rho_i$,
and subtracting the equation for $\pi_i$ tested by $x^\gamma$ and the equation for $\rho_i$
tested by $z^\gamma$, integrating from $0$ to $T$ and taking expectations, 
using the initial conditions for $(x^\gamma, z^\gamma)$ and the final conditions 
for $(\pi_i, \rho_i)$, we get the duality relation
\[
  \E\int_Q\pi_i\gamma_1 + \E\int_Q\rho_i\gamma_2 = \beta_1\E\int_Q(\varphi-\varphi_Q)x^\gamma +
   \E\int_D\left(\beta_2(\varphi(T)-\varphi_T)+\frac{\beta_3}{2}\right)x^\gamma(T)
\]
for every $\gamma\in L^2(\Omega; L^2(0,T; H))^2$ and for $i=1,2$.
Hence, subtracting we infer that 
\[
  \E\int_Q(\pi_1-\pi_2)\gamma_1 + \E\int_Q(\rho_1-\rho_2)\gamma_2 = 0\,,
\]
from which $\pi_1=\pi_2$ and $\rho_1=\rho_2$. The uniqueness of the other solution
components follows then by comparison in the equations.

\subsection{First order conditions for optimality}
We prove here the first version of the necessary conditions
for optimality contained in Proposition~\ref{th:7} and Theorem~\ref{th:9}.

Let $(\bar u, \bar w)\in \mathcal U$ be an optimal control, 
$\varphi:=\mathcal S(\bar u, \bar w)$ the corresponding state
and let $(u,w)\in\mathcal U$ be arbitrary.
Setting $k:=(k_u,k_w):=(u-\bar u, w-\bar w)$, 
by the convexity of $\mathcal U$ we have that 
$(\bar u+\eps k_u, \bar w + \eps k_w)\in\mathcal U$ for every 
$\eps\in[0,1]$: hence, by definition of optimal control we have,
setting $\varphi_\eps:=\mathcal S(\bar u+\eps k_u, \bar w + \eps k_w)$,
\begin{align*}
  J(\bar\varphi, \bar u,\bar w)
  &\leq 
  \frac{\beta_1}{2}\E\int_Q|\varphi_\eps-\varphi_Q|^2
  +\frac{\beta_2}{2}\E\int_D|\varphi_\eps(T)-\varphi_T|^2
  +\frac{\beta_3}{2}\E\int_D(\varphi_\eps(T) + 1)\\
  &+\frac{\beta_4}{2}\E\int_Q|\bar u+\eps k_u|^2
  +\frac{\beta_5}{2}\E\int_Q|\bar w+\eps k_w|^2\,.
\end{align*}
Solving the square-powers on the right-hand side and 
plugging in the definition of $J$ yields
\begin{align*}
  0
  &\leq 
  \frac{\beta_1}{2}\E\int_Q\left(|\varphi_\eps|^2-|\bar\varphi|^2-2(\varphi_\eps-\bar\varphi)\varphi_Q\right)
  +\frac{\beta_2}{2}\E\int_D\left(|\varphi_\eps(T)|^2-|\bar\varphi(T)|^2
  -2(\varphi_\eps-\bar\varphi)(T)\varphi_T\right)\\
  &+\frac{\beta_3}{2}\E\int_D(\varphi_\eps-\bar\varphi)(T)
  +\frac{\beta_4}{2}\E\int_Q\left(\eps^2|k_u|^2+2\eps \bar uk_u\right)
  +\frac{\beta_5}{2}\E\int_Q\left(\eps^2|k_w|^2+2\eps \bar wk_w\right)\,.
\end{align*}
Dividing by $\eps$ and using the G\^ateaux-differentiability of $J$ (all the terms admit G\^ateaux differential) we infer that 
\begin{align*}
  0
  &\leq 
  \beta_1\E\int_Q\left(
  \int_0^1(\bar\varphi+\tau(\varphi_\eps-\bar\varphi))\,d\tau - \varphi_Q
  \right)\frac{\varphi_\eps-\varphi}{\eps}\\
  &+\beta_2\E\int_D\left(
  \int_0^1(\bar\varphi+\tau(\varphi_\eps-\bar\varphi))(T)\,d\tau - \varphi_T
  \right)\frac{\varphi_\eps-\varphi}{\eps}(T)\\
  &+\frac{\beta_3}{2}\E\int_D\frac{\varphi_\eps-\bar\varphi}\eps(T)
  +\beta_4\E\int_Q \bar uk_u + \beta_5\E\int_Q\bar wk_w
  +\eps\left(\frac{\beta_4}2\E\int_Q|k_u|^2 + \frac{\beta_5}{2}\E\int_Q|k_w|^2\right)\,.
\end{align*}
By Theorem~\ref{th:6} we have that 
$\varphi_\eps-\bar\varphi\to 0$ strongly in $L^2(\Omega; C^0([0,T]; H))$,
and $\frac{\varphi_\eps-\bar\varphi}{\eps}\to x_k$ strongly in $L^2(\Omega; L^2(0,T; H))$
and $\frac{\varphi_\eps-\bar\varphi}{\eps}(T)\rightharpoonup x_k(T)$ weakly in $L^2(\Omega; H)$,
so that letting $\eps\searrow 0$ we can conclude that 
\[
  \beta_1\E\int_Q(\bar\varphi-\varphi_Q)x_k
  +\beta_2\E\int_D(\bar\varphi(T)-\varphi_T)x_k(T)
  +\frac{\beta_3}2\E\int_Dx_k(T)
  +\beta_4\E\int_Q\bar uk_u 
  +\beta_5\E\int_Q\bar w k_w
  \geq 0\,,
\]
and Proposition~\ref{th:7} is proved.

Finally, note that choosing $\gamma_1=-\alpha h(\varphi) k_u$ and $\gamma_2=b k_w$
we get $x^\gamma=x_k$, $y^\gamma=y_k$ and $z^\gamma=z_k$ by Theorem~\ref{th:5}, so that 
as we have already pointed out in the previous sections, the following duality formula holds:
\[
  -\alpha\E\int_Q\pi h(\bar\varphi) k_u+ b\E\int_Q\rho k_w = \beta_1\E\int_Q(\bar\varphi-\varphi_Q)x_k +
   \E\int_D\left(\beta_2(\bar\varphi(T)-\varphi_T)+\frac{\beta_3}{2}\right)x_k(T)\,.
\]
By comparison we obtain the desired inequality, and Theorem~\ref{th:9} is finally proved.

{\footnotesize
\bibliography{ref}{}
\bibliographystyle{abbrv}
}
\end{document}